\documentclass[english]{article}
\usepackage[T1]{fontenc}
\usepackage[latin9]{inputenc}
\usepackage{babel}
\usepackage{float}
\usepackage{mathtools}
\usepackage{dsfont}
\usepackage{amsmath}
\usepackage{amsthm}
\usepackage{amssymb}
\usepackage[unicode=true,pdfusetitle,
 bookmarks=true,bookmarksnumbered=false,bookmarksopen=false,
 breaklinks=false,pdfborder={0 0 1},backref=false,colorlinks=false]
 {hyperref}

\makeatletter

\floatstyle{ruled}
\newfloat{algorithm}{tbp}{loa}
\providecommand{\algorithmname}{Algorithm}
\floatname{algorithm}{\protect\algorithmname}

\theoremstyle{plain}
\newtheorem{thm}{\protect\theoremname}
\theoremstyle{remark}
\newtheorem{rem}[thm]{\protect\remarkname}
\theoremstyle{plain}
\newtheorem{fact}[thm]{\protect\factname}
\theoremstyle{plain}
\newtheorem{lem}[thm]{\protect\lemmaname}
\ifx\proof\undefined
\newenvironment{proof}[1][\protect\proofname]{\par
	\normalfont\topsep6\p@\@plus6\p@\relax
	\trivlist
	\itemindent\parindent
	\item[\hskip\labelsep\scshape #1]\ignorespaces
}{%
	\endtrivlist\@endpefalse
}
\providecommand{\proofname}{Proof}
\fi
\theoremstyle{plain}
\newtheorem{cor}[thm]{\protect\corollaryname}

\usepackage[paperwidth=8.5in, paperheight=11in, margin=1in]{geometry}

\author{ 
Zijian Liu\thanks{Stern School of Business, New York University, \texttt{zl3067@stern.nyu.edu}.}  
 \and 
Zhengyuan Zhou\thanks{Stern School of Business, New York University, \texttt{zzhou@stern.nyu.edu}.}}

\date{}

\makeatother

\providecommand{\corollaryname}{Corollary}
\providecommand{\factname}{Fact}
\providecommand{\lemmaname}{Lemma}
\providecommand{\remarkname}{Remark}
\providecommand{\theoremname}{Theorem}

\begin{document}
\global\long\def\E{\mathbb{\mathbb{E}}}%
\global\long\def\F{\mathcal{F}}%
\global\long\def\R{\mathbb{R}}%
\global\long\def\N{\mathbb{\mathbb{N}}}%
\global\long\def\bzero{\mathbb{\mathbf{\mathbf{0}}}}%
\global\long\def\na{\nabla}%
\global\long\def\indi{\mathds{1}}%
\global\long\def\dom{\mathcal{X}}%
\global\long\def\pa{\partial}%
\global\long\def\hp{\widehat{\partial}}%
\global\long\def\argmin{\mathrm{argmin}}%

\title{Stochastic Nonsmooth Convex Optimization with Heavy-Tailed Noises:
High-Probability Bound, In-Expectation Rate and Initial Distance Adaptation}
\maketitle
\begin{abstract}
Recently, several studies consider the stochastic optimization problem
but in a heavy-tailed noise regime, i.e., the difference between the
stochastic gradient and the true gradient is assumed to have a finite
$p$-th moment (say being upper bounded by $\sigma^{p}$ for some
$\sigma\geq0$) where $p\in(1,2]$, which not only generalizes the
traditional finite variance assumption ($p=2$) but also has been
observed in practice for several different tasks. Under this challenging
assumption, lots of new progress has been made for either convex or
nonconvex problems, however, most of which only consider smooth objectives.
In contrast, people have not fully explored and well understood this
problem when functions are nonsmooth. This paper aims to fill this
crucial gap by providing a comprehensive analysis of stochastic nonsmooth
convex optimization with heavy-tailed noises. We revisit a simple
clipping-based algorithm, whereas, which is only proved to converge
in expectation but under the additional strong convexity assumption.
Under appropriate choices of parameters, for both convex and strongly
convex functions, we not only establish the first high-probability
rates but also give refined in-expectation bounds compared with existing
works. Remarkably, all of our results are optimal (or nearly optimal
up to logarithmic factors) with respect to the time horizon $T$ even
when $T$ is unknown in advance. Additionally, we show how to make
the algorithm parameter-free with respect to $\sigma$, in other words,
the algorithm can still guarantee convergence without any prior knowledge
of $\sigma$. Furthermore, an initial distance adaptive convergence
rate is provided if $\sigma$ is assumed to be known.
\end{abstract}

\section{Introduction}

In this paper, we consider the constrained optimization problem
$\min_{\dom}F(x)$ where $F(x)$ is convex and Lipschitz and $\dom\subseteq\R^{d}$
is a closed convex set and possible to be $\R^{d}$. With a stochastic
oracle $\hp F(x)$ satisfying $\E[\hp F(x)\vert x]\in\pa F(x)$ where
$\pa F(x)$ denotes the set of subgradients at $x$, the classic algorithm,
stochastic gradient descent (SGD) \cite{robbins1951stochastic}, guarantees
a convergence rate of $O(T^{-\frac{1}{2}})$ in expectation after
$T$ iterations running under the finite variance condition for the
noise, i.e., $\E[\|\hp F(x)-\E[\hp F(x)\vert x]\|^{2}\vert x]\leq\sigma^{2}$
for some $\sigma\geq0$ representing the noise level. Due to the easy
implementation and empirical success of SGD, it has become one of
the most standard and popular algorithms for optimization problems
nowadays.

However, a huge part of studies, e.g., \cite{mirek2011heavy,simsekli2019tail,zhang2020adaptive,zhou2020towards,gurbuzbalaban2021fractional,camuto2021asymmetric,hodgkinson2021multiplicative},
points out that the finite variance assumption may be too optimistic
and is indeed violated in different machine learning tasks from empirical
observations. Instead, it is more proper to assume the bounded $p$-th
moment noise. Specifically, the noise is considered to satisfy $\E[\|\hp F(x)-\E[\hp F(x)\vert x]\|^{p}\vert x]\leq\sigma^{p}$
for some $p\in(1,2]$ and $\sigma\geq0$, which is known as heavy-tailed.
In particular, if $p=2$, this is exactly the finite variance assumption.
In contrast, the case of $p\in(1,2)$ is much more complicated as
the existing theory for vanilla SGD becomes invalid and the SGD algorithm
itself may fail to converge.

Following this new challenging and more realistic assumption, several
works propose different algorithms to overcome this problem. When
specialized to our case, i.e., a convex and Lipschitz objective, \cite{vural2022mirror}
is the first and the only one to provide an algorithm based on mirror
descent (MD) \cite{nemirovskij1983problem} achieving the in-expectation
rate of $O(T^{\frac{1-p}{p}})$, which matches the lower bound $\Omega(T^{\frac{1-p}{p}})$
\cite{nemirovskij1983problem,raginsky2009information,vural2022mirror}.
In addition, if strong convexity is assumed, \cite{zhang2020adaptive}
gives the first clipping-based algorithm to achieve the rate of $O(T^{\frac{2(1-p)}{p}})$
in expectation, in which a biased estimator $g=\min\left\{ 1,M/\|\hp F(x)\|_{2}\right\} \hp F(x)$
is used to deal with the heavy-tailed issue where $M>0$ denotes the
clipping magnitude. Besides, in the same work, they also establish
the first lower bound of $\Omega(T^{\frac{2(1-p)}{p}})$ to show their
algorithm is optimal. However, even though the two bounds in \cite{zhang2020adaptive,vural2022mirror}
are both optimal with respect to $T$, both of them are not adaptive
to $\sigma$. To be more precise, when $\sigma=0$, they can not recover
the optimal deterministic rates of $O(T^{-\frac{1}{2}})$ and $O(T^{-1})$
for convex and strongly convex cases respectively.

Despite two time-optimal in-expectation bounds have been proved for
our problem, another crucial part, the high-probability convergence
guarantee, still lacks, which turns out to be more helpful in describing
the convergence behavior for an individual running. Notably, if we
instead consider a smooth optimization problem (i.e., $F(x)$ is differentiable
and the gradient of $F(x)$ is Lipschitz), \cite{cutkosky2021high,sadiev2023high,nguyen2023high,liu2023breaking}
make different progress in both convex and nonconvex optimization.
Naturally, one may wonder whether the high-probability bounds can
also be proved in the Lipschitz case. Motivated by this important
gap, we give an affirmative answer to this question in this work by
considering the same clipping algorithm in \cite{zhang2020adaptive}
and show (nearly) optimal high-probability convergence rates for both
convex and strongly convex objectives. Remarkably, under the same
settings of parameters used for the high-probability bounds, we also
prove refined in-expectation rates that are adaptive to $\sigma$.
Hence, we give an exhaustive analysis for stochastic nonsmooth convex
optimization with heavy-tailed noises.

\subsection{Our Contributions}

We use a simple clipping algorithm to handle noises with only bounded
$p$-th moment for $p\in(1,2]$ and establish several new results
for different cases.
\begin{itemize}
\item When the function is assumed to be $G$-Lipschitz and convex
\begin{itemize}
\item We provide the first high-probability bound for the stochastic nonsmooth
convex optimization with heavy-tailed noises. Notably, our choices
of the clipping magnitude $M_{t}$ and step size $\eta_{t}$ are very
flexible. The corresponding rate $\widetilde{O}(T^{\frac{1-p}{p}})$
always matches the lower bound $\Omega(T^{\frac{1-p}{p}})$ only up
to logarithmic factors whenever the time horizon $T$ and the noise
level $\sigma$ is known or not. In other words, we give an any-time
bound that is parameter-free with respect to $\sigma$ at the same
time. Moreover, the bound will be adaptive to $\sigma$ when $\sigma$
is assumed to be known.\\
Besides, the dependence on the failure probability $\delta$ is $O(\log(1/\delta))$
in all the above cases. Our high-probability analysis is done in a
direct style in contrast to the induction-based proof in prior works,
which always leads to a sub-optimal dependence $O(\log(T/\delta))$.
It is also worth emphasizing that our proof neither makes any compact
assumption on the constrained set $\dom$ nor requires the knowledge
of the distance between $x_{1}$ and $x_{*}$ where $x_{1}$ and $x_{*}$
are the initial point and a local minimizer in $\dom$ respectively.
\item Under the same settings of $M_{t}$ and $\eta_{t}$ used in the proof
of high-probability bounds, we also show a nearly optimal in-expectation
convergence rate $\widetilde{O}(T^{\frac{1-p}{p}})$. To our best
knowledge, this is the first in-expectation bound for a clipping algorithm
under this problem. Especially, when $T$ is assumed to be known,
the extra logarithmic factors in $\widetilde{O}(T^{\frac{1-p}{p}})$
can be removed, which leads to the best possible rate of $O(T^{\frac{1-p}{p}})$.
Moreover, if both $\sigma$ and $T$ can be used to set $M_{t}$ and
$\eta_{t}$, the rate at this time will be $O((G+\sigma)T^{-\frac{1}{2}}+\sigma T^{\frac{1-p}{p}})$,
which is adaptive to $\sigma.$
\item When $\sigma$ and $\delta$ are assumed to be known in advance, we
also prove an initial distance adaptive bound under more careful choices
of parameters. More precisely, the dependence on the initial distance
is only in the order of $O(\log\frac{r+\|x_{1}-x_{*}\|}{r}(r+\|x_{1}-x_{*}\|))$
where $r>0$ can be any number rather than the traditional quadratic
bound $O(\alpha+\|x_{1}-x_{*}\|^{2}/\alpha)$ where $\alpha>0$ can
be viewed as the learning rate needed to be tuned. Moreover, the convergence
rate is still optimal in $T$ up to logarithmic factors, adaptive
to $\sigma$ and only has $O(\log(1/\delta))$ dependence on $\delta$.
\end{itemize}
\item When the objective is assumed to be $G$-Lipschitz and strongly convex
\begin{itemize}
\item We give the first high-probability convergence analysis and show an
optimal convergence rate $O(\log^{2}(1/\delta)((G^{2}+\sigma^{2})T^{-1}+(M^{2}+\sigma^{2p}M^{2-2p}+\sigma^{p}G^{2-p})T^{\frac{2(1-p)}{p}}))$
with probability at least $1-\delta$ where $M>0$ can be any real
number used to decide the clipping magnitude. The same as the convex
case, the choices of $M_{t}$ and $\eta_{t}$ don't require any prior
knowledge of the time horizon $T$ and the noise level $\sigma$,
either. Especially, when $\sigma$ is known, setting $M=\sigma$ leads
to the noise adaptive rate $O(\log^{2}(1/\delta)(G^{2}T^{-1}+(\sigma^{2}+\sigma^{p}G^{2-p})T^{\frac{2(1-p)}{p}}))$.
\item With the same $M_{t}$ and $\eta_{t}$ used for high-probability bounds,
our algorithm also guarantees a rate of $O((G^{2}+\sigma^{2})T^{-1}+(\sigma^{p}M^{2-p}+\sigma^{2p}M^{2-2p})T^{\frac{2(1-p)}{p}})$
in expectation where $M$ can be any positive real number. It is worth
mentioning that, unlike the convex case, such an $\sigma$-adaptive
rate doesn't need $\sigma$ to set $M_{t}$ and $\eta_{t}$.
\end{itemize}
\end{itemize}

\subsection{Related Work}

We review the literature related to nonsmooth convex optimization
with heavy-tailed noises. As for the heavy-tailed smooth problems
(either convex or non-convex), the reader can refer to \cite{csimcsekli2019heavy,cutkosky2021high,wang2021convergence,jakovetic2022nonlinear,sadiev2023high,nguyen2023high,liu2023breaking}
for recent progress.

\textbf{High-probability convergence with heavy-tailed noises:} As
far as we know, there doesn't exist any prior work establishing the
high-probability convergence rate when considering nonsmooth convex
(or strongly convex) optimization with heavy-tailed noises. However,
a previous paper \cite{zhang2022parameter} is very close to our problem,
in which the authors focus on online nonsmooth convex optimization
and present an algorithm with a provable high-probability convergence
bound to address heavy-tailed noises. One can employ their algorithm
to solve our problem as online convex optimization is more general.
But there are still lots of differences between our work and \cite{zhang2022parameter}.
A comprehensive comparison can be found in Section \ref{sec: algo}.

If only considering the finite variance case (i.e., $p=2$) with convex
objectives, \cite{parletta2022high} provides a high-probability bound
$O(\log(1/\delta)T^{-\frac{1}{2}})$ where $\delta$ is the failure
probability. But they require a bounded domain $\dom$ in the proof,
which is a restrictive assumption and significantly simplifies the
analysis. \cite{gorbunov2021near} is the first to show a high-probability
rate when the domain of the problem is $\R^{d}$. Whereas to set up
the parameters in their algorithm, the initial distance $\|x_{1}-x_{*}\|_{2}$
(or any upper bound on it) needs to be known, which is also hard to
estimate when $\dom$ is unbounded. Besides, the dependence on $\delta$
in \cite{gorbunov2021near} is sub-optimal $\log(T/\delta)$.

\textbf{In-expectation convergence with heavy-tailed noises: }For
the Lipschitz convex problem, \cite{vural2022mirror} is the first
and the only work to show an $O(T^{\frac{1-p}{p}})$ rate under the
framework of MD. However, unlike the popular clipping method, their
algorithm is based on the uniform convexity property (see Definition
1 in \cite{vural2022mirror}) of the mirror map. Therefore, an in-expectation
bound for the clipping-based algorithm still lacks in this case. When
the functions are additionally considered to be strongly convex, \cite{zhang2020adaptive}
is the first and the only work to prove an $O(T^{\frac{2(1-p)}{p}})$
convergence rate in expectation by combining SGD and clipped gradients.

However, we would like to mention that both \cite{vural2022mirror}
and \cite{zhang2020adaptive} only assume that $\E[\|\widehat{\pa}F(x)\|^{p}\vert x]$
is uniformly bounded by a constant $C^{p}$ for some $C>0$ unlike
our assumption of $\E[\|\hp F(x)-\E[\widehat{\pa}F(x)\vert x]\|^{p}\vert x]\leq\sigma^{p}$.
This difference can let us obtain a refined in-expectation bound,
which is able to be adaptive to $\sigma$. In other words, when $\sigma=0$,
our in-expectation bounds will automatically recover the well-known
optimal rates of $O(T^{-\frac{1}{2}})$ and $O(T^{-1})$ for convex
and strongly convex cases respectively. We refer the reader to Section
\ref{sec: algo} for more detailed comparisons with these two previous
works.

\textbf{Lower bound with heavy-tailed noises: }When the noises only
have $p$-th finite moment for some $p\in(1,2]$, for the convex functions,
\cite{nemirovskij1983problem,raginsky2009information,vural2022mirror}
show that the convergence rate of any first-order algorithm cannot
be faster than $\Omega(T^{\frac{1-p}{p}})$. If the strong convexity
is additionally assumed, \cite{zhang2020adaptive} is the first to
establish a lower bound of $\Omega(T^{\frac{2(1-p)}{p}})$.

\section{Preliminaries\label{sec:Preliminaries}}

\textbf{Notations}: Let $\left[d\right]$ denote the set $\left\{ 1,2,\cdots,d\right\} $
for any integer $d\geq1$. $\langle\cdot,\cdot\rangle$ is the standard
Euclidean inner product on $\R^{d}$ and $\|\cdot\|$ represents the
$\ell_{2}$ norm. $\mathrm{int}(A)$ stands for the interior points
of any set $A\subseteq\R^{d}$. Given a closed and convex set $C$,
$\Pi_{C}(\cdot)$ is the projection operator onto $C$, i.e., $\Pi_{C}(x)=\argmin_{z\in C}\|z-x\|$.
$a\lor b$ and $a\land b$ are defined as $\max\left\{ a,b\right\} $
and $\min\left\{ a,b\right\} $ respectively. Given a function $f$,
$\partial f(x)$ denotes the set of subgradients at $x$.

We focus on the following optimization problem in this work:

\textbf{
\[
\min_{x\in\dom}F(x)
\]
}where $F$ is convex and $\dom\subseteq\mathrm{int}(\mathrm{dom}(F))\subseteq\R^{d}$
is a closed convex set. The requirement of $\dom\subseteq\mathrm{int}(\mathrm{dom}(F))$
is only to guarantee the existence of subgradients for every point
in $\dom$ with no other special reason. We remark that there is no
compactness assumption on $\dom$. Additionally, our analysis relies
on the following assumptions

\textbf{1. Existence of a local minimizer}: $\exists x_{*}\in\arg\min_{x\in\dom}F(x)$
satisfying $F(x_{*})>-\infty$.

\textbf{2. }$\mu$\textbf{-strongly convex}: $\exists\mu\geq0$ such
that $F(x)\geq F(y)+\langle g,x-y\rangle+\frac{\mu}{2}\|x-y\|^{2},\forall x,y\in\dom,g\in\pa F(y)$.

\textbf{3. }$G$\textbf{-Lipschitz}: $\exists G>0$ such that $\|g\|\leq G,\forall x\in\dom,g\in\pa F(x)$.

\textbf{4. Unbiased gradient estimator}: We can access a history-independent,
unbiased gradient estimator $\hp F(x)$ for any $x\in\dom$, i.e.,
$\E[\widehat{\pa}F(x)\vert x]\in\pa F(x),\forall x\in\dom$.

\textbf{5. Bounded $p$-th moment noise}: There exist $p\in(1,2]$
and $\sigma\geq0$ denoting the noise level such that $\E[\|\hp F(x)-\E[\widehat{\pa}F(x)\vert x]\|^{p}\vert x]\leq\sigma^{p}$.

We briefly discuss the assumptions here. Assumptions 1-3 are standard
in the nonsmooth convex optimization literature. For Assumption 2,
the objective will degenerate to the convex function when $\mu=0$.
Assumption 4 is commonly used in stochastic optimization. Assumption
5 is the definition of heavy-tailed noise. Lastly, we would like to
mention that the reason for using the $\ell_{2}$ norm is only for
convenience. When considering a general norm, similar results to our
theorems (except Theorem \ref{thm:lip-dog-prob}) still hold after
changing the algorithmic framework into MD. A more detailed discussion
will be given in Section \ref{sec:general-norm} in the appendix.

\section{Algorithm and its Convergence Guarantee\label{sec: algo}}

\begin{algorithm}[h]
\caption{\label{alg:algo}Projected SGD with Clipping}

\textbf{Input}: $x_{1}\in\dom$, $M_{t}>0$, $\eta_{t}>0$.

\textbf{for} $t=1$ \textbf{to} $T$ \textbf{do}

$\quad$$g_{t}=\left(1\land\frac{M_{t}}{\left\Vert \hp F(x_{t})\right\Vert }\right)\hp F(x_{t})$

$\quad$$x_{t+1}=\Pi_{\dom}(x_{t}-\eta_{t}g_{t}).$

\textbf{end for}
\end{algorithm}

The projected clipped SGD algorithm is shown in Algorithm \ref{alg:algo}.
The algorithm itself is simple to understand. Compared with SGD, the
only difference is to clip the stochastic gradient $\hp F(x_{t})$
with a threshold $M_{t}$. In the next two sections, we will show
that properly picked $M_{t}$ and $\eta_{t}$ guarantee both high-probability
and in-expectation convergence for Algorithm \ref{alg:algo}. Again,
we remark that our results in Sections \ref{subsec:convex} and \ref{subsec:str-convex}
can be extened to any norm $\|\cdot\|$ on $\R^{d}$. Theorem \ref{thm:lip-dog-prob}
in \ref{subsec:dog} still holds when changing the $\ell_{2}$ norm
to the Mahalanobis norm, i.e., $\|x\|=\sqrt{x^{\top}Ax}$ for $A\succ0$.

\subsection{General Convergence Theorems When $\mu=0$\label{subsec:convex}}

In this section, we present the convergence theorems of Algorithm
\ref{alg:algo} for convex functions, i.e., $\mu=0$.

First, when $T$ is not assumed to be known, Theorem \ref{thm:lip-prob}
gives any-time high-probability convergence bounds for two cases,
i.e., whether the noise level $\sigma$ is known or not. As far as
we know, Theorem \ref{thm:lip-prob} is the first to describe an any-time
high-probability convergence rate for nonsmooth convex optimization
problems when the noise is assumed to be heavy-tailed.
\begin{thm}
\label{thm:lip-prob}Suppose Assumptions (1)-(5) hold with $\mu=0$
and let $\bar{x}_{T}=\frac{1}{T}\sum_{t=1}^{T}x_{t}$. Under the choices
of $M_{t}=2G\lor Mt^{\frac{1}{p}}$ where $M\geq0$ can be any real
number and $\eta_{t}=\frac{\alpha}{G\sqrt{t}}\land\frac{\alpha}{M_{t}}$
where $\alpha=\frac{\beta}{\log(4/\delta)}$ and $\beta>0$ can be
any real number, for any $T\geq1$ and $\delta\in(0,1)$, the following
bound holds with probability at least $1-\delta$,
\[
F(\bar{x}_{T})-F(x_{*})\leq O\left(\left(\beta\left(1+(\sigma/M)^{2p}\right)\log^{2}T+\log(1/\delta)\left(\beta+\frac{\left\Vert x_{1}-x_{*}\right\Vert ^{2}}{\beta}\right)\right)\left(\frac{G}{\sqrt{T}}\lor\frac{M}{T^{\frac{p-1}{p}}}\right)\right).
\]
Especially, by setting $M=\sigma$ when $\sigma$ is known, we have
\[
F(\bar{x}_{T})-F(x_{*})\leq O\left(\left(\beta\log^{2}T+\log(1/\delta)\left(\beta+\frac{\left\Vert x_{1}-x_{*}\right\Vert ^{2}}{\beta}\right)\right)\left(\frac{G}{\sqrt{T}}\lor\frac{\sigma}{T^{\frac{p-1}{p}}}\right)\right).
\]
\end{thm}
\begin{rem}
The choice of $\alpha=\frac{\beta}{\log(4/\delta)}$ is only for optimizing
the dependence on $\log(1/\delta)$. Our theoretical analysis works
for any $\alpha>0.$ Additionally, it is possible to choose $\eta_{t}=\frac{\alpha_{1}}{G\sqrt{t}}\land\frac{\alpha_{2}}{M_{t}}$
for different $\alpha_{1},\alpha_{2}>0$. However, we keep the same
$\alpha$ for simplicity in Theorem \ref{thm:lip-prob} and the following
Theorems \ref{thm:lip-prob-fix}, \ref{thm:lip-exp} and \ref{thm:lip-exp-fix}.
\end{rem}
We note that whenever $\sigma$ is known or not, our choice always
leads to the (nearly) optimal rate $\widetilde{O}(T^{\frac{1-p}{p}})$
in $T$. Moreover, if we assume $\sigma$ is known and consider the
choice of $M_{t}=2G\lor\sigma t^{\frac{1}{p}}$ when $\sigma=0$,
in other words, the deterministic case, the clipping magnitude $M_{t}$
will be $M_{t}=2G$ and the step size $\eta_{t}$ is $\eta_{t}=\frac{\alpha}{G\sqrt{t}}\land\frac{\alpha}{2G}=O(\frac{\alpha}{G\sqrt{t}})$.
Recall that the norm of any subgradient is bounded by $G$, which
implies $M_{t}=2G$ won't have any effect now. Hence, the algorithm
will be the totally same as the traditional Projected SGD. The corresponding
bound will also be the (nearly) optimal rate $\widetilde{O}(G/\sqrt{T})$.
We would like to emphasize that the appearance of the term $\log^{2}T$
is due to the time-varying step size rather than the analysis technique
for the high-probability bound. Notably, the dependence on $\delta$
is only $\log(1/\delta)$ rather than the sub-optimal $\log(T/\delta)$
in previous works.

Next, in Theorem \ref{thm:lip-prob-fix}, we state the fixed time
bound, i.e., the case of known $T$. As mentioned above, the extra
term $\log^{2}T$ will be removed.
\begin{thm}
\label{thm:lip-prob-fix}Suppose Assumptions (1)-(5) hold with $\mu=0$
and let $\bar{x}_{T}=\frac{1}{T}\sum_{t=1}^{T}x_{t}$. Additionally,
assume $T$ is known. Under the choices of $M_{t}=2G\lor MT^{\frac{1}{p}}$
where $M\geq0$ can be any real number and $\eta_{t}=\frac{\alpha}{G\sqrt{T}}\land\frac{\alpha}{M_{t}}$
where $\alpha=\frac{\beta}{\log(4/\delta)}$ and $\beta>0$ can be
any real number, for any $T\geq1$ and $\delta\in(0,1)$, the following
bound holds with probability at least $1-\delta$,
\[
F(\bar{x}_{T})-F(x_{*})\leq O\left(\left(\beta(\sigma/M)^{2p}+\log(1/\delta)\left(\beta+\frac{\left\Vert x_{1}-x_{*}\right\Vert ^{2}}{\beta}\right)\right)\left(\frac{G}{\sqrt{T}}\lor\frac{M}{T^{\frac{p-1}{p}}}\right)\right).
\]
Especially, by setting $M=\sigma$ when $\sigma$ is known, we have
\[
F(\bar{x}_{T})-F(x_{*})\leq O\left(\log(1/\delta)\left(\beta+\frac{\left\Vert x_{1}-x_{*}\right\Vert ^{2}}{\beta}\right)\left(\frac{G}{\sqrt{T}}\lor\frac{\sigma}{T^{\frac{p-1}{p}}}\right)\right).
\]
\end{thm}
To finish the high-probability bounds, we would like to make a comprehensive
comparison with \cite{zhang2022parameter}, which is the only existing
work showing a high-probability bound of $\widetilde{O}(\epsilon\log(1/\delta)T^{-1}+(\sigma+G)\log(T/\delta)\log(\|x_{1}-x_{*}\|T/\epsilon)\|x_{1}-x_{*}\|T^{\frac{1-p}{p}})$
(where $\epsilon>0$ is any user-specified parameter and $\|\cdot\|$
is the Mahalanobis norm\footnote{The bound in \cite{zhang2022parameter} is proved in the Hilbert space.
Hence, $\|\cdot\|$ will be the Mahalanobis norm when specialized
to $\R^{d}$.}) in the related literature. We need to emphasize our work is different
in several aspects.
\begin{enumerate}
\item The algorithm in \cite{zhang2022parameter} is much more complicated
than ours. To be more precise, their main algorithm needs to call
several outer algorithms. The outer algorithms themselves are even
very involved. This difference is because we only focus on convex
optimization, in contrast, their algorithm is designed for online
convex optimization, which is known to be more general. Hence, when
only considering solving the heavy-tailed nonsmooth convex optimization,
we believe our algorithm is much easier to be implemented.
\item When choosing $M_{t}$ and $\eta_{t}$, \cite{zhang2022parameter}
requires not only the time horizon $T$ but also the noise $\sigma$,
which means their result is neither an any-time bound nor parameter-free
with respect to $\sigma$. In comparison, our Theorem \ref{thm:lip-prob}
doesn't require $T$. We also show how to set $M_{t}$ and $\eta_{t}$
when $\sigma$ is unknown. Besides, the dependence on $\delta$ in
both Theorems \ref{thm:lip-prob} and \ref{thm:lip-prob-fix} is always
$O(\log(1/\delta))$, which is significantly better than $O(\log(T/\delta))$
in \cite{zhang2022parameter}.
\item However, our result is not as good as \cite{zhang2022parameter} for
the dependence on the initial distance $\|x_{1}-x_{*}\|$. As one
can see, our obtained bound is always in the form of $O(\beta+\|x_{1}-x_{*}\|^{2}/\beta)$,
which is worse than $\|x_{1}-x_{*}\|\log(\|x_{1}-x_{*}\|)$ in \cite{zhang2022parameter}.
To deal with this issue, a high-probability bound, $\widetilde{O}(\log(\|x_{1}-x_{*}\|/r)(r+\|x_{1}-x_{*})(G\log(1/\delta)T^{-\frac{1}{2}}+\sigma\log(1/\delta)^{1-\frac{1}{p}}T^{\frac{1-p}{p}}))$
where $r>0$ can be any number, is provided in Theorem \ref{thm:lip-dog-prob}
Section \ref{subsec:dog}. As a tradeoff, compared with Theorems \ref{thm:lip-prob}
and \ref{thm:lip-prob-fix}, Theorem \ref{thm:lip-dog-prob} needs
to assume a known $\sigma$ and can at most be applied to the same
Mahalanobis norm used in \cite{zhang2022parameter} 
\item Finally, it is worth pointing out that our proof techniques are very
different from \cite{zhang2022parameter}. Our analysis is done in
a direct way compared to the reduction-based manner in \cite{zhang2022parameter}.
\end{enumerate}
Now, we turn to provide the first (nearly) optimal in-expectation
convergence rate of clipping algorithms in Theorems \ref{thm:lip-exp}
(any-time bound) and \ref{thm:lip-exp-fix} (fixed time bound), which
correspond to the cases of unknown $T$ and known $T$ respectively.
\begin{thm}
\label{thm:lip-exp}Suppose Assumptions (1)-(5) hold with $\mu=0$
and let $\bar{x}_{T}=\frac{1}{T}\sum_{t=1}^{T}x_{t}$. Under the choices
of $M_{t}=2G\lor Mt^{\frac{1}{p}}$ where $M\geq0$ can be any real
number and $\eta_{t}=\frac{\alpha}{G\sqrt{t}}\land\frac{\alpha}{M_{t}}$
where $\alpha>0$ can be any real number, for any $T\geq1$, we have
\[
\E\left[F(\bar{x}_{T})-F(x_{*})\right]\leq O\left(\left(\alpha\left(\log T+(\sigma/M)^{2p}\log^{2}T\right)+\frac{\left\Vert x_{1}-x_{*}\right\Vert ^{2}}{\alpha}\right)\left(\frac{G}{\sqrt{T}}\lor\frac{M}{T^{\frac{p-1}{p}}}\right)\right).
\]
Especially, by setting $M=\sigma$ when $\sigma$ is known, we have
\[
\E\left[F(\bar{x}_{T})-F(x_{*})\right]\leq O\left(\left(\alpha\log^{2}T+\frac{\left\Vert x_{1}-x_{*}\right\Vert ^{2}}{\alpha}\right)\left(\frac{G}{\sqrt{T}}\lor\frac{\sigma}{T^{\frac{p-1}{p}}}\right)\right).
\]
\end{thm}
\begin{thm}
\label{thm:lip-exp-fix}Suppose Assumptions (1)-(5) hold with $\mu=0$
and let $\bar{x}_{T}=\frac{1}{T}\sum_{t=1}^{T}x_{t}$. Additionally,
assume $T$ is known. Under the choices of $M_{t}=2G\lor MT^{\frac{1}{p}}$
where $M\geq0$ can be any real number and $\eta_{t}=\frac{\alpha}{G\sqrt{T}}\land\frac{\alpha}{M_{t}}$
where $\alpha>0$ can be any real number, for any $T\geq1$, we have
\[
\E\left[F(\bar{x}_{T})-F(x_{*})\right]\leq O\left(\left(\alpha\left(1+(\sigma/M)^{2p}\right)+\frac{\left\Vert x_{1}-x_{*}\right\Vert ^{2}}{\alpha}\right)\left(\frac{G}{\sqrt{T}}\lor\frac{M}{T^{\frac{p-1}{p}}}\right)\right).
\]
Especially, by setting $M=\sigma$ when $\sigma$ is known, we have
\[
\E\left[F(\bar{x}_{T})-F(x_{*})\right]\leq O\left(\left(\alpha+\frac{\left\Vert x_{1}-x_{*}\right\Vert ^{2}}{\alpha}\right)\left(\frac{G}{\sqrt{T}}\lor\frac{\sigma}{T^{\frac{p-1}{p}}}\right)\right).
\]
\end{thm}
We first remark that the choices of $M_{t}$ and $\eta_{t}$ in Theorems
\ref{thm:lip-exp} and \ref{thm:lip-exp-fix} are the same as them
in Theorems \ref{thm:lip-prob} and \ref{thm:lip-prob-fix}. Hence,
our $M_{t}$ and $\eta_{t}$ guarantee both high-probability and in-expectation
convergence. Next, compared with the any-time bounds in Theorem \ref{thm:lip-exp},
the extra logarithmic factors are removed in the fixed time bounds
in Theorem \ref{thm:lip-exp-fix}. Additionally, the rates for the
case of known $\sigma$ are always adaptive to the noise. In particular,
when $T$ is known and $p=2$, our result matches the traditional
bound of SGD perfectly.

Lastly, let us talk about the differences with the prior work \cite{vural2022mirror}
providing the only in-expectation bound but for a different algorithm.
\begin{enumerate}
\item The algorithm in \cite{vural2022mirror} is based on MD, but more
importantly, requires the property of uniform convexity (see Definition
1 in \cite{vural2022mirror}) for the mirror map. However, our in-expectation
bounds are for the algorithm employing the clipping method, which
is widely used to deal with heavy-tailed problems in several areas
but lacks theoretical justifications in nonsmooth convex optimization.
\item \cite{vural2022mirror} only assumes $\E[\|\widehat{\pa}F(x)\|^{p}\vert x]\leq\sigma^{p}$
for some $p\in(1,2]$ and $\sigma>0$ (strictly speaking, the norm
in \cite{vural2022mirror} is $\ell_{q}$ norm for some $q\in[1,\infty]$,
however, our method can be extended to an arbitrary norm including
$\ell_{q}$ norm as a subcase). This assumption is equivalent to
$\E[\|\widehat{\pa}F(x)\|^{p}\vert x]\leq O((\sigma+G)^{p})$ under
our assumptions, which means \cite{vural2022mirror} needs both $G$
and $\sigma$ as input but their final rate doesn't adapt to $\sigma$.
In contrast, we not only give a rate being adaptive to the noise when
$\sigma$ is known but also show how to run our algorithm without
any prior knowledge of $\sigma$.
\item Besides, our parameter settings not only guarantee in-expectation
convergence but also admit provable high-probability bounds as shown
in Theorems \ref{thm:lip-prob} and \ref{thm:lip-prob-fix}. But \cite{vural2022mirror}
only provides the in-expectation result for their algorithm.
\item Finally, our proof strategy is completely different from \cite{vural2022mirror}
as there is no clipping step in which. We believe that our techniques
in the proof will lead to a better understanding of the clipping method.
\end{enumerate}

\subsection{Initial Distance Adaptive Convergence Rate When $\mu=0$\label{subsec:dog}}

As mentioned above, in this section, we show that Algorithm \ref{alg:algo}
can achieve an initial distance adaptive convergence under sophisticated
parameters. In the traditional bound for SGD, a quadratic dependence
on the initial distance $O(\alpha+\|x_{1}-x_{*}\|^{2}/\alpha)$ always
shows up where $\alpha>0$ is the learning rate. Such a term also
appears in our above results, e.g., Theorem \ref{thm:lip-prob}. Hence,
theoretically speaking, the optimal learning rate $\alpha^{*}=\Theta(\|x_{1}-x_{*}\|)$.
If the domain $\dom$ is bounded with diameter $D$, one can set $\alpha=\Theta(D)$
as a proxy of $\alpha^{*}$. However, in the general unbounded case,
e.g., $\dom=\R^{d}$, the strategy of $\alpha=\Theta(D)$ is no longer
useful as $D=\infty$ now.

One may think it is impossible to achieve a better dependence on $\|x_{1}-x_{*}\|$
for SGD based algorithm if no prior information on $x_{*}$ is known,
whereas \cite{mcmahan2012no} is the first to improve it to the order
of $O(\|x_{1}-x_{*}\|\log(\|x_{1}-x_{*}\|))$ on $\R$. More surprisingly,
the problem considered in \cite{mcmahan2012no} is online learning,
which can cover the optimization problem considered in this paper.
Later on, different algorithms (see, e.g., \cite{mcmahan2014unconstrained,orabona2016coin,zhang2022pde})
are proposed to achieve such an initial distance adaptive bound on
$\R^{d}$. \cite{zhang2022parameter} is the first to extend such
kind of algorithms to deal with online learning problems with heavy-tailed
noises. However, all of these methods are designed for online learning
problems originally causing the algorithms to be complicated when
using them to deal with convex optimization problems. 

Recently, three different works \cite{carmon2022making,defazio2023learning,ivgi2023dog}
come up with different methods to achieve the initial distance adaptive
bound for convex optimization. In our paper, we borrow the key idea
provided in \cite{ivgi2023dog}, i.e., using the term $r_{t}=(\max_{s\in\left[t\right]}\|x_{1}-x_{s}\|)\lor r$
in the step size where $r>0$ can be any real number to approximate
the optimal choice $\Theta(\|x_{1}-x_{*}||)$, to achieve the better
dependence on $\|x_{1}-x_{*}\|$ as shown in the following theorem.
\begin{thm}
\label{thm:lip-dog-prob}Suppose Assumptions (1)-(5) hold with $\mu=0$.
Given $\delta\in(0,1)$, under the choices of 
\begin{itemize}
\item $w_{t\geq1}>0$ is a non-decreasing sequence satisfying $\sum_{t=1}^{T}\frac{1}{tw_{t}}\leq W<\infty$
for any $T\geq1$ and some $W\in\R$;
\item $M_{t}=2G\lor\sigma(tw_{t}/\log(4/\delta)){}^{\frac{1}{p}}$;
\item $\eta_{t}=r_{t}\gamma_{t}$ and $\gamma_{t}=\frac{\alpha_{1}}{G\sqrt{tw_{t}}}\land\frac{\alpha_{2}}{M_{t}}$
where $r_{t}=(\max_{s\in\left[t\right]}\|x_{1}-x_{s}\|)\lor r$ and
$r>0$ can be set arbitrarily;
\item $\alpha_{1}=\frac{1}{\sqrt{32W}},\alpha_{2}=\frac{1}{8\left(\frac{16}{3}+8\sqrt{5W}+4W\right)\log\frac{4}{\delta}}\land\frac{1}{\sqrt{16\left(\frac{32}{3}+8\sqrt{5W}+80W\right)\log\frac{4}{\delta}}}$;
\end{itemize}
then with probability at least $1-\delta$, for any sufficiently large
$T\geq\Omega(\log\frac{r+\left\Vert x_{1}-x_{*}\right\Vert }{r})$,
there is
\begin{align}
F(\bar{x}_{I(T)})-F(x_{*})\leq & O\left(\left(1+\log\frac{r+\left\Vert x_{1}-x_{*}\right\Vert }{r}\right)\left(r+\left\Vert x_{1}-x_{*}\right\Vert \right)\right.\nonumber \\
 & \quad\left.\times\left(\frac{G\sqrt{Ww_{T}}}{\sqrt{T}}\lor\frac{G\left(1+W\right)\log\frac{1}{\delta}}{T}\lor\frac{\sigma\left(1+W\right)\left(w_{T}\right)^{\frac{1}{p}}\left(\log\frac{1}{\delta}\right)^{1-\frac{1}{p}}}{T^{\frac{p-1}{p}}}\right)\right).\label{eq:dog-bound}
\end{align}
where
\[
\bar{x}_{T}=\frac{\sum_{t=1}^{T}r_{t}x_{x}}{\sum_{t=1}^{T}r_{t}},I(T)\in\mathrm{argmax}_{t\in\left[T\right]}\sum_{s=1}^{t}\frac{r_{s}}{r_{t+1}}.
\]

Under the first example, $w_{t}=1+\log^{2}(t)$ and $W=1+\frac{\pi}{2}$,
given in Fact \ref{fact:dog-order}, there is
\[
F(\bar{x}_{I(T)})-F(x_{*})\leq O\left(\left(1+\log\frac{r+\left\Vert x_{1}-x_{*}\right\Vert }{r}\right)\left(r+\left\Vert x_{1}-x_{*}\right\Vert \right)\left(\frac{G\log T}{\sqrt{T}}\lor\frac{G\log\frac{1}{\delta}}{T}\lor\frac{\sigma\left(\log T\right)^{\frac{2}{p}}\left(\log\frac{1}{\delta}\right)^{1-\frac{1}{p}}}{T^{\frac{p-1}{p}}}\right)\right).
\]
\end{thm}
We provide two examples of $w_{t}$ before explaining more about the
theorem.
\begin{fact}
\label{fact:dog-order}The following two choices satisfy the requirements
on $w_{t}$ in Theorem \ref{thm:lip-dog-prob}:
\begin{itemize}
\item $w_{t}=1+\log^{2}(t)$ and $W=1+\frac{\pi}{2}$
\item $w_{t}=\left[y^{(n+1)}(t)\right]^{1+\varepsilon}\prod_{i=1}^{n}y^{(i)}(t)$
and $W=1+\frac{1}{\varepsilon}$ where $y(t)=1+\log(t)$, $y^{(n)}(t)=y(y^{(n-1)}(t))$
is the $n$-times composition with itself for any non-negative integer
$n$, and $\varepsilon>0$ can be chosen arbitrarily.
\end{itemize}
\end{fact}
There are several points we would like to discuss here. First, the
requirement of $T\geq\Omega(\log\frac{r+\left\Vert x_{1}-x_{*}\right\Vert }{r})$
is not necessary, we indeed prove that $O((\frac{r_{1}+\|x_{1}-x_{*}\|}{r})^{\frac{1}{T}})\times$R.H.S.
of (\ref{eq:dog-bound}) holds for any $T\geq1$. For simplicity,
$T$ is assumed to be large enough to make $O((\frac{r_{1}+\|x_{1}-x_{*}\|}{r})^{\frac{1}{T}})=O(1)$.
Next, we would like to emphasize that Theorem \ref{thm:lip-dog-prob}
still holds under the Mahalanobis norm (the same as \cite{zhang2022parameter}).
However, for the general norm combined with the framework of MD, how
to achieve this initial distance adaptive extension still remains
unclear to us. Besides, our bound is an any time bound (without knowing
$T$) and achieves $O(\log(1/\delta))$ (rather than $O(\log(T/\delta))$)
dependence simultaneously, which are both better than \cite{zhang2022parameter}.
However, compared with Theorems \ref{thm:lip-prob} and \ref{thm:lip-prob-fix},
Theorem \ref{thm:lip-dog-prob} requires knowing the $p$-th moment
$\sigma$ and the failure probability $\delta$ in advance as a tradeoff.
Additionally, compared with \cite{ivgi2023dog}, our proof is very
different since we consider the heavy-tailed noises.

\subsection{General Convergence Theorems When $\mu>0$\label{subsec:str-convex}}

In this section, we focus on establishing the convergence rate of
Algorithm \ref{alg:algo} for strongly convex objectives, i.e., $\mu>0$.
In this case, even when $T$ is assumed to be known, we no longer
consider using $T$ to set $M_{t}$ and $\eta_{t}$ since a step size
$\eta_{t}$ depending on $T$ is rarely used under the strong convexity
assumption.

The first result, Theorem \ref{thm:str-prob}, describes the high-probability
behavior of Algorithm \ref{alg:algo}. To our best knowledge, this
is the first high-probability bound for nonsmooth strongly convex
optimization with heavy-tailed noises matching the in-expectation
lower bound of $\Omega(T^{\frac{2(1-p)}{p}})$.
\begin{thm}
\label{thm:str-prob}Suppose Assumptions (1)-(5) hold with $\mu>0$
let $\bar{x}_{T}=\frac{2}{T(T+1)}\sum_{t=1}^{T}tx_{t}$. Under the
choices of $M_{t}=2G\lor Mt^{\frac{1}{p}}$ and $\eta_{t}=\frac{4}{\mu(t+1)}$
where $M\geq0$ can be any real number, for any $T\geq1$ and $\delta\in(0,1)$,
the following two bounds hold simultaneously with probability at least
$1-\delta$,
\begin{align*}
F(\bar{x}_{T})-F(x_{*}) & \leq O\left(\log^{2}(1/\delta)\left(\frac{G^{2}+\sigma^{2}}{\mu T}+\frac{M^{2}+\sigma^{2p}M^{2-2p}+\sigma^{p}G^{2-p}}{\mu T^{\frac{2(p-1)}{p}}}\right)\right);\\
\left\Vert x_{T+1}-x_{*}\right\Vert ^{2} & \leq O\left(\log^{2}(1/\delta)\left(\frac{G^{2}+\sigma^{2}}{\mu^{2}T}+\frac{M^{2}+\sigma^{2p}M^{2-2p}+\sigma^{p}G^{2-p}}{\mu^{2}T^{\frac{2(p-1)}{p}}}\right)\right).
\end{align*}
Especially, by setting $M=\sigma$ when $\sigma$ is known, we have
\begin{align*}
F(\bar{x}_{T})-F(x_{*}) & \leq O\left(\log^{2}(1/\delta)\left(\frac{G^{2}}{\mu T}+\frac{\sigma^{2}+\sigma^{p}G^{2-p}}{\mu T^{\frac{2(p-1)}{p}}}\right)\right)\\
\left\Vert x_{T+1}-x_{*}\right\Vert ^{2} & \leq O\left(\log^{2}(1/\delta)\left(\frac{G^{2}}{\mu T}+\frac{\sigma^{2}+\sigma^{p}G^{2-p}}{\mu T^{\frac{2(p-1)}{p}}}\right)\right).
\end{align*}
\end{thm}
Our choices of $M_{t}$ and $\eta_{t}$ are inspired by \cite{zhang2020adaptive}
but $M_{t}$ is very different at the same time. We need to emphasize
that the parameter $G$ in $M_{t}=Gt^{\frac{1}{p}}$ used in \cite{zhang2020adaptive}
is not the same as our definition of $G$ since \cite{zhang2020adaptive}
only assumes $\E[\|\hp F(x)\|^{p}\vert x]\leq G^{p}$, $\forall x\in\dom$
rather than separates the assumption on noises independently. Under
our assumptions, there is only $\E[\|\hp F(x)\|^{p}\vert x]\leq O((\sigma+G)^{p})$,
which implies $M_{t}=Gt^{\frac{1}{p}}$ in \cite{zhang2020adaptive}
is equivalent to $M_{t}=\Theta((\sigma+G)t^{\frac{1}{p}})$. But this
diverges from our choice of $M_{t}=Mt^{\frac{1}{p}}$. As one can
see, $M_{t}$ in our settings doesn't rely on $\sigma$, this property
makes our choices more practical.

Finally, though an in-expectation bound of clipping algorithms for
the strongly convex case has been established in \cite{zhang2020adaptive},
we provide a refined rate in Theorem \ref{thm:str-exp}.
\begin{thm}
\label{thm:str-exp}Suppose Assumptions (1)-(5) hold with $\mu>0$
let $\bar{x}_{T}=\frac{2}{T(T+1)}\sum_{t=1}^{T}tx_{t}$. Under the
choices of $M_{t}=2G\lor Mt^{\frac{1}{p}}$ and $\eta_{t}=\frac{4}{\mu(t+1)}$
where $M\geq0$ can be any real number, for any $T\geq1$, we have
\begin{align*}
\E\left[F(\bar{x}_{T})-F(x_{*})\right] & \leq O\left(\frac{G^{2}+\sigma^{2}}{\mu T}+\frac{\sigma^{p}M^{2-p}+\sigma^{2p}M^{2-2p}}{\mu T^{\frac{2(p-1)}{p}}}\right);\\
\E\left[\left\Vert x_{T+1}-x_{*}\right\Vert ^{2}\right] & \leq O\left(\frac{G^{2}+\sigma^{2}}{\mu^{2}T}+\frac{\sigma^{p}M^{2-p}+\sigma^{2p}M^{2-2p}}{\mu^{2}T^{\frac{2(p-1)}{p}}}\right).
\end{align*}
Especially, by setting $M=\sigma$ when $\sigma$ is known, we have
\begin{align*}
\E\left[F(\bar{x}_{T})-F(x_{*})\right] & \leq O\left(\frac{G^{2}}{\mu T}+\frac{\sigma^{2}}{\mu T^{\frac{2(p-1)}{p}}}\right)\\
\E\left[\left\Vert x_{T+1}-x_{*}\right\Vert ^{2}\right] & \leq O\left(\frac{G^{2}}{\mu^{2}T}+\frac{\sigma^{2}}{\mu^{2}T^{\frac{2(p-1)}{p}}}\right).
\end{align*}
\end{thm}
We briefly discuss Theorem \ref{thm:str-exp} here before finishing
this section. First, we remark that the clipping magnitude $M_{t}$
and step size $\eta_{t}$ are the same as Theorem \ref{thm:str-prob}
without any extra modifications. Additionally, these two in-expectation
bounds are both optimal as they attain the best possible in-expectation
rate $\Omega(T^{\frac{2(1-p)}{p}})$. It is also worth pointing out
that our results have a more explicit dependence on the noise level
$\sigma$ and the Lipschitz constant $G$ compared with \cite{zhang2020adaptive}
as in which there is no explicit assumption on noises as mentioned
before. Notably, the two rates are both adaptive to the noise $\sigma$
while not requiring any prior knowledge on $\sigma$ to set $M_{t}$
and $\eta_{t}$. In other words, when $\sigma=0$, we obtain the optimal
rate $O(T^{-1})$ automatically even not knowing $\sigma$.

\section{Theoretical Analysis\label{sec: analysis}}

We show the ideas for proving our theorems in this section. In Section
\ref{subsec:Two-lemmas}, two of the most fundamental lemmas used
in the proof are presented. In Sections \ref{subsec:lip-prob} and
\ref{subsec:lip-exp}, we focus on the high-probability rate and in-expectation
bound respectively for the case $\mu=0$. Several lemmas used to prove
these two results will be given, the omitted proofs of which are delivered
in Section \ref{sec:app-missing-proofs}. To the end, the proofs of
Theorems \ref{thm:lip-prob}, \ref{thm:lip-prob-fix} and \ref{thm:lip-exp},
\ref{thm:lip-exp-fix} are provided. The proof of the initial distance
adaptive bound is provided in Section \ref{sec:app-dog}. The analysis
for the strongly convex case (i.e., $\mu>0$) is deferred into Section
\ref{sec:app-str} in the appendix. 

Before going through the proof, we introduce some notations used in
the analysis. Let $\F_{t}=\sigma(\hp F(x_{1}),\cdots,\hp F(x_{t}))$
be the natural filtration. Under this definition, $x_{t}$ is $\F_{t-1}$
measurable. $\E_{t}[\cdot]$ denotes $\E[\cdot\mid\F_{t-1}]$ for
brevity. We also employ the following definitions:
\begin{align*}
\Delta_{t}\coloneqq & F(x_{t})-F(x_{*});\quad\pa_{t}\coloneqq\E_{t}\left[\hp F(x_{t})\right]\in\pa F(x_{t});\\
\xi_{t}\coloneqq & g_{t}-\pa_{t};\quad\xi_{t}^{u}\coloneqq g_{t}-\E_{t}\left[g_{t}\right];\quad\xi_{t}^{b}\coloneqq\E_{t}\left[g_{t}\right]-\pa_{t};\\
d_{t}\coloneqq & \left\Vert x_{t}-x_{*}\right\Vert ;\quad D_{t}\coloneqq\max_{s\in\left[t\right]}d_{s};\quad\mathfrak{D}_{t}\coloneqq D_{t}\lor\alpha;
\end{align*}
where $\alpha>0$ is the parameter used to set the step size $\eta_{t}$.
We remark that $\xi_{t},\xi_{t}^{u}\in\F_{t}$ and $\pa_{t},\xi_{t}^{b},d_{t},D_{t},\mathfrak{D}_{t}\in\F_{t-1}$.

\subsection{Fundamental Lemmas\label{subsec:Two-lemmas}}

To start with the analysis, we introduce Lemmas \ref{lem:err-bound}
and \ref{lem:basic}, which serve as foundations in our proof. 
\begin{lem}
\label{lem:err-bound}For any $t\in\left[T\right]$, if $M_{t}\geq2G$,
we have
\begin{align*}
\|\xi_{t}^{u}\| & \le2M_{t};\quad\E_{t}[\|\xi_{t}^{u}\|^{2}]\le10\sigma^{p}M_{t}^{2-p};\\
\|\xi_{t}^{b}\| & \le2\sigma^{p}M_{t}^{1-p};\quad\|\xi_{t}^{b}\|^{2}\leq10\sigma^{p}M_{t}^{2-p}.
\end{align*}
\end{lem}
Several similar results (except the bound on $\|\xi_{t}^{b}\|^{2}$)
to Lemma \ref{lem:err-bound} appear in \cite{zhang2020adaptive,gorbunov2020stochastic,gorbunov2021near,zhang2022parameter,sadiev2023high,nguyen2023high,liu2023breaking}
before. However, every existing analysis only considers $\ell_{2}$
norm. When a general norm is used, we provide an extended version
of Lemma \ref{lem:err-bound}, Lemma \ref{lem:err-bound-general}
in Section \ref{sec:general-norm} in the appendix along with the
proof. We refer the interested reader to Section \ref{sec:general-norm}
for details.

From a high-level overview, Lemma \ref{lem:err-bound} tells us how
small the errors $\xi_{t}^{u}$ and $\xi_{t}^{b}$ can be when $M_{t}\geq2G$.
Note that the part of $\xi_{t}^{u}$ will get larger as $M_{t}$ becomes
bigger, in contrast, the error $\xi_{t}^{b}$ (consider the bound
of $\|\xi_{t}^{b}\|\leq2\sigma^{p}M_{t}^{1-p}$) will decrease since
$1-p<0$ given $p\in(1,2]$. So $M_{t}$ should be chosen appropriately
to balance the order between $\|\xi_{t}^{u}\|$ and $\|\xi_{t}^{b}\|$.
Besides, note that our choice of $M_{t}$ always satisfies the condition
$M_{t}\geq2G$. Hereinafter, we will apply Lemma \ref{lem:err-bound}
directly in the analysis.
\begin{lem}
\label{lem:basic}For any $t\in\left[T\right]$, we have
\[
\Delta_{t}+\frac{\eta_{t}^{-1}}{2}d_{t+1}^{2}-\frac{\eta_{t}^{-1}-\mu}{2}d_{t}^{2}\leq\langle\xi_{t},x_{*}-x_{t}\rangle+\eta_{t}\left(2\left\Vert \xi_{t}^{u}\right\Vert ^{2}+2\left\Vert \xi_{t}^{b}\right\Vert ^{2}+G^{2}\right).
\]
\end{lem}
Lemma \ref{lem:basic} is the basic inequality used to bound both
the function value gap $\Delta_{t}$ and the distance term $d_{t}^{2}$.
The same as Lemma \ref{lem:err-bound}, we will present a generalized
version (Lemma \ref{lem:basic-general} in Section \ref{sec:general-norm})
for the general norm with its proof in the appendix.

Now let us explain Lemma \ref{lem:basic} a bit more here. If $\xi_{t}=\bzero$,
then one can view Lemma \ref{lem:basic} as a one-step descent lemma.
However, even after taking the expectation on both sides, the term
$\E[\langle\xi_{t},x_{*}-x_{t}\rangle]$ won't vanish since $\xi_{t}=g_{t}-\pa_{t}$
but $g_{t}$ is not an unbiased gradient estimator of $\pa_{t}$ due
to the clipping. So one of the hard parts of the analysis is how to
deal with the term $\langle\xi_{t},x_{*}-x_{t}\rangle$ both in expectation
and in a high-probability way. 

Another challenge is to deal with $\|\xi_{t}^{u}\|^{2}$ and $\|\xi_{t}^{b}\|^{2}$.
For $\|\xi_{t}^{b}\|^{2}$, Lemma \ref{lem:err-bound} already tells
us how to bound it. As for $\|\xi_{t}^{u}\|^{2}$, though it can be
bounded by $4M_{t}^{2}$ (by Lemma \ref{lem:err-bound} again). However,
this simple bound is not enough to obtain the correct order. For the
in-expectation analysis, after taking expectations on both sides of
Lemma \ref{lem:basic}, we can instead use the bound $\E_{t}[\|\xi_{t}^{u}\|^{2}]\leq10\sigma^{p}M_{t}^{2-p}$
in Lemma \ref{lem:err-bound}, which turns to be in a strictly smaller
order compared with $4M_{t}^{2}$ under our choice of $M_{t}$. Hence,
for the more complicated high-probability analysis, a hint of using
the bound on $\E_{t}[\|\xi_{t}^{u}\|^{2}]$ arises from analyzing
the in-expectation bound. As such, a natrual decomposition, $\|\xi_{t}^{u}\|^{2}=\|\xi_{t}^{u}\|^{2}-\E_{t}[\|\xi_{t}^{u}\|^{2}]+\E_{t}[\|\xi_{t}^{u}\|^{2}]$,
shows up. We make this thought formally and show a high-probability
bound of $\|\xi_{t}^{u}\|^{2}-\E_{t}[\|\xi_{t}^{u}\|^{2}]$ in Lemma
\ref{lem:xi-u-lip} in next section.

\subsection{High-Probability Analysis when $\mu=0$\label{subsec:lip-prob}}

In this section, our ultimate goal is to prove the high-probability
convergence rate, i.e., Theorems \ref{thm:lip-prob} and \ref{thm:lip-prob-fix}.
To save space, only the lemmas used for the case of unknown $T$ will
be stated formally. We will describe how the lemmas will change for
known $T$ accordingly in the remarks.

First, we present Lemma \ref{lem:normalize}, which is a powerful
tool when specialized to the case $\mu=0$. Lemma \ref{lem:normalize}
is immediately obtained from the definition of $\eta_{t}$, hence,
the proof of which is omitted.
\begin{lem}
\label{lem:normalize}When $\mu=0$, under our choices of $M_{t}$
and $\eta_{t}$ whenever $T$ is known or not, for any $t\in\left[T\right]$,
we have
\[
\eta_{t}M_{t}\leq\alpha.
\]
\end{lem}
Next, we introduce Lemma \ref{lem:basic-lip-prob}, which can be viewed
as a finer result of Lemma \ref{lem:basic} for $\mu=0$. 
\begin{lem}
\label{lem:basic-lip-prob}When $\mu=0$, under the choices of $M_{t}=2G\lor Mt^{\frac{1}{p}}$
and $\eta_{t}=\frac{\alpha}{G\sqrt{t}}\land\frac{\alpha}{M_{t}}$,
for any $\tau\in\left[T\right]$, we have
\[
d_{\tau+1}^{2}+\sum_{t=1}^{\tau}2\eta_{t}\Delta_{t}\leq\mathfrak{D}_{\tau}\left(d_{1}+h(\tau)+\sum_{t=1}^{\tau}2\eta_{t}\left\langle \xi_{t},\frac{x_{*}-x_{t}}{\mathfrak{D}_{t}}\right\rangle +\frac{4\eta_{t}^{2}}{\alpha}\left(\left\Vert \xi_{t}^{u}\right\Vert ^{2}-\E_{t}\left[\left\Vert \xi_{t}^{u}\right\Vert ^{2}\right]\right)\right)
\]
where
\[
h(\tau)=2(1+40(\sigma/M)^{p})\alpha\log(e\tau).
\]
\end{lem}
\begin{rem}
\label{rem:basic-lip-prob}For the case of known $T$, under the choices
of $M_{t}=2G\lor MT^{\frac{1}{p}}$ and $\eta_{t}=\frac{\alpha}{G\sqrt{T}}\land\frac{\alpha}{M_{t}}$,
$h(\tau)$ in Lemma \ref{lem:basic-lip-prob} will be $2(1+40(\sigma/M)^{p})\alpha$.
\end{rem}
Lemma \ref{lem:basic-lip-prob} is interesting in several ways. As
mentioned above, to use the conditional expectation bound on $\|\xi_{t}^{u}\|^{2}$,
the term $\eta_{t}^{2}(\|\xi_{t}^{u}\|^{2}-\E_{t}[\|\xi_{t}^{u}\|^{2}])$
appears, which can be bounded by Freedman's inequality (Lemma \ref{lem:freedman}).
Next, the inner product $\eta_{t}\langle\xi_{t},\frac{x_{*}-x_{t}}{\mathfrak{D}_{t}}\rangle$
seems very strange at first glance. However, this term can help us
to obtain an $O(\log(1/\delta))$ dependence on $\delta$ finally
instead of the sub-optimal $\log(T/\delta)$ shown in previous works.
We briefly explain why this is the right term here. Due to $\xi_{t}$
in the inner product, to use the bound on $\|\xi_{t}^{u}\|$ and $\|\xi_{t}^{b}\|$,
it is natural to consider $\eta_{t}\langle\xi_{t},\frac{x_{*}-x_{t}}{\mathfrak{D}_{t}}\rangle=\eta_{t}\langle\xi_{t}^{b}+\xi_{t}^{u},\frac{x_{*}-x_{t}}{\mathfrak{D}_{t}}\rangle$.
For the term $\sum_{t=1}^{\tau}\eta_{t}\langle\xi_{t}^{b},\frac{x_{*}-x_{t}}{\mathfrak{D}_{t}}\rangle$,
it will be bounded by $O(\log T)$ in the proof of Theorem \ref{thm:lip-prob}
directly. The other term, $\sum_{t=1}^{\tau}\eta_{t}\langle\xi_{t}^{u},\frac{x_{*}-x_{t}}{\mathfrak{D}_{t}}\rangle$,
is more interesting. A key observation is that $\eta_{t}\langle\xi_{t}^{u},\frac{x_{*}-x_{t}}{\mathfrak{D}_{t}}\rangle$
is a martingale difference sequence, which may allow us to use the
concentration inequality to bound the summation again. Because of
the divisor $\mathfrak{D}_{t}$, $\eta_{t}\langle\xi_{t}^{u},\frac{x_{*}-x_{t}}{\mathfrak{D}_{t}}\rangle$
admit an almost surely time uniform bound which implies we can apply
Freedman's inequality directly. As a result, we are able to obtain
a time-uniform high-probability bound at once.

In contrast, lots of existing works are to bound $\sum_{t=1}^{\tau}\eta_{t}\langle\xi_{t}^{u},x_{*}-x_{t}\rangle$,
however, $\eta_{t}\langle\xi_{t}^{u},x_{*}-x_{t}\rangle$ doesn't
have an almost surely bound necessarily. For such a reason, prior
works need to first bound the distance $d_{t}=\|x_{t}-x_{*}\|$ with
a high probability by induction then go back to bound $\sum_{t=1}^{\tau}\eta_{t}\langle\xi_{t}^{u},x_{*}-x_{t}\rangle$
for every $\tau$ by employing the high-probability bound on $d_{t}$.
This kind of roundabout argument leads to the sub-optimal dependence
of $\log(T/\delta)$.

Now we provide the desired bound described above in Lemma \ref{lem:lip-prob-concen},
As one can see, our bound holds for any $\tau\in[T]$ uniformly, hence,
which lifts the extra $\log T$ term. We refer the reader to Section
\ref{sec:app-missing-proofs} for more details of our proof of Lemma
\ref{lem:lip-prob-concen}.
\begin{lem}
\label{lem:lip-prob-concen}When $\mu=0$, under the choices of $M_{t}=2G\lor Mt^{\frac{1}{p}}$
and $\eta_{t}=\frac{\alpha}{G\sqrt{t}}\land\frac{\alpha}{M_{t}}$,
we have with probability at least $1-\frac{\delta}{2}$, for any $\tau\in\left[T\right]$,
\[
\sum_{t=1}^{\tau}\eta_{t}\left\langle \xi_{t}^{u},\frac{x_{*}-x_{t}}{\mathfrak{D}_{t}}\right\rangle \leq5\left(\log\frac{4}{\delta}+\sqrt{(\sigma/M)^{p}\log(eT)\log\frac{4}{\delta}}\right)\alpha.
\]
\end{lem}
\begin{rem}
\label{rem:lib-prob-loglog}For the case of known $T$, under the
choices of $M_{t}=2G\lor MT^{\frac{1}{p}}$ and $\eta_{t}=\frac{\alpha}{G\sqrt{T}}\land\frac{\alpha}{M_{t}}$,
the term $\log(eT)$ in Lemma \ref{lem:lip-prob-concen} will be removed
(or replaced by $1$ equivalently).
\end{rem}
Next, in Lemma \ref{lem:xi-u-lip}, we provide an any time high-probability
bound of the term $\sum_{t=1}^{\tau}\frac{\eta_{t}^{2}}{\alpha}(\|\xi_{t}^{u}\|^{2}-\E_{t}[\|\xi_{t}^{u}\|^{2}])$
by using Freedman's inequality.
\begin{lem}
\label{lem:xi-u-lip}When $\mu=0$, under the choices of $M_{t}=2G\lor Mt^{\frac{1}{p}}$
and $\eta_{t}=\frac{\alpha}{G\sqrt{t}}\land\frac{\alpha}{M_{t}}$,
we have with probability at least $1-\frac{\delta}{2}$, for any $\tau\in\left[T\right]$,
\[
\sum_{t=1}^{\tau}\frac{\eta_{t}^{2}}{\alpha}\left(\left\Vert \xi_{t}^{u}\right\Vert ^{2}-\E_{t}\left[\left\Vert \xi_{t}^{u}\right\Vert ^{2}\right]\right)\leq9\left(\log\frac{4}{\delta}+\sqrt{(\sigma/M)^{p}\log(eT)\log\frac{4}{\delta}}\right)\alpha.
\]
\end{lem}
\begin{rem}
\label{rem:xi-u-lip}For the case of known $T$, under the choices
of $M_{t}=2G\lor MT^{\frac{1}{p}}$ and $\eta_{t}=\frac{\alpha}{G\sqrt{T}}\land\frac{\alpha}{M_{t}}$,
the term $\log(eT)$ in both bounds in Lemma \ref{lem:xi-u-lip} will
be removed (or replaced by $1$ equivalently).
\end{rem}
Equipped with the above lemmas, we are finally able to prove Theorems
\ref{thm:lip-prob} and \ref{thm:lip-prob-fix}.

\subsubsection{Proof of Theorem \ref{thm:lip-prob} }

\begin{proof}[Proof of Theorem \ref{thm:lip-prob}]
We first define a constant $K$ as follows
\begin{align}
K\coloneqq & \alpha^{2}+\left(d_{1}+2\alpha\log(eT)+107(\sigma/M)^{p}\alpha\log(eT)+69\alpha\log\frac{4}{\delta}\right)^{2}\label{eq:lip-prob-def-k}\\
= & O\left(\alpha^{2}\left((1+(\sigma/M)^{2p})\log^{2}T+\log^{2}(1/\delta)\right)+\left\Vert x_{1}-x_{*}\right\Vert ^{2}\right).\nonumber 
\end{align}

We sart with Lemma \ref{lem:basic-lip-prob} to get for any $\tau\in\left[T\right]$,
there is
\begin{align}
d_{\tau+1}^{2}+\sum_{t=1}^{\tau}2\eta_{t}\Delta_{t}\leq & \mathfrak{D}_{\tau}\left(d_{1}+h(\tau)+\sum_{t=1}^{\tau}2\eta_{t}\left\langle \xi_{t},\frac{x_{*}-x_{t}}{\mathfrak{D}_{t}}\right\rangle +\frac{4\eta_{t}^{2}}{\alpha}\left(\left\Vert \xi_{t}^{u}\right\Vert ^{2}-\E_{t}\left[\left\Vert \xi_{t}^{u}\right\Vert ^{2}\right]\right)\right)\nonumber \\
= & \mathfrak{D}_{\tau}\left(d_{1}+h(\tau)+\sum_{t=1}^{\tau}2\eta_{t}\left\langle \xi_{t}^{b},\frac{x_{*}-x_{t}}{\mathfrak{D}_{t}}\right\rangle +2\eta_{t}\left\langle \xi_{t}^{u},\frac{x_{*}-x_{t}}{\mathfrak{D}_{t}}\right\rangle +\frac{4\eta_{t}^{2}}{\alpha}\left(\left\Vert \xi_{t}^{u}\right\Vert ^{2}-\E_{t}\left[\left\Vert \xi_{t}^{u}\right\Vert ^{2}\right]\right)\right)\label{eq:lip-1}
\end{align}
\begin{itemize}
\item Bounding the term $\sum_{t=1}^{\tau}2\eta_{t}\left\langle \xi_{t}^{b},\frac{x_{*}-x_{t}}{\mathfrak{D}_{t}}\right\rangle $:
We know
\begin{align}
\sum_{t=1}^{\tau}2\eta_{t}\left\langle \xi_{t}^{b},\frac{x_{*}-x_{t}}{\mathfrak{D}_{t}}\right\rangle  & \leq\sum_{t=1}^{\tau}2\eta_{t}\left\Vert \xi_{t}^{b}\right\Vert \frac{d_{t}}{\mathfrak{D}_{t}}\overset{(a)}{\leq}\sum_{t=1}^{\tau}2\cdot\eta_{t}\cdot2\sigma^{p}M_{t}^{1-p}\nonumber \\
 & \overset{(b)}{\leq}\sum_{t=1}^{\tau}\frac{4(\sigma/M)^{p}\alpha}{t}\leq4(\sigma/M)^{p}\alpha\log(e\tau)\label{eq:lip-xi-b}
\end{align}
where $(a)$ is due to $\left\Vert \xi_{t}^{b}\right\Vert \leq2\sigma^{p}M_{t}^{1-p}$
by Lemma \ref{lem:err-bound} and $d_{t}\leq\mathfrak{D}_{t}$. $(b)$
is by $\eta_{t}M_{t}\leq\alpha$ from Lemma \ref{lem:normalize} and
$M_{t}\geq Mt^{\frac{1}{p}}$ from our choice.
\item Bounding the term $\sum_{t=1}^{\tau}2\eta_{t}\left\langle \xi_{t}^{u},\frac{x_{*}-x_{t}}{\mathfrak{D}_{t}}\right\rangle $:
By Lemma \ref{lem:lip-prob-concen}, we have with probability at least
$1-\frac{\delta}{2}$, for any $\tau\in\left[T\right]$:
\begin{equation}
\sum_{t=1}^{\tau}2\eta_{t}\left\langle \xi_{t}^{u},\frac{x_{*}-x_{t}}{\mathfrak{D}_{t}}\right\rangle \leq10\left(\log\frac{4}{\delta}+\sqrt{(\sigma/M)^{p}\log(eT)\log\frac{4}{\delta}}\right)\alpha.\label{eq:lip-xi-u}
\end{equation}
\item Bounding the term $\sum_{t=1}^{\tau}\frac{4\eta_{t}^{2}}{\alpha}\left(\left\Vert \xi_{t}^{u}\right\Vert ^{2}-\E_{t}\left[\left\Vert \xi_{t}^{u}\right\Vert ^{2}\right]\right)$:
By Lemma \ref{lem:xi-u-lip}, we have with probability at least $1-\frac{\delta}{2}$,
for any $\tau\in\left[T\right]$:
\begin{equation}
\sum_{t=1}^{\tau}\frac{4\eta_{t}^{2}}{\alpha}\left(\left\Vert \xi_{t}^{u}\right\Vert ^{2}-\E_{t}\left[\left\Vert \xi_{t}^{u}\right\Vert ^{2}\right]\right)\leq36\left(\log\frac{4}{\delta}+\sqrt{(\sigma/M)^{p}\log(eT)\log\frac{4}{\delta}}\right)\alpha.\label{eq:lip-xi-u-2}
\end{equation}
\end{itemize}
Combining (\ref{eq:lip-1}), (\ref{eq:lip-xi-b}), (\ref{eq:lip-xi-u})
and (\ref{eq:lip-xi-u-2}), we have with probability at least $1-\delta$,
for any $\tau\in\left[T\right]$:
\begin{align*}
d_{\tau+1}^{2}+\sum_{t=1}^{\tau}2\eta_{t}\Delta_{t}\leq & \mathfrak{D}_{\tau}\left(d_{1}+h(\tau)+4(\sigma/M)^{p}\alpha\log(e\tau)+46\left(\log\frac{4}{\delta}+\sqrt{(\sigma/M)^{p}\log(eT)\log\frac{4}{\delta}}\right)\alpha\right)\\
\overset{(c)}{\leq} & \mathfrak{D}_{\tau}\left(d_{1}+2\alpha\log(eT)+84(\sigma/M)^{p}\alpha\log(eT)+46\left(\log\frac{4}{\delta}+\sqrt{(\sigma/M)^{p}\log(eT)\log\frac{4}{\delta}}\right)\alpha\right)\\
\leq & \mathfrak{D}_{\tau}\left(d_{1}+2\alpha\log(eT)+107(\sigma/M)^{p}\alpha\log(eT)+69\alpha\log\frac{4}{\delta}\right)\\
\leq & \frac{\mathfrak{D_{\tau}^{2}}+\left(d_{1}+2\alpha\log(eT)+107(\sigma/M)^{p}\alpha\log(eT)+69\alpha\log\frac{4}{\delta}\right)^{2}}{2}\\
\overset{(d)}{\leq} & \frac{D_{\tau}^{2}+\alpha^{2}+\left(d_{1}+2\alpha\log(eT)+107(\sigma/M)^{p}\alpha\log(eT)+69\alpha\log\frac{4}{\delta}\right)^{2}}{2}\\
\overset{(e)}{=} & \frac{D_{\tau}^{2}+K}{2}
\end{align*}
where $(c)$ is by plugging in $h(\tau)=2(1+40(\sigma/M)^{p})\alpha\log(e\tau)\leq2(1+40(\sigma/M)^{p})\alpha\log(eT)$;
$(d)$ is by $\mathfrak{D}_{\tau}^{2}=(D_{\tau}\lor\alpha)^{2}\leq D_{\tau}^{2}+\alpha^{2}$;
$(e)$ is due to the definition of $K$ (see (\ref{eq:lip-prob-def-k})).
Hence, by using $\Delta_{t}\geq0$, we have for any $\tau\in\left[T\right]$,
\[
d_{\tau+1}^{2}\leq\frac{D_{\tau}^{2}+K}{2},
\]
which implies $d_{t}^{2}\leq D_{t}^{2}\leq K$ for any $t\in\left[T+1\right]$
by simple induction.

Finally, we consider time $T$ to get with probability at least $1-\delta$
\[
\sum_{t=1}^{T}2\eta_{t}\Delta_{t}\leq d_{T+1}^{2}+\sum_{t=1}^{T}2\eta_{t}\Delta_{t}\leq\frac{D_{T}^{2}+K}{2}\leq K.
\]
Note that $\eta_{t}$ is non-increasing and $F(\bar{x}_{T})-F(x_{*})\leq\frac{\sum_{t=1}^{T}\Delta_{t}}{T}$
by the convexity of $F$ where $\bar{x}_{T}=\frac{1}{T}\sum_{t=1}^{T}x_{t}$,
we conclude that
\[
F(\bar{x}_{T})-F(x_{*})\leq\frac{K}{2\eta_{T}T}.
\]
Plugging $K$ and $\eta_{t}$, we get the desired result.
\end{proof}

\subsubsection{Proof of Theorem \ref{thm:lip-prob-fix}}

\begin{proof}[Proof of Theorem \ref{thm:lip-prob-fix}]
Following a similar proof of Theorem \ref{thm:lip-prob}, we will
obtain that with probability at least $1-\delta$, for any $\tau\in\left[T\right]$,
there is
\[
d_{\tau+1}^{2}+\sum_{t=1}^{\tau}2\eta_{t}\Delta_{t}\leq K
\]
where $K$ is currently defined as
\begin{align}
K\coloneqq & \alpha^{2}+\left(d_{1}+2\alpha+107(\sigma/M)^{p}\alpha+69\alpha\log\frac{4}{\delta}\right)^{2}\label{eq:lip-prob-fix-def-k}\\
= & O\left(\alpha^{2}\left((\sigma/M)^{2p}+\log^{2}(1/\delta)\right)+\left\Vert x_{1}-x_{*}\right\Vert ^{2}\right).\nonumber 
\end{align}

Hence, we have with probability at least $1-\delta$,
\[
F(\bar{x}_{T})-F(x_{*})\leq\frac{1}{T}\sum_{t=1}^{T}\Delta_{t}\leq\frac{K}{2\eta_{T}T}
\]
where $\bar{x}_{T}=\frac{1}{T}\sum_{t=1}^{T}x_{t}$. Finally, plugging
in $K$ and $\eta_{T}$, the proof is finished.
\end{proof}

\subsection{In-Expectation Analysis when $\mu=0$\label{subsec:lip-exp}}

Now we turn to the in-expectation bound of Algorithm \ref{alg:algo}
for the general convex case. We introduce Lemma \ref{lem:basic-lip-exp}.
which is enough to let us prove Theorems \ref{thm:lip-exp} and \ref{thm:lip-exp-fix}.
\begin{lem}
\label{lem:basic-lip-exp}When $\mu=0$, for any $\tau\in\left[T\right]$,
we have
\[
\frac{\E\left[d_{\tau+1}^{2}\right]-\E\left[d_{1}^{2}\right]}{2}+\sum_{t=1}^{\tau}\eta_{t}\E\left[\Delta_{t}\right]\leq\frac{\alpha h(\tau)}{2}+2(\sigma/M)^{p}\alpha\sum_{t=1}^{\tau}\frac{\sqrt{\E\left[d_{t}^{2}\right]}}{t}
\]
where $h(\tau)$ is defined in Lemma \ref{lem:basic-lip-prob}.
\end{lem}
\begin{rem}
For the case of known $T$, under the choices of $M_{t}=2G\lor MT^{\frac{1}{p}}$
and $\eta_{t}=\frac{\alpha}{G\sqrt{T}}\land\frac{\alpha}{M_{t}}$,
in Lemma \ref{lem:basic-lip-exp}, $\sqrt{\E[d_{t}^{2}]}/t$ will
be replaced by $\sqrt{\E[d_{t}^{2}]}/T$ and $h(\tau)$ will be changed
according to Remark \ref{rem:basic-lip-prob}.
\end{rem}
Note that Lemma \ref{lem:basic-lip-prob} is a very interesting result.
On the L.H.S., there is $\E[d_{\tau+1}^{2}]$, however, the R.H.S.
only has $\sqrt{\E[d_{t}^{2}]}$ for any $t\le\tau$. If we assume
$\E[d_{t}^{2}]$ can be uniformly bounded by some $K$ for any $t\leq\tau$,
we know $\sum_{t=1}^{\tau}\sqrt{\E[d_{t}^{2}]}/t=O(\sqrt{K}\log T)$,
which implies $\E[d_{\tau+1}^{2}]=O(\alpha h(\tau)+(\sigma/M)^{p}\alpha\sqrt{K}\log T+d_{1}^{2})$.
Such a result tells us $\E[d_{\tau+1}^{2}]$ can still be bounded
by $K$ once $K$ is picked properly. So we can expect a uniform bound
on $\max_{s\in\left[t\right]}\E[d_{s}^{2}]$ for any $t\leq T$. We
will show how to demonstrate this idea formally in the proof of Theorems
\ref{thm:lip-exp} and \ref{thm:lip-exp-fix}.

\subsubsection{Proof of Theorem \ref{thm:lip-exp}}

\begin{proof}[Proof of Theorem \ref{thm:lip-exp}]
We first define the notation $\mathcal{D}_{t}=\max_{s\in\left[t\right]}\sqrt{\E\left[d_{s}^{2}\right]}$.
Next, consider the following constant
\begin{align}
K\coloneqq & 16(\sigma/M)^{2p}\alpha^{2}\log^{2}(eT)+4(1+40(\sigma/M)^{p})\alpha^{2}\log(eT)+2d_{1}^{2}\label{eq:lip-exp-def-K}\\
= & O\left(\alpha^{2}\left(\log T+(\sigma/M)^{2p}\log^{2}T\right)+\left\Vert x_{1}-x_{*}\right\Vert ^{2}\right).\nonumber 
\end{align}

We invoke Lemma \ref{lem:basic-lip-exp} to get for any time $\tau\in\left[T\right]$,
\begin{align*}
\frac{\E\left[d_{\tau+1}^{2}\right]-\E\left[d_{1}^{2}\right]}{2}+\sum_{t=1}^{\tau}\eta_{t}\E\left[\Delta_{t}\right]\leq & \frac{\alpha h(\tau)}{2}+2(\sigma/M)^{p}\alpha\sum_{t=1}^{\tau}\frac{\sqrt{\E\left[d_{t}^{2}\right]}}{t}\\
\overset{(a)}{\leq} & \frac{\alpha h(\tau)}{2}+2\alpha\sum_{t=1}^{\tau}\frac{\mathcal{D}_{\tau}}{t}\leq\frac{\alpha h(\tau)}{2}+2(\sigma/M)^{p}\alpha\log(eT)\mathcal{D}_{\tau}\\
\Rightarrow\frac{\E\left[d_{\tau+1}^{2}\right]}{2}+\sum_{t=1}^{\tau}\eta_{t}\E\left[\Delta_{t}\right]\overset{(b)}{\leq} & \frac{\mathcal{D}_{\tau}^{2}}{4}+4(\sigma/M)^{2p}\alpha^{2}\log^{2}(eT)+(1+40(\sigma/M)^{p})\alpha^{2}\log(eT)+\frac{d_{1}^{2}}{2}\\
\overset{(c)}{=} & \frac{\mathcal{D}_{\tau}^{2}}{4}+\frac{K}{4}
\end{align*}
where $(a)$ is by the definition of $\mathcal{D}_{\tau}$; $(b)$
is by using due to AM-GM inequality for the term $2(\sigma/M)^{p}\alpha\log(eT)\mathcal{D}_{\tau}$
and plugging in $\alpha h(\tau)/2=(1+40(\sigma/M)^{p})\alpha^{2}\log(e\tau)\leq(1+40(\sigma/M)^{p})\alpha^{2}\log(eT)$;
$(c)$ is from the definition of $K$ (see (\ref{eq:lip-exp-def-K})).
Thus, for any $\tau\in\left[T\right]$, there is
\[
\frac{\E\left[d_{\tau+1}^{2}\right]}{2}\leq\frac{\E\left[d_{\tau+1}^{2}\right]}{2}+\sum_{t=1}^{\tau}\eta_{t}\E\left[\Delta_{t}\right]\leq\frac{\mathcal{D}_{\tau}^{2}}{4}+\frac{K}{4}\Rightarrow d_{\tau+1}^{2}\leq\frac{\mathcal{D}_{\tau}^{2}}{2}+\frac{K}{2},
\]
which implies $\E\left[d_{t}^{2}\right]\leq\mathcal{D}_{t}^{2}\leq K$
for any $t\in\left[T+1\right]$ by simple induction.

Finally, for time $T$, we know
\[
\sum_{t=1}^{T}\eta_{t}\E\left[\Delta_{t}\right]\leq\frac{\E\left[d_{T+1}^{2}\right]}{2}+\sum_{t=1}^{T}\eta_{t}\E\left[\Delta_{t}\right]\leq\frac{\mathcal{D}_{T}^{2}}{4}+\frac{K}{4}\leq\frac{K}{2}.
\]
Note that is $\eta_{t}$ non-increasing and $\E\left[F(\bar{x}_{T})-F(x_{*})\right]\leq\frac{\sum_{t=1}^{T}\E\left[\Delta_{t}\right]}{T}$
by the convexity of $F$ where $\bar{x}_{T}=\frac{1}{T}\sum_{t=1}^{T}x_{t}$,
we conclude that
\[
\E\left[F(\bar{x}_{T})-F(x_{*})\right]\leq\frac{K}{2\eta_{T}T}.
\]
Plugging $K$ and $\eta_{t}$, we get the desired result.
\end{proof}

\subsubsection{Proof of Theorem \ref{thm:lip-exp-fix}.}

\begin{proof}[Proof of Theorem \ref{thm:lip-exp-fix}]
Following the same line in the proof of Theorem \ref{thm:lip-exp},
we only need to notice that the constant $K$ now is defined as
\begin{align}
K\coloneqq & 16(\sigma/M)^{2p}\alpha^{2}+4(1+40(\sigma/M)^{p})\alpha^{2}+2d_{1}^{2}\label{eq:lip-exp-known-def-K}\\
= & O\left(\alpha^{2}\left(1+(\sigma/M)^{2p}\right)+\left\Vert x_{1}-x_{*}\right\Vert ^{2}\right).\nonumber 
\end{align}
By similar steps, we will reach $\E\left[F(\bar{x}_{T})-F(x_{*})\right]\leq\frac{K}{2\eta_{T}T}$
again where $\bar{x}_{T}=\frac{1}{T}\sum_{t=1}^{T}x_{t}$. Plugging
$K$ and $\eta_{t}$ for two cases respectively, we get the desired
result.
\end{proof}

\section{Conclusion\label{sec:conclusion}}

In this paper, we present a comprehensive analysis of stochastic
nonsmooth optimization with heavy-tailed noises and obtain several
new results. More specifically, under properly picked parameters,
we show a simple clipping algorithm provably converges both in expectation
and probability for convex or strongly convex objectives. Furthermore,
no matter whether the time horizon $T$ or noise level $\sigma$ is
known or not, our choices of clipping magnitude and step size still
guarantee (nearly) optimal in-expectation and high-probability rates.

However, there still remains an interesting direction worth exploring.
The same as the previous works, our results heavily rely on the prior
knowledge of $p$, $G$ and $\mu$ (when considering strongly convex
functions), all of which may be hard to estimate in practice. Hence,
finding an algorithm without requiring any parameters is very important
for both theoretical and practical sides. We leave this important
question as future work and expect it to be addressed.

\clearpage

\bibliographystyle{plain}
\bibliography{ref}

\newpage

\appendix
\onecolumn

\section{Extension to Arbitrary Norms\label{sec:general-norm}}

In this section, we relax the $\ell_{2}$ norm used in the previous
assumptions (see Section \ref{sec:Preliminaries}) to an arbitrary
norm $\|\cdot\|$ on $\R^{d}$. $\|\cdot\|_{*}$ denotes the dual
norm of $\|\cdot\|$ induced by $\langle\cdot,\cdot\rangle$. Additionally,
let $\psi$ be a differentiable and $1$-strongly convex function
with respect to $\|\cdot\|$ on $\dom$, i.e.,
\[
\psi(x)\geq\psi(y)+\langle\na\psi(y),x-y\rangle+\frac{1}{2}\left\Vert x-y\right\Vert ^{2},\forall x,y\in\dom.
\]
We note that, rigorously speaking, $y$ can only be chosen in $\mathrm{int}(\dom)$.
However, one can think there is $\dom\subseteq\mathrm{int}(\mathrm{dom}(\psi))$
to avoid this potential issue. Now, define the Bregman divergence
with respect to $\psi$ as
\[
D_{\psi}(x,y)=\psi(x)-\psi(y)-\langle\na\psi(y),x-y\rangle.
\]
Note that $D_{\psi}(x,y)\geq\frac{1}{2}\|x-y\|^{2}$ from the $1$-strongly
convexity assumption of $\psi$. In particular $D_{\psi}(x,y)=\frac{\|x-y\|_{2}^{2}}{2}$
when considering $\|\cdot\|=\|\cdot\|_{2}$ as used in the main text
and $\psi(x)=\frac{1}{2}\|x\|_{2}^{2}$.

\subsection{New Assumptions and A Useful Fact}

With the above preparations, we can provide new assumptions under
the general norm.

\textbf{1. Existence of a local minimizer}: $\exists x_{*}\in\argmin_{x\in\dom}F(x)$
satisfying $F(x_{*})>-\infty$.

\textbf{2'. Relatively }$\mu$\textbf{-strongly convex}: $\exists\mu\geq0$
such that $F(x)\geq F(y)+\langle g,x-y\rangle+\mu D_{\psi}(x,y),\forall x,y\in\dom,g\in\pa F(y)$.

\textbf{3'. }$G$\textbf{-Lipschitz}: $\exists G>0$ such that $\|g\|_{*}\leq G,\forall x\in\dom,g\in\pa F(x)$ 

\textbf{4. Unbiased gradient estimator}: We are able to access a history-independent,
unbiased gradient estimator $\hp F(x)$ for any $x\in\dom$. In other
words, $\E[\widehat{\pa}F(x)\vert x]\in\pa F(x),\forall x\in\dom$.

\textbf{5'. Bounded $p$th moment noise}: There exist $p\in(1,2]$
and $\sigma\geq0$ denoting the noise level such that $\E[\|\hp F(x)-\E[\widehat{\pa}F(x)\vert x]\|_{*}^{p}\vert x]\leq\sigma^{p}$.

The concept of relatively strong convexity in Assumption 2' is introduced
in \cite{lu2018relatively}. Note that when $\|\cdot\|=\|\cdot\|_{2}$,
Assumptions 2', 3' and 5' are the same as Assumptions 2, 3 and 5 in
Section \ref{sec:Preliminaries}. Hence, these new assumptions are
more general. Next, we provide a useful fact under Assumptions 1,
2' and 3'. This result can help us to simplify the final bound in
the proof of Theorem \ref{thm:str-prob}.
\begin{fact}
\label{fact:fact}Under Assumption 1, 2' and 3' with $\mu>0$, there
is
\[
\left\Vert x-x_{*}\right\Vert ^{2}\leq2D_{\psi}(x,x_{*})\leq\frac{G^{2}}{\mu^{2}},\forall x\in\dom.
\]
\end{fact}
\begin{proof}
Given $x\in\dom$, by assumption 2', for any fixed $g\in\pa F(x)$,
we have
\begin{align*}
F(x_{*})\geq & F(x)+\langle g,x_{*}-x\rangle+\mu D_{\psi}(x_{*},x)\\
\overset{(a)}{\geq} & F(x)-\left\Vert g\right\Vert _{*}\left\Vert x_{*}-x\right\Vert +\mu D_{\psi}(x_{*},x)\\
\overset{(b)}{\geq} & F(x)-\left\Vert g\right\Vert _{*}\left\Vert x_{*}-x\right\Vert +\frac{\mu}{2}\left\Vert x_{*}-x\right\Vert ^{2}\\
\overset{(c)}{\geq} & F(x)-\frac{\left\Vert g\right\Vert _{*}^{2}}{2\mu}\overset{(d)}{\geq}F(x)-\frac{G^{2}}{2\mu}\\
\Rightarrow\frac{G^{2}}{2\mu}\geq & F(x)-F(x_{*})
\end{align*}
where $(a)$ is due to Cauchy-Schwarz inequality; $(b)$ is by $D_{\psi}(x,y)\geq\frac{1}{2}\left\Vert x-y\right\Vert ^{2}$;
$(c)$ is because of Young's inequality; $(d)$ is by Assumption 3'. 

Now by Assumption 2' again, for any fixed $g\in\pa F(x_{*})$, we
have
\begin{align*}
F(x)\geq & F(x_{*})+\langle g,x-x_{*}\rangle+\mu D_{\psi}(x,x_{*})\\
\overset{(e)}{\geq} & F(x_{*})+\mu D_{\psi}(x,x_{*})\overset{(f)}{\geq}F(x_{*})+\frac{\mu}{2}\left\Vert x-x_{*}\right\Vert ^{2}\\
\Rightarrow F(x)-F(x_{*})\geq & \mu D_{\psi}(x,x_{*})
\end{align*}
where $(e)$ is by $\langle g,x-x_{*}\rangle\geq0$ due to $g\in\pa F(x_{*})$
and $x_{*}\in\argmin_{x\in\dom}F(x)$; $(f)$ is by $D_{\psi}(x,y)\geq\frac{1}{2}\left\Vert x-y\right\Vert ^{2}$.

Finally, we know
\[
\frac{\mu}{2}\left\Vert x-x_{*}\right\Vert ^{2}\leq\mu D_{\psi}(x,x_{*})\leq F(x)-F(x_{*})\leq\frac{G^{2}}{2\mu}\Rightarrow\left\Vert x-x_{*}\right\Vert ^{2}\leq2D_{\psi}(x,x_{*})\leq\frac{G^{2}}{\mu^{2}}.
\]
\end{proof}

\subsection{Algorithm with the General Norm}

\begin{algorithm}[h]
\caption{\label{alg:algo-md}Projected Stochastic MD with Clipping}

\textbf{Input}: $x_{1}\in\dom$, $M_{t}>0$, $\eta_{t}>0$.

\textbf{for} $t=1$ \textbf{to} $T$ \textbf{do}

$\quad$$g_{t}=\left(1\land\frac{M_{t}}{\left\Vert \hp F(x_{t})\right\Vert _{*}}\right)\hp F(x_{t})$

$\quad$$x_{t+1}=\argmin_{x\in\dom}\langle g_{t},x-x_{t}\rangle+\frac{1}{\eta_{t}}D_{\psi}(x,x_{t}).$

\textbf{end for}
\end{algorithm}

With the new assumptions, we provide a general version of Algorithm
\ref{alg:algo} as shown in Algorithm \ref{alg:algo-md}, which employs
the mirror descent framework. Note that when $\left\Vert \cdot\right\Vert =\left\Vert \cdot\right\Vert _{2}$
and $\psi(x)=\frac{1}{2}\left\Vert x\right\Vert _{2}^{2}$, Algorithm
\ref{alg:algo-md} is totally the same as Algorithm \ref{alg:algo}.

\subsection{Generalized Fundamental Lemmas}

In this section, we present the generalized fundamental lemmas used
in the proof. First, recall the notations used in the main text:
\begin{align*}
\Delta_{t}\coloneqq & \Delta_{t}(x_{*});\quad\pa_{t}\coloneqq\E_{t}\left[\hp F(x_{t})\right]\in\pa F(x_{t});\\
\xi_{t}\coloneqq & g_{t}-\pa_{t};\quad\xi_{t}^{u}\coloneqq g_{t}-\E_{t}\left[g_{t}\right];\quad\xi_{t}^{b}=\E_{t}\left[g_{t}\right]-\pa_{t};\\
d_{t}\coloneqq & \left\Vert x_{t}-x_{*}\right\Vert ;\quad D_{t}\coloneqq\max_{s\in\left[t\right]}d_{s};\quad\mathfrak{D}_{t}\coloneqq D_{t}\lor\alpha;
\end{align*}
where $\F_{t}=\sigma(\hp F(x_{1}),\cdots,\hp F(x_{t}))$ is the natural
filtration. $\E_{t}\left[\cdot\right]$ is used to denote $\E\left[\cdot\mid\F_{t-1}\right]$
for brevity. Now we are able to present the general version of Lemmas
\ref{lem:err-bound}, \ref{lem:basic} and \ref{lem:normalize}, which
play the most important roles in the proof. 

The proof of Lemma \ref{lem:err-bound-general} is by extending the
ideas in \cite{liu2023breaking} to general norms. The modification
appears when bounding the term $\E_{t}\left[\|\xi_{t}^{u}\|_{*}^{2}\right]$.
But the final bound is still in the order of $O(\sigma^{p}M_{t}^{2-p})$.
\begin{lem}
\label{lem:err-bound-general}For any $t\in\left[T\right]$, if $M_{t}\geq2G$,
we have
\begin{align*}
\left\Vert \xi_{t}^{u}\right\Vert _{*} & \le2M_{t};\quad\E_{t}\left[\left\Vert \xi_{t}^{u}\right\Vert _{*}^{2}\right]\le\begin{cases}
10\sigma^{p}M_{t}^{2-p} & \text{if }\left\Vert \cdot\right\Vert =\left\Vert \cdot\right\Vert _{2}\\
40\sigma^{p}M_{t}^{2-p} & \text{o.w.}
\end{cases};\\
\left\Vert \xi_{t}^{b}\right\Vert _{*} & \le2\sigma^{p}M_{t}^{1-p};\quad\left\Vert \xi_{t}^{b}\right\Vert _{*}^{2}\leq10\sigma^{p}M_{t}^{2-p}.
\end{align*}
\end{lem}
\begin{proof}
First, $\|\xi_{t}^{u}\|_{*}\le2M_{t}$ is always true due to 
\[
\left\Vert \xi_{t}^{u}\right\Vert _{*}=\left\Vert g_{t}-\E_{t}\left[g_{t}\right]\right\Vert _{*}\leq\left\Vert g_{t}\right\Vert _{*}+\left\Vert \E_{t}\left[g_{t}\right]\right\Vert _{*}\leq2M_{t}.
\]

Next, let us prove the bound on $\E_{t}\left[\|\xi_{t}^{u}\|_{*}^{2}\right]$.
Note that if $\left\Vert \cdot\right\Vert $ is the general norm,
we have
\begin{align*}
\E_{t}\left[\left\Vert \xi_{t}^{u}\right\Vert _{*}^{2}\right]= & \E_{t}\left[\left\Vert g_{t}-\E_{t}\left[g_{t}\right]\right\Vert _{*}^{2}\right]\leq\E_{t}\left[2\left\Vert g_{t}-\pa_{t}\right\Vert _{*}^{2}+2\left\Vert \pa_{t}-\E_{t}\left[g_{t}\right]\right\Vert _{*}^{2}\right]\\
\leq & \E_{t}\left[2\left\Vert g_{t}-\pa_{t}\right\Vert _{*}^{2}+2\E_{t}\left[\left\Vert \pa_{t}-g_{t}\right\Vert _{*}^{2}\right]\right]=4\E_{t}\left[\left\Vert g_{t}-\pa_{t}\right\Vert _{*}^{2}\right].
\end{align*}
If $\left\Vert \cdot\right\Vert =\left\Vert \cdot\right\Vert _{2}$,
then $\left\Vert \cdot\right\Vert _{*}=\left\Vert \cdot\right\Vert _{2}$.
In this case, we know
\begin{align*}
\E_{t}\left[\left\Vert \xi_{t}^{u}\right\Vert _{2}^{2}\right]= & \E_{t}\left[\left\Vert g_{t}-\E_{t}\left[g_{t}\right]\right\Vert _{2}^{2}\right]\\
= & \E_{t}\left[\left\Vert g_{t}-\pa_{t}\right\Vert _{2}^{2}+2\langle g_{t}-\pa_{t},\pa_{t}-\E_{t}\left[g_{t}\right]\rangle+\left\Vert \pa_{t}-\E_{t}\left[g_{t}\right]\right\Vert _{2}^{2}\right]\\
= & \E_{t}\left[\left\Vert g_{t}-\pa_{t}\right\Vert _{2}^{2}\right]-\left\Vert \pa_{t}-\E_{t}\left[g_{t}\right]\right\Vert _{2}^{2}\leq\E_{t}\left[\left\Vert g_{t}-\pa_{t}\right\Vert _{2}^{2}\right].
\end{align*}
Thus, there is
\[
\E_{t}\left[\left\Vert \xi_{t}^{u}\right\Vert _{*}^{2}\right]\le\begin{cases}
\E_{t}\left[\left\Vert g_{t}-\pa_{t}\right\Vert _{*}^{2}\right] & \text{if }\left\Vert \cdot\right\Vert =\left\Vert \cdot\right\Vert _{2}\\
4\E_{t}\left[\left\Vert g_{t}-\pa_{t}\right\Vert _{*}^{2}\right] & \text{o.w.}
\end{cases}.
\]
So our next goal is to bound $\E_{t}\left[\left\Vert g_{t}-\pa_{t}\right\Vert _{*}^{2}\right]$
by noticing that
\begin{align}
\E_{t}\left[\left\Vert g_{t}-\pa_{t}\right\Vert _{*}^{2}\right]= & \E_{t}\left[\left\Vert g_{t}-\pa_{t}\right\Vert _{*}^{2}\indi_{\left\Vert \hp F(x_{t})\right\Vert _{*}\geq M_{t}}+\left\Vert g_{t}-\pa_{t}\right\Vert _{*}^{2}\indi_{\left\Vert \hp F(x_{t})\right\Vert _{*}<M_{t}}\right]\nonumber \\
= & \E_{t}\left[\left\Vert g_{t}-\pa_{t}\right\Vert _{*}^{2}\indi_{\left\Vert \hp F(x_{t})\right\Vert _{*}\geq M_{t}}+\left\Vert g_{t}-\pa_{t}\right\Vert _{*}^{2-p}\left\Vert \hp F(x_{t})-\pa_{t}\right\Vert _{*}^{p}\indi_{\left\Vert \hp F(x_{t})\right\Vert _{*}<M_{t}}\right]\nonumber \\
\overset{(a)}{\leq} & \E_{t}\left[\frac{9}{4}M_{t}^{2}\indi_{\left\Vert \hp F(x_{t})-\pa_{t}\right\Vert _{*}\geq M_{t}/2}+\left(\frac{9}{4}M_{t}\right)^{2-p}\left\Vert \hp F(x_{t})-\pa_{t}\right\Vert _{*}^{p}\indi_{\left\Vert \hp F(x_{t})\right\Vert _{*}<M_{t}}\right]\nonumber \\
\overset{(b)}{\leq} & \frac{9}{4}M_{t}^{2}\cdot\frac{\sigma^{p}}{\left(M_{t}/2\right)^{p}}+\left(\frac{9}{4}\right)^{2-p}M_{t}^{2-p}\cdot\sigma^{p}\leq10\sigma^{p}M_{t}^{2-p}\label{eq:mid}
\end{align}
where $(a)$ is due to $\left\Vert g_{t}-\pa_{t}\right\Vert _{*}\leq\left\Vert g_{t}\right\Vert _{*}+\left\Vert \pa_{t}\right\Vert _{*}\leq M_{t}+M_{t}/2=3M_{t}/2$;
$(b)$ is by using Markov's inequality to get
\[
\E_{t}\left[\mathds{1}_{\left\Vert \hp F(x_{t})-\pa_{t}\right\Vert _{*}\geq M_{t}/2}\right]=\Pr\left[\left\Vert \hp F(x_{t})-\pa_{t}\right\Vert _{*}^{p}\geq\left(M_{t}/2\right)^{p}\mid\F_{t-1}\right]\leq\frac{\sigma^{p}}{\left(M_{t}/2\right)^{p}}
\]
and by Assumption 5' to obtain
\[
\E_{t}\left[\left\Vert \hp F(x_{t})-\pa_{t}\right\Vert _{*}^{p}\indi_{\left\Vert \hp F(x_{t})\right\Vert _{*}<M_{t}}\right]\leq\E_{t}\left[\left\Vert \hp F(x_{t})-\pa_{t}\right\Vert _{*}^{p}\right]\leq\sigma^{p}.
\]
Hence, we know
\[
\E_{t}\left[\left\Vert \xi_{t}^{u}\right\Vert _{*}^{2}\right]\le\begin{cases}
10\sigma^{p}M_{t}^{2-p} & \text{if }\left\Vert \cdot\right\Vert =\left\Vert \cdot\right\Vert _{2}\\
40\sigma^{p}M_{t}^{2-p} & \text{o.w.}
\end{cases}.
\]

Then, we prove $\left\Vert \xi_{t}^{b}\right\Vert _{*}\le2\sigma^{p}M_{t}^{1-p}$
by
\begin{align*}
\left\Vert \xi_{t}^{b}\right\Vert _{*}= & \left\Vert \E_{t}\left[g_{t}\right]-\pa_{t}\right\Vert _{*}=\left\Vert \E_{t}\left[g_{t}-\hp F(x_{t})\right]\right\Vert _{*}\\
\leq & \E_{t}\left[\left\Vert g_{t}-\hp F(x_{t})\right\Vert _{*}\right]=\E_{t}\left[\left\Vert \frac{M_{t}}{\left\Vert \hp F(x_{t})\right\Vert _{*}}\hp F(x_{t})-\hp F(x_{t})\right\Vert _{*}\indi_{\left\Vert \hp F(x_{t})\right\Vert _{*}\geq M_{t}}\right]\\
= & \E_{t}\left[\left(\left\Vert \hp F(x_{t})\right\Vert _{*}-M_{t}\right)\indi_{\left\Vert \hp F(x_{t})\right\Vert _{*}\geq M_{t}}\right]\overset{(c)}{\leq}\E_{t}\left[\left\Vert \hp F(x_{t})-\pa_{t}\right\Vert _{*}\indi_{\left\Vert \hp F(x_{t})\right\Vert _{*}\geq M_{t}}\right]\\
\overset{(d)}{\leq} & \E_{t}\left[\left\Vert \hp F(x_{t})-\pa_{t}\right\Vert _{*}\indi_{\left\Vert \hp F(x_{t})-\pa_{t}\right\Vert _{*}\geq M_{t}/2}\right]\overset{(e)}{\leq}\E_{t}\left[\left\Vert \hp F(x_{t})-\pa_{t}\right\Vert _{*}^{p}\cdot\left(\frac{2}{M_{t}}\right)^{p-1}\right]\\
\overset{(f)}{\leq} & 2^{p-1}\sigma^{p}M_{t}^{1-p}\leq2\sigma^{p}M_{t}^{1-p}
\end{align*}
where $(c)$ is due to $\left\Vert \hp F(x_{t})\right\Vert _{*}-M_{t}\leq\left\Vert \hp F(x_{t})-\pa_{t}\right\Vert _{*}+\left\Vert \pa_{t}\right\Vert _{*}-M_{t}\leq\left\Vert \hp F(x_{t})-\pa_{t}\right\Vert _{*}$
when $M_{t}\geq2G$ and $\left\Vert \pa_{t}\right\Vert _{*}\leq G$;
$(d)$ is by 
\begin{align*}
M_{t}\leq & \left\Vert \hp F(x_{t})\right\Vert _{*}\leq\left\Vert \hp F(x_{t})-\pa_{t}\right\Vert _{*}+\left\Vert \pa_{t}\right\Vert _{*}\leq\left\Vert \hp F(x_{t})-\pa_{t}\right\Vert _{*}+G\\
\leq & \left\Vert \hp F(x_{t})-\pa_{t}\right\Vert _{*}+M_{t}/2\\
\Rightarrow M_{t}/2\leq & \left\Vert \hp F(x_{t})-\pa_{t}\right\Vert _{*}
\end{align*}
which implies 
\[
\mathds{1}_{\left\Vert \hp F(x_{t})\right\Vert _{*}\geq M_{t}}\leq\mathds{1}_{\left\Vert \hp F(x_{t})-\pa_{t}\right\Vert _{*}\geq M_{t}/2};
\]
$(e)$ is by 
\[
\mathds{1}_{\left\Vert \hp F(x_{t})-\pa_{t}\right\Vert _{*}\geq M_{t}/2}\leq\left(\frac{\left\Vert \hp F(x_{t})-\pa_{t}\right\Vert _{*}}{M_{t}/2}\right)^{p-1}\indi_{\left\Vert \hp F(x_{t})-\pa_{t}\right\Vert _{*}\geq M_{t}/2}\leq\left(\frac{\left\Vert \hp F(x_{t})-\pa_{t}\right\Vert _{*}}{M_{t}/2}\right)^{p-1};
\]
$(f)$ is due to the new Assumption 5', i.e., $\E_{t}\left[\left\Vert \hp F(x_{t})-\pa_{t}\right\Vert _{*}^{p}\right]\leq\sigma^{p}$.

Finally, we show $\left\Vert \xi_{t}^{b}\right\Vert _{*}^{2}\leq10\sigma^{p}M_{t}^{2-p}$.
Note that
\[
\left\Vert \xi_{t}^{b}\right\Vert _{*}^{2}=\left\Vert \E_{t}\left[g_{t}\right]-\pa_{t}\right\Vert _{*}^{2}\leq\E_{t}\left[\left\Vert g_{t}-\pa_{t}\right\Vert _{*}^{2}\right]\leq10\sigma^{p}M_{t}^{2-p}
\]
where the last step is by (\ref{eq:mid}).
\end{proof}

Next, we inroduce Lemma \ref{lem:basic-general}, which will degenerate
to Lemma \ref{lem:basic} when $\|\cdot\|=\|\cdot\|_{2}$ and $\psi(x)=\frac{1}{2}\|x\|_{2}^{2}$. 
\begin{lem}
\label{lem:basic-general}For any $t\in\left[T\right]$, we have
\[
\Delta_{t}+\eta_{t}^{-1}D_{\psi}\left(x_{*},x_{t+1}\right)-\left(\eta_{t}^{-1}-\mu\right)D_{\psi}\left(x_{*},x_{t}\right)\leq\langle\xi_{t},x_{*}-x_{t}\rangle+\eta_{t}\left(2\left\Vert \xi_{t}^{u}\right\Vert _{*}^{2}+2\left\Vert \xi_{t}^{b}\right\Vert _{*}^{2}+G^{2}\right).
\]
\end{lem}
\begin{proof}
We start with the relative $\mu$-stronogly convexity assumption
\begin{align*}
\Delta_{t}\leq & \langle\pa_{t},x_{t}-x_{*}\rangle-\mu D_{\psi}\left(x_{t},x_{*}\right)\\
= & \langle g_{t},x_{t+1}-x_{*}\rangle+\langle g_{t},x_{t}-x_{t+1}\rangle+\langle\xi_{t},x_{*}-x_{t}\rangle-\mu D_{\psi}\left(x_{t},x_{*}\right)\\
\overset{(a)}{\leq} & \left(\eta_{t}^{-1}-\mu\right)D_{\psi}\left(x_{*},x_{t}\right)-\eta_{t}^{-1}D_{\psi}\left(x_{*},x_{t+1}\right)+\langle\xi_{t},x_{*}-x_{t}\rangle\\
 & +\langle g_{t},x_{t}-x_{t+1}\rangle-\eta_{t}^{-1}D_{\psi}\left(x_{t+1},x_{t}\right)\\
\overset{(b)}{\leq} & \left(\eta_{t}^{-1}-\mu\right)D_{\psi}\left(x_{*},x_{t}\right)-\eta_{t}^{-1}D_{\psi}\left(x_{*},x_{t+1}\right)+\langle\xi_{t},x_{*}-x_{t}\rangle\\
 & +\langle g_{t},x_{t}-x_{t+1}\rangle-\frac{\eta_{t}^{-1}}{2}\left\Vert x_{t}-x_{t+1}\right\Vert ^{2}\\
\overset{(c)}{\leq} & \left(\eta_{t}^{-1}-\mu\right)D_{\psi}\left(x_{*},x_{t}\right)-\eta_{t}^{-1}D_{\psi}\left(x_{*},x_{t+1}\right)+\langle\xi_{t},x_{*}-x_{t}\rangle+\frac{\eta_{t}}{2}\left\Vert g_{t}\right\Vert _{*}^{2}\\
\overset{(d)}{\leq} & \left(\eta_{t}^{-1}-\mu\right)D_{\psi}\left(x_{*},x_{t}\right)-\eta_{t}^{-1}D_{\psi}\left(x_{*},x_{t+1}\right)+\langle\xi_{t},x_{*}-x_{t}\rangle+\eta_{t}\left(2\left\Vert \xi_{t}^{u}\right\Vert _{*}^{2}+2\left\Vert \xi_{t}^{b}\right\Vert _{*}^{2}+G^{2}\right)
\end{align*}
where for $(a)$, we use the optimality condition for $x_{t+1}=\argmin_{x\in\dom}\langle g_{t},x-x_{t}\rangle+\frac{1}{\eta_{t}}D_{\psi}(x,x_{t})$
to get for any $x\in\dom$
\begin{align*}
\langle g_{t}+\eta_{t}^{-1}\left(\na\psi(x_{t+1})-\na\psi(x_{t})\right),x_{t+1}-x\rangle\leq & 0\\
\Rightarrow\langle g_{t},x_{t+1}-x\rangle\leq & \eta_{t}^{-1}\langle\na\psi(x_{t})-\na\psi(x_{t+1}),x_{t+1}-x\rangle\\
= & \eta_{t}^{-1}\left(D_{\psi}\left(x,x_{t}\right)-D_{\psi}\left(x,x_{t+1}\right)-D_{\psi}\left(x_{t+1},x_{t}\right)\right).
\end{align*}
$(b)$ is by $D_{\psi}\left(x_{t+1},x_{t}\right)\geq\frac{\left\Vert x_{t}-x_{t+1}\right\Vert ^{2}}{2}$.
In $(c)$, we apply Cauchy--Schwarz inequality to get 
\[
\langle g_{t},x_{t}-x_{t+1}\rangle\leq\frac{\eta_{t}^{-1}}{2}\left\Vert x_{t}-x_{t+1}\right\Vert ^{2}+\frac{\eta_{t}}{2}\left\Vert g_{t}\right\Vert _{*}^{2}.
\]
$(d)$ is by $\left\Vert g_{t}\right\Vert _{*}^{2}\leq2\left\Vert \xi_{t}\right\Vert _{*}^{2}+2\left\Vert \pa_{t}\right\Vert _{*}^{2}\leq2\left\Vert \xi_{t}\right\Vert _{*}^{2}+2G^{2}$
and $\left\Vert \xi_{t}\right\Vert _{*}^{2}\leq2\left\Vert \xi_{t}^{u}\right\Vert _{*}^{2}+2\left\Vert \xi_{t}^{b}\right\Vert _{*}^{2}$.
After rearranging the terms, we finish the proof.
\end{proof}

With the above two lemmas, one can follow almost the same line to
prove the general version of the convergence theorems (except Theorem
\ref{thm:lip-dog-prob}) both in expectation and probability. We leave
this simple extension to the interested reader and finish this section.

\section{A Technical Tool\label{sec:app-tech}}

In this section, we provide a technical tool, Freedman's inequality,
in Lemma \ref{lem:freedman}, the omitted proof of which can be found
in \cite{bennett1962probability,freedman1975tail,dzhaparidze2001bernstein}.
This famous inequality can help us to quantify the concentration phenomenon
of a bounded martingale difference sequence.
\begin{lem}
\label{lem:freedman}(Freedman's inequality) Suppose $X_{t\in\mathbb{N}^{+}}\in\R$
is a martingale difference sequence adapted to the filtration $\F_{t\in\mathbb{N}}$
satisfying $\left|X_{t}\right|\leq R$ almost surely for some constant
$R$. Let $\sigma_{t}^{2}=\E\left[\left|X_{t}\right|{}^{2}\mid\F_{t-1}\right]$,
then for any $a>0$ and $F>0$, there is
\[
\Pr\left[\exists\tau\geq1,\left|\sum_{t=1}^{\tau}X_{t}\right|>a\text{ and }\sum_{t=1}^{\tau}\sigma_{t}^{2}\leq F\right]\leq2\exp\left(-\frac{a^{2}}{2F+2Ra/3}\right).
\]
\end{lem}
Next, we provide a simple corollary of Lemma \ref{lem:freedman},
which is easier to use in the analysis. For example, the high-probability
bound of $\sum_{t=1}^{\tau}\eta_{t}^{2}(\|\xi_{t}^{u}\|^{2}-\E_{t}[\|\xi_{t}^{u}\|^{2}])$
in Lemma \ref{lem:xi-u-lip} is done by Corollary \ref{cor:ez-any-freedman-1}.
\begin{cor}
\label{cor:ez-any-freedman-1}Under the same settings in Lemma \ref{lem:freedman}.
If $\sum_{t=1}^{T}\sigma_{t}^{2}\leq F$ with probability $1$ for
some $T\in\N_{+}$, by choosing $a=\frac{R}{3}\log\frac{4}{\delta}+\sqrt{\left(\frac{1}{3}R\log\frac{4}{\delta}\right)^{2}+2F\log\frac{4}{\delta}}$,
we have
\[
\Pr\left[\forall\tau\in\left[T\right],\left|\sum_{t=1}^{\tau}X_{t}\right|\leq\frac{2R}{3}\log\frac{4}{\delta}+\sqrt{2F\log\frac{4}{\delta}}\right]\geq1-\frac{\delta}{2}.
\]
\end{cor}
For the initial distance adaptive choices, we need the following stronger
version of Corollary \ref{cor:ez-any-freedman-1}.
\begin{cor}
\label{cor:ez-any-freedman-2}Under the same settings in Lemma \ref{lem:freedman}.
If $\sum_{t=1}^{T}\sigma_{t}^{2}\leq F$ with probability $1$ for
any $T\geq1$, by choosing $a=\frac{R}{3}\log\frac{4}{\delta}+\sqrt{\left(\frac{1}{3}R\log\frac{4}{\delta}\right)^{2}+2F\log\frac{4}{\delta}}$,
we have
\[
\Pr\left[\forall\tau\geq1,\left|\sum_{t=1}^{\tau}X_{t}\right|\leq\frac{2R}{3}\log\frac{4}{\delta}+\sqrt{2F\log\frac{4}{\delta}}\right]\geq1-\frac{\delta}{2}.
\]
\end{cor}

\section{Missing Proofs in Section \ref{sec: analysis}\label{sec:app-missing-proofs}}

In this section, we provide the missing proofs in Section \ref{sec: analysis}.

\subsection{Proof of Lemma \ref{lem:basic-lip-prob}}

\begin{proof}
We first invoke Lemma \ref{lem:basic} for $\mu=0$ to get
\[
\Delta_{t}+\frac{\eta_{t}^{-1}}{2}d_{t+1}^{2}-\frac{\eta_{t}^{-1}}{2}d_{t}^{2}\leq\langle\xi_{t},x_{*}-x_{t}\rangle+\eta_{t}\left(2\left\Vert \xi_{t}^{u}\right\Vert ^{2}+2\left\Vert \xi_{t}^{b}\right\Vert ^{2}+G^{2}\right).
\]
Multiplying both sides by $\eta_{t}/\mathfrak{D}_{t}$, we obtain
\begin{align}
\frac{\eta_{t}\Delta_{t}}{\mathfrak{D}_{t}}+\frac{d_{t+1}^{2}-d_{t}^{2}}{2\mathfrak{D}_{t}}\leq & \eta_{t}\left\langle \xi_{t},\frac{x_{*}-x_{t}}{\mathfrak{D}_{t}}\right\rangle +\frac{\eta_{t}^{2}}{\mathfrak{D}_{t}}\left(2\left\Vert \xi_{t}^{u}\right\Vert ^{2}+2\left\Vert \xi_{t}^{b}\right\Vert ^{2}+G^{2}\right)\nonumber \\
\overset{(a)}{\leq} & \eta_{t}\left\langle \xi_{t},\frac{x_{*}-x_{t}}{\mathfrak{D}_{t}}\right\rangle +\frac{\eta_{t}^{2}}{\alpha}\left(2\left\Vert \xi_{t}^{u}\right\Vert ^{2}+2\left\Vert \xi_{t}^{b}\right\Vert ^{2}+G^{2}\right)\nonumber \\
= & \eta_{t}\left\langle \xi_{t},\frac{x_{*}-x_{t}}{\mathfrak{D}_{t}}\right\rangle +\frac{2\eta_{t}^{2}}{\alpha}\left(\left\Vert \xi_{t}^{u}\right\Vert ^{2}-\E_{t}\left[\left\Vert \xi_{t}^{u}\right\Vert ^{2}\right]\right)\nonumber \\
 & +\frac{\eta_{t}^{2}}{\alpha}\left(2\E_{t}\left[\left\Vert \xi_{t}^{u}\right\Vert ^{2}\right]+2\left\Vert \xi_{t}^{b}\right\Vert ^{2}+G^{2}\right),\label{eq:lip-prob-1}
\end{align}
where $(a)$ is by $\mathfrak{D}_{\tau}\geq\alpha$. Next, summing
up (\ref{eq:lip-prob-1}) from $t=1$ to $\tau$ , there is
\begin{align*}
\frac{d_{\tau+1}^{2}}{2\mathfrak{D}_{\tau}}-\frac{d_{1}^{2}}{2\mathfrak{D}_{1}}+\sum_{t=2}^{\tau}\left(\frac{1}{2\mathfrak{D}_{t-1}}-\frac{1}{2\mathfrak{D}_{t}}\right)d_{t}^{2}+\sum_{t=1}^{\tau}\frac{\eta_{t}\Delta_{t}}{\mathfrak{D}_{t}}\leq & \sum_{t=1}^{\tau}\eta_{t}\left\langle \xi_{t},\frac{x_{*}-x_{t}}{\mathfrak{D}_{t}}\right\rangle +\frac{2\eta_{t}^{2}}{\alpha}\left(\left\Vert \xi_{t}^{u}\right\Vert ^{2}-\E_{t}\left[\left\Vert \xi_{t}^{u}\right\Vert ^{2}\right]\right)\\
 & +\sum_{t=1}^{\tau}\frac{\eta_{t}^{2}}{\alpha}\left(2\E_{t}\left[\left\Vert \xi_{t}^{u}\right\Vert ^{2}\right]+2\left\Vert \xi_{t}^{b}\right\Vert ^{2}+G^{2}\right),
\end{align*}
which yields
\begin{align*}
\frac{d_{\tau+1}^{2}}{2\mathfrak{D}_{\tau}}-\frac{d_{1}^{2}}{2\mathfrak{D}_{1}}+\sum_{t=1}^{\tau}\frac{\eta_{t}\Delta_{t}}{\mathfrak{D}_{t}}\leq & \sum_{t=1}^{\tau}\eta_{t}\left\langle \xi_{t},\frac{x_{*}-x_{t}}{\mathfrak{D}_{t}}\right\rangle +\frac{2\eta_{t}^{2}}{\alpha}\left(\left\Vert \xi_{t}^{u}\right\Vert ^{2}-\E_{t}\left[\left\Vert \xi_{t}^{u}\right\Vert ^{2}\right]\right)\\
 & +\sum_{t=1}^{\tau}\frac{\eta_{t}^{2}}{\alpha}\left(2\E_{t}\left[\left\Vert \xi_{t}^{u}\right\Vert ^{2}\right]+2\left\Vert \xi_{t}^{b}\right\Vert ^{2}+G^{2}\right)
\end{align*}
by noticing that $\mathfrak{D}_{t}\geq\mathfrak{D}_{t-1}$. Finally,
by using $\mathfrak{D}_{1}\geq d_{1}$, $\mathfrak{D}_{\tau}\geq\mathfrak{D}_{t}$
for $t\leq\tau$ and rearranging the terms, we know
\begin{align*}
d_{\tau+1}^{2}+\sum_{t=1}^{\tau}2\eta_{t}\Delta_{t}\leq & \mathfrak{D}_{\tau}\left(\sum_{t=1}^{\tau}2\eta_{t}\left\langle \xi_{t},\frac{x_{*}-x_{t}}{\mathfrak{D}_{t}}\right\rangle +\frac{4\eta_{t}^{2}}{\alpha}\left(\left\Vert \xi_{t}^{u}\right\Vert ^{2}-\E_{t}\left[\left\Vert \xi_{t}^{u}\right\Vert ^{2}\right]\right)\right)\\
 & +\mathfrak{D}_{\tau}\left(d_{1}+\sum_{t=1}^{\tau}\frac{2\eta_{t}^{2}}{\alpha}\left(2\E_{t}\left[\left\Vert \xi_{t}^{u}\right\Vert ^{2}\right]+2\left\Vert \xi_{t}^{b}\right\Vert ^{2}+G^{2}\right)\right).
\end{align*}

Now we bound $\E_{t}[\|\xi_{t}^{u}\|^{2}]\le10\sigma^{p}M_{t}^{2-p}$
and $\|\xi_{t}^{b}\|^{2}\le10\sigma^{p}M_{t}^{2-p}$ by Lemma \ref{lem:err-bound}
to get
\begin{align*}
2\E_{t}\left[\left\Vert \xi_{t}^{u}\right\Vert ^{2}\right]+2\left\Vert \xi_{t}^{b}\right\Vert ^{2}\leq & 40\sigma^{p}M_{t}^{2-p}.
\end{align*}
Hence, we have
\begin{align*}
\frac{2\eta_{t}^{2}}{\alpha}\left(2\E_{t}\left[\left\Vert \xi_{t}^{u}\right\Vert ^{2}\right]+2\left\Vert \xi_{t}^{b}\right\Vert ^{2}+G^{2}\right)\leq & \frac{2\eta_{t}^{2}}{\alpha}\left(40\sigma^{p}M_{t}^{2-p}+G^{2}\right)\\
\leq & \frac{80\sigma^{p}(\eta_{t}M_{t})^{2}}{\alpha M_{t}^{p}}+\frac{2\eta_{t}^{2}G^{2}}{\alpha}\\
\overset{(b)}{\leq} & \frac{80\alpha(\sigma/M)^{p}}{t}+\frac{2\alpha}{t}\\
\Rightarrow\sum_{t=1}^{\tau}\frac{2\eta_{t}^{2}}{\alpha}\left(2\E_{t}\left[\left\Vert \xi_{t}^{u}\right\Vert ^{2}\right]+2\left\Vert \xi_{t}^{b}\right\Vert ^{2}+G^{2}\right)\leq & 2\left(1+40(\sigma/M)^{p}\right)\alpha\log(e\tau)=h(\tau)
\end{align*}
where $(b)$ is by $\eta_{t}M_{t}\leq\alpha$ from Lemma \ref{lem:normalize},
$M_{t}\geq Mt^{\frac{1}{p}}$ and $\eta_{t}\leq\frac{\alpha}{G\sqrt{t}}$
from our choices.
\end{proof}

\subsection{Proof of Lemma \ref{lem:lip-prob-concen}}

\begin{proof}
We first note that $Z_{t}\coloneqq\eta_{t}\left\langle \xi_{t}^{u},\frac{x_{*}-x_{t}}{\mathfrak{D}_{t}}\right\rangle \in\F_{t},\forall t\in\left[T\right]$
is a martingale difference sequence. Next, observe that
\begin{align*}
\left|Z_{t}\right|\leq & \eta_{t}\left\Vert \xi_{t}^{u}\right\Vert \frac{d_{t}}{\mathfrak{D}_{t}}\overset{(a)}{\leq}\eta_{t}\left\Vert \xi_{t}^{u}\right\Vert \overset{(b)}{\leq}2\eta_{t}M_{t}\overset{(c)}{\leq}2\alpha;\\
\E_{t}\left[\left(Z_{t}\right)^{2}\right]= & \E_{t}\left[\eta_{t}^{2}\left\langle \xi_{t}^{u},\frac{x_{*}-x_{t}}{\mathfrak{D}_{t}}\right\rangle ^{2}\right]\leq\eta_{t}^{2}\E_{t}\left[\left\Vert \xi_{t}^{u}\right\Vert ^{2}\right]\overset{(d)}{\leq}10\eta_{t}^{2}\sigma^{p}M_{t}^{2-p}\overset{(e)}{\leq}\frac{10(\sigma/M)^{p}\alpha^{2}}{t};
\end{align*}
where $(a)$ is due to $d_{t}\leq\mathfrak{D}_{t}$; $\|\xi_{t}^{u}\|\leq2M_{t}$
in $(b)$ and $\E_{t}[\|\xi_{t}^{u}\|^{2}]\leq10\sigma^{p}M_{t}^{2-p}$
in $(d)$ are both by Lemma \ref{lem:err-bound}; $\eta_{t}M_{t}\leq\alpha$
in $(c)$ is by Lemma \ref{lem:normalize}; $\eta_{t}^{2}\sigma^{p}M_{t}^{2-p}\leq(\sigma/M)^{p}\alpha^{2}/t$
in $(e)$ is by $\eta_{t}M_{t}\leq\alpha$ from Lemma \ref{lem:normalize}
again and $M_{t}\geq Mt^{\frac{1}{p}}$ from our choice.

Now we know
\[
\sum_{t=1}^{T}\E_{t}\left[\left(Z_{t}\right)^{2}\right]\leq\sum_{t=1}^{T}\frac{10(\sigma/M)^{p}\alpha^{2}}{t}\leq10(\sigma/M)^{p}\alpha^{2}\log(eT)
\]
Let $R=2\alpha$, $F=10(\sigma/M)^{p}\alpha^{2}\log(eT)$. By Freedman's
inequality (Corollary \ref{cor:ez-any-freedman-1}), we know with
probability at least $1-\frac{\delta}{2}$, for any $\tau\in\left[T\right]$
\begin{align*}
\left|\sum_{t=1}^{\tau}Z_{t}\right|\leq & \frac{2R}{3}\log\frac{4}{\delta}+\sqrt{2F\log\frac{4}{\delta}}\\
= & \frac{4\alpha}{3}\log\frac{4}{\delta}+\sqrt{20(\sigma/M)^{p}\alpha^{2}\log(eT)\log\frac{4}{\delta}}\\
\leq & 5\left(\log\frac{4}{\delta}+\sqrt{(\sigma/M)^{p}\log(eT)\log\frac{4}{\delta}}\right)\alpha.
\end{align*}
\end{proof}

\subsection{Proof of Lemma \ref{lem:xi-u-lip}}

\begin{proof}
We first note that $\eta_{t}^{2}\left(\left\Vert \xi_{t}^{u}\right\Vert ^{2}-\E_{t}\left[\left\Vert \xi_{t}^{u}\right\Vert ^{2}\right]\right)\in\F_{t}$
is a martingale difference sequence. Next, observe that
\begin{align*}
\eta_{t}^{2}\left|\left\Vert \xi_{t}^{u}\right\Vert ^{2}-\E_{t}\left[\left\Vert \xi_{t}^{u}\right\Vert ^{2}\right]\right|\leq & \eta_{t}^{2}\left\Vert \xi_{t}^{u}\right\Vert ^{2}+\eta_{t}^{2}\E_{t}\left[\left\Vert \xi_{t}^{u}\right\Vert ^{2}\right]\overset{(a)}{\leq}8\eta_{t}^{2}M_{t}^{2}\overset{(b)}{\leq}8\alpha^{2};\\
\E_{t}\left[\eta_{t}^{4}\left(\left\Vert \xi_{t}^{u}\right\Vert ^{2}-\E_{t}\left[\left\Vert \xi_{t}^{u}\right\Vert ^{2}\right]\right)^{2}\right]\leq & \eta_{t}^{4}\E_{t}\left[\left\Vert \xi_{t}^{u}\right\Vert ^{4}\right]\overset{(c)}{\leq}\eta_{t}^{4}\cdot4M_{t}^{2}\cdot10\sigma^{p}M_{t}^{2-p}\overset{(d)}{\leq}\frac{40(\sigma/M)^{p}\alpha^{4}}{t};
\end{align*}
where both $(a)$ and $(c)$ are due to Lemma \ref{lem:err-bound}.
$(b)$ is due to Lemma \ref{lem:normalize}. $(d)$ is by Lemma \ref{lem:normalize}
again and $M_{t}\geq Mt^{\frac{1}{p}}$. Now we know
\[
\sum_{t=1}^{T}\E_{t}\left[\eta_{t}^{4}\left(\left\Vert \xi_{t}^{u}\right\Vert ^{2}-\E_{t}\left[\left\Vert \xi_{t}^{u}\right\Vert ^{2}\right]\right)^{2}\right]\leq\sum_{t=1}^{T}\frac{40(\sigma/M)^{p}\alpha^{4}}{t}\leq40(\sigma/M)^{p}\alpha^{4}\log(eT).
\]
Let $R=8\alpha^{2}$, $F=40(\sigma/G)^{p}\alpha^{4}\log(eT))$. By
Freedman's inequality (Corollary \ref{cor:ez-any-freedman-1}), we
know with probability at least $1-\frac{\delta}{2}$, for any $\tau\in\left[T\right]$,
\begin{align*}
\sum_{t=1}^{\tau}\eta_{t}^{2}\left(\left\Vert \xi_{t}^{u}\right\Vert ^{2}-\E_{t}\left[\left\Vert \xi_{t}^{u}\right\Vert ^{2}\right]\right)\leq & \frac{2R}{3}\log\frac{4}{\delta}+\sqrt{2F\log\frac{4}{\delta}}\\
= & \frac{16\alpha^{2}}{3}\log\frac{4}{\delta}+\sqrt{80(\sigma/M)^{p}\alpha^{4}\log(eT)\log\frac{4}{\delta}}\\
\leq & 9\left(\log\frac{4}{\delta}+\sqrt{(\sigma/M)^{p}\log(eT)\log\frac{4}{\delta}}\right)\alpha^{2}\\
\Rightarrow\sum_{t=1}^{\tau}\frac{\eta_{t}^{2}}{\alpha}\left(\left\Vert \xi_{t}^{u}\right\Vert ^{2}-\E_{t}\left[\left\Vert \xi_{t}^{u}\right\Vert ^{2}\right]\right)\leq & 9\left(\log\frac{4}{\delta}+\sqrt{(\sigma/M)^{p}\log(eT)\log\frac{4}{\delta}}\right)\alpha.
\end{align*}
\end{proof}

\subsection{Proof of Lemma \ref{lem:basic-lip-exp}}

\begin{proof}
We first invoke Lemma \ref{lem:basic} for $\mu=0$ to get
\[
\Delta_{t}+\frac{\eta_{t}^{-1}}{2}d_{t+1}^{2}-\frac{\eta_{t}^{-1}}{2}d_{t}^{2}\leq\langle\xi_{t},x_{*}-x_{t}\rangle+\eta_{t}\left(2\left\Vert \xi_{t}^{u}\right\Vert ^{2}+2\left\Vert \xi_{t}^{b}\right\Vert ^{2}+G^{2}\right).
\]
Multiplying both sides by $\eta_{t}$, we obtain
\begin{align*}
\eta_{t}\Delta_{t}+\frac{d_{t+1}^{2}-d_{t}^{2}}{2}\leq & \eta_{t}\langle\xi_{t},x_{*}-x_{t}\rangle+\eta_{t}^{2}\left(2\left\Vert \xi_{t}^{u}\right\Vert ^{2}+2\left\Vert \xi_{t}^{b}\right\Vert ^{2}+G^{2}\right)\\
= & \eta_{t}\langle\xi_{t},x_{*}-x_{t}\rangle+2\eta_{t}^{2}\left(\left\Vert \xi_{t}^{u}\right\Vert ^{2}-\E_{t}\left[\left\Vert \xi_{t}^{u}\right\Vert ^{2}\right]\right)\\
 & +\eta_{t}^{2}\left(2\E_{t}\left[\left\Vert \xi_{t}^{u}\right\Vert ^{2}\right]+2\left\Vert \xi_{t}^{b}\right\Vert ^{2}+G^{2}\right).
\end{align*}
Next, summing up from $t=1$ to $\tau$ , there is
\begin{align*}
\frac{d_{\tau+1}^{2}-d_{1}^{2}}{2}+\sum_{t=1}^{\tau}\eta_{t}\Delta_{t}\leq & \sum_{t=1}^{\tau}\eta_{t}\langle\xi_{t},x_{*}-x_{t}\rangle+2\eta_{t}^{2}\left(\left\Vert \xi_{t}^{u}\right\Vert ^{2}-\E_{t}\left[\left\Vert \xi_{t}^{u}\right\Vert ^{2}\right]\right)\\
 & +\sum_{t=1}^{\tau}\eta_{t}^{2}\left(2\E_{t}\left[\left\Vert \xi_{t}^{u}\right\Vert ^{2}\right]+2\left\Vert \xi_{t}^{b}\right\Vert ^{2}+G^{2}\right).
\end{align*}
From the proof of Lemma \ref{lem:basic-lip-prob}, we have
\[
\sum_{t=1}^{\tau}\eta_{t}^{2}\left(2\E_{t}\left[\left\Vert \xi_{t}^{u}\right\Vert ^{2}\right]+2\left\Vert \xi_{t}^{b}\right\Vert ^{2}+G^{2}\right)\leq\frac{\alpha h(\tau)}{2}.
\]

Hence, there is
\[
\frac{d_{\tau+1}^{2}-d_{1}^{2}}{2}+\sum_{t=1}^{\tau}\eta_{t}\Delta_{t}\leq\frac{\alpha h(\tau)}{2}+\sum_{t=1}^{\tau}\eta_{t}\langle\xi_{t},x_{*}-x_{t}\rangle+2\eta_{t}^{2}\left(\left\Vert \xi_{t}^{u}\right\Vert ^{2}-\E_{t}\left[\left\Vert \xi_{t}^{u}\right\Vert ^{2}\right]\right).
\]
Taking expectations on both sides to obtain
\begin{align*}
 & \frac{\E\left[d_{\tau+1}^{2}\right]-\E\left[d_{1}^{2}\right]}{2}+\sum_{t=1}^{\tau}\eta_{t}\E\left[\Delta_{t}\right]\\
\leq & \frac{\alpha h(\tau)}{2}+\sum_{t=1}^{\tau}\eta_{t}\E\left[\langle\xi_{t},x_{*}-x_{t}\rangle\right]+2\eta_{t}^{2}\E\left[\left\Vert \xi_{t}^{u}\right\Vert ^{2}-\E_{t}\left[\left\Vert \xi_{t}^{u}\right\Vert ^{2}\right]\right]\\
\overset{(a)}{=} & \frac{\alpha h(\tau)}{2}+\sum_{t=1}^{\tau}\eta_{t}\E\left[\langle\xi_{t}^{b},x_{*}-x_{t}\rangle\right]\overset{(b)}{\leq}\frac{\alpha h(\tau)}{2}+\sum_{t=1}^{\tau}\eta_{t}\E\left[\left\Vert \xi_{t}^{b}\right\Vert \left\Vert x_{*}-x_{t}\right\Vert \right]\\
\overset{(c)}{\leq} & \frac{\alpha h(\tau)}{2}+\sum_{t=1}^{\tau}2\eta_{t}\sigma^{p}M_{t}^{1-p}\E\left[\left\Vert x_{*}-x_{t}\right\Vert \right]\frac{\alpha h(\tau)}{2}+\sum_{t=1}^{\tau}2\eta_{t}\sigma^{p}M_{t}^{1-p}\sqrt{\E\left[\left\Vert x_{*}-x_{t}\right\Vert ^{2}\right]}\\
= & \frac{\alpha h(\tau)}{2}+\sum_{t=1}^{\tau}2\eta_{t}\sigma^{p}M_{t}^{1-p}\sqrt{\E\left[d_{t}^{2}\right]}
\end{align*}
where we use 
\[
\E\left[\langle\xi_{t},x_{*}-x_{t}\rangle\right]=\E\left[\E_{t}\left[\langle\xi_{t},x_{*}-x_{t}\rangle\right]\right]=\E\left[\langle\E_{t}\left[\xi_{t}\right],x_{*}-x_{t}\rangle\right]=\E\left[\langle\xi_{t}^{b},x_{*}-x_{t}\rangle\right]
\]
and $\E[\|\xi_{t}^{u}\|^{2}]=\E[\E_{t}[\|\xi_{t}^{u}\|^{2}]]$ in
$(a)$; $\langle\xi_{t}^{b},x_{*}-x_{t}\rangle\leq\|\xi_{t}^{b}\|\|x_{*}-x_{t}\|$
in $(b)$ is due to Cauchy--Schwarz inequality; The bound of $\|\xi_{t}^{b}\|\leq2\sigma^{p}M_{t}^{1-p}$
in $(c)$ is from Lemma \ref{lem:err-bound}. $(d)$ is because of
$\E\left[X\right]\leq\sqrt{\E\left[X^{2}\right]}$. Fianlly, we use
\[
2\eta_{t}\sigma^{p}M_{t}^{1-p}=\frac{2\sigma^{p}\eta_{t}M_{t}}{M_{t}^{p}}\leq\frac{2(\sigma/M)^{p}\alpha}{t}
\]
 to finish the proof.
\end{proof}

\section{Additional Theoretical Analysis for Initial Distance Adaptive Choices
When $\mu=0$\label{sec:app-dog}}

First we prove Fact \ref{fact:dog-order}

\begin{proof}
When $w_{t}=1+\log^{2}(t)$, note that
\[
\sum_{t=1}^{T}\frac{1}{tw_{t}}\leq\frac{1}{w_{1}}+\int_{1}^{\infty}\frac{1}{t(1+\log^{2}(t))}\mathrm{d}t=1+\arctan(\log(t))\big\vert_{1}^{\infty}=1+\frac{\pi}{2}=W.
\]

When $w_{t}=\left[y^{(n+1)}(t)\right]^{1+\varepsilon}\prod_{i=1}^{n}y^{(i)}(t)$
where $y(t)=1+\log(t)$, $y^{(n)}(t)=y(y^{(n-1)}(t))$ is the $n$-times
composition with itself for some non-negative integer $n$, and $\varepsilon>0$
can be chosen arbitrarily, note that
\[
\sum_{t=1}^{T}\frac{1}{tw_{t}}\leq\frac{1}{w_{1}}+\int_{1}^{\infty}\frac{1}{tw_{t}}\mathrm{d}t=1+\left(-\frac{1}{\varepsilon}\left[y^{(n+1)}(t)\right]^{-\varepsilon}\right)\big\vert_{1}^{\infty}=1+\frac{1}{\varepsilon}=W.
\]
\end{proof}

Now let's start the proof of Theorem \ref{thm:lip-dog-prob}. We begin
with the following basic inequality.
\begin{lem}
\label{lem:basic-dog}When $\mu=0$, under the choices described in
Theorem \ref{thm:lip-dog-prob}, for any $\tau\geq1$, we have
\begin{align*}
d_{\tau+1}^{2}-d_{1}^{2}+\sum_{t=1}^{\tau}2r_{t}\gamma_{t}\Delta_{t}\leq & 4D_{\tau}r_{\tau}\max_{t\in\left[\tau\right]}\left|\sum_{s=1}^{t}\gamma_{s}\left\langle \xi_{s}^{u},\frac{x_{*}-x_{s}}{D_{s}}\right\rangle \right|+2r_{\tau}^{2}\sum_{t=1}^{\tau}\gamma_{t}^{2}\left(\left\Vert \xi_{t}^{u}\right\Vert ^{2}-\E_{t}\left[\left\Vert \xi_{t}^{u}\right\Vert ^{2}\right]\right)\\
 & +4D_{\tau}r_{\tau}\alpha_{2}\log(4/\delta)W+r_{\tau}^{2}\left(80\alpha_{2}^{2}\log(4/\delta)+2\alpha_{1}^{2}\right)W.
\end{align*}
\end{lem}
\begin{proof}
For any fixed $\tau\geq1$, we first invoke Lemma \ref{lem:basic}
for $\mu=0$ to get
\begin{align*}
\Delta_{t}+\frac{\eta_{t}^{-1}}{2}d_{t+1}^{2}-\frac{\eta_{t}^{-1}}{2}d_{t}^{2} & \leq\langle\xi_{t},x_{*}-x_{t}\rangle+\eta_{t}\left(2\left\Vert \xi_{t}^{u}\right\Vert ^{2}+2\left\Vert \xi_{t}^{b}\right\Vert ^{2}+G^{2}\right).
\end{align*}
Multiplying both sides by $2\eta_{t}=2r_{t}\gamma_{t}$, we obtain
\begin{align*}
2r_{t}\gamma_{t}\Delta_{t}+d_{t+1}^{2}-d_{t}^{2}\leq & 2\eta_{t}\left\langle \xi_{t},x_{*}-x_{t}\right\rangle +2\eta_{t}^{2}\left(2\left\Vert \xi_{t}^{u}\right\Vert ^{2}+2\left\Vert \xi_{t}^{b}\right\Vert ^{2}+G^{2}\right)\\
= & 2r_{t}\gamma_{t}\left\langle \xi_{t}^{u},x_{*}-x_{t}\right\rangle +2r_{t}\gamma_{t}\left\langle \xi_{t}^{b},x_{*}-x_{t}\right\rangle \\
 & +2r_{t}^{2}\gamma_{t}^{2}\left(2\left\Vert \xi_{t}^{u}\right\Vert ^{2}+2\left\Vert \xi_{t}^{b}\right\Vert ^{2}+G^{2}\right)\\
\overset{(a)}{\leq} & 2D_{t}r_{t}\gamma_{t}\left\langle \xi_{t},\frac{x_{*}-x_{t}}{D_{t}}\right\rangle +2D_{\tau}r_{\tau}\gamma_{t}\left\Vert \xi_{t}^{b}\right\Vert \\
 & +2r_{\tau}^{2}\gamma_{t}^{2}\left(2\left\Vert \xi_{t}^{u}\right\Vert ^{2}+2\left\Vert \xi_{t}^{b}\right\Vert ^{2}+G^{2}\right)\\
= & 2D_{t}r_{t}\gamma_{t}\left\langle \xi_{t},\frac{x_{*}-x_{t}}{D_{t}}\right\rangle +2D_{\tau}r_{\tau}\gamma_{t}\left\Vert \xi_{t}^{b}\right\Vert \\
 & +2r_{\tau}^{2}\gamma_{t}^{2}\left(\left\Vert \xi_{t}^{u}\right\Vert ^{2}-\E_{t}\left[\left\Vert \xi_{t}^{u}\right\Vert ^{2}\right]\right)+2r_{\tau}^{2}\gamma_{t}^{2}\left(2\E_{t}\left[\left\Vert \xi_{t}^{u}\right\Vert ^{2}\right]+2\left\Vert \xi_{t}^{b}\right\Vert ^{2}+G^{2}\right)
\end{align*}
where $(a)$ is by $||x_{t}-x_{*}\|\leq D_{t}\leq D_{\tau}$ and $r_{t}\leq r_{\tau}$
for $t\leq\tau$. Now summing up from $t=1$ to $\tau$ to obtain
\begin{align}
d_{\tau+1}^{2}-d_{1}^{2}+\sum_{t=1}^{\tau}2\mathfrak{r}_{t}\gamma_{t}\Delta_{t}\leq & \sum_{t=1}^{\tau}2D_{t}r_{t}\gamma_{t}\left\langle \xi_{t},\frac{x_{*}-x_{t}}{D_{t}}\right\rangle +D_{\tau}r_{\tau}\sum_{t=1}^{\tau}2\gamma_{t}\left\Vert \xi_{t}^{b}\right\Vert \nonumber \\
 & +r_{\tau}^{2}\sum_{t=1}^{\tau}2\gamma_{t}^{2}\left(\left\Vert \xi_{t}^{u}\right\Vert ^{2}-\E_{t}\left[\left\Vert \xi_{t}^{u}\right\Vert ^{2}\right]\right)\nonumber \\
 & +r_{\tau}^{2}\sum_{t=1}^{\tau}2\gamma_{t}^{2}\left(2\E_{t}\left[\left\Vert \xi_{t}^{u}\right\Vert ^{2}\right]+2\left\Vert \xi_{t}^{b}\right\Vert ^{2}+G^{2}\right).\label{eq:dog}
\end{align}

By Lemma 5 in \cite{ivgi2023dog}, we have
\begin{equation}
\sum_{t=1}^{\tau}2D_{t}r_{t}\gamma_{t}\left\langle \xi_{t},\frac{x_{*}-x_{t}}{D_{t}}\right\rangle \leq4D_{\tau}r_{\tau}\max_{t\in\left[\tau\right]}\left|\sum_{s=1}^{t}\gamma_{s}\left\langle \xi_{s}^{u},\frac{x_{*}-x_{s}}{D_{s}}\right\rangle \right|.\label{eq:dog-1}
\end{equation}

Then we use $\|\xi_{t}^{b}\|\leq2\sigma^{p}M_{t}^{1-p}$ from Lemma
\ref{lem:err-bound} to obtain
\begin{align}
\sum_{t=1}^{\tau}2\gamma_{t}\left\Vert \xi_{t}^{b}\right\Vert \leq & \sum_{t=1}^{\tau}4\cdot\gamma_{t}M_{t}\cdot\frac{\sigma^{p}}{M_{t}^{p}}\leq\sum_{t=1}^{\tau}4\cdot\alpha_{2}\cdot\frac{\log(4/\delta)}{tw_{t}}\leq4\alpha_{2}\log(4/\delta)W.\label{eq:dog-2}
\end{align}

Next we bound $\E_{t}[\|\xi_{t}^{u}\|^{2}]\le10\sigma^{p}M_{t}^{2-p}$
and $\|\xi_{t}^{b}\|^{2}\le10\sigma^{p}M_{t}^{2-p}$ by Lemma \ref{lem:err-bound}
to get
\begin{align*}
2\E_{t}\left[\left\Vert \xi_{t}^{u}\right\Vert ^{2}\right]+2\left\Vert \xi_{t}^{b}\right\Vert ^{2}\leq & 40\sigma^{p}M_{t}^{2-p}.
\end{align*}
Hence, there is
\begin{align}
2\gamma_{t}^{2}\left(2\E_{t}\left[\left\Vert \xi_{t}^{u}\right\Vert ^{2}\right]+2\left\Vert \xi_{t}^{b}\right\Vert ^{2}+G^{2}\right)\leq & 2\gamma_{t}^{2}\left(40\sigma^{p}M_{t}^{2-p}+G^{2}\right)\nonumber \\
\leq & 80(\gamma_{t}M_{t})^{2}\cdot\frac{\sigma^{p}}{M_{t}^{p}}+2\gamma_{t}^{2}G^{2}\nonumber \\
\overset{(b)}{\leq} & \frac{80\alpha_{2}^{2}\log(4/\delta)+2\alpha_{1}^{2}}{tw_{t}}\nonumber \\
\Rightarrow\sum_{t=1}^{\tau}2\gamma_{t}^{2}\left(2\E_{t}\left[\left\Vert \xi_{t}^{u}\right\Vert ^{2}\right]+2\left\Vert \xi_{t}^{b}\right\Vert ^{2}+G^{2}\right)\leq & \sum_{t=1}^{\tau}\frac{80\alpha_{2}^{2}\log(4/\delta)+2\alpha_{1}^{2}}{tw_{t}}\nonumber \\
= & \left(80\alpha_{2}^{2}\log(4/\delta)+2\alpha_{1}^{2}\right)W\label{eq:dog-3}
\end{align}
where $(b)$ is by $\gamma_{t}\leq\frac{\alpha_{2}}{M_{t}}$, $M_{t}\geq\sigma(tw_{t}/\log(4/\delta)){}^{\frac{1}{p}}$
and $\gamma_{t}\leq\frac{\alpha_{1}}{G\sqrt{tw_{t}}}$ from our choices.

Finally, plugging (\ref{eq:dog-1}), (\ref{eq:dog-2}) and (\ref{eq:dog-3})
into (\ref{eq:dog}), we obtain
\begin{align*}
d_{\tau+1}^{2}-d_{1}^{2}+\sum_{t=1}^{\tau}2\mathfrak{r}_{t}\gamma_{t}\Delta_{t}\leq & 4D_{\tau}r_{\tau}\max_{t\in\left[\tau\right]}\left|\sum_{s=1}^{t}\gamma_{s}\left\langle \xi_{s}^{u},\frac{x_{*}-x_{s}}{D_{s}}\right\rangle \right|+2r_{\tau}^{2}\sum_{t=1}^{\tau}\gamma_{t}^{2}\left(\left\Vert \xi_{t}^{u}\right\Vert ^{2}-\E_{t}\left[\left\Vert \xi_{t}^{u}\right\Vert ^{2}\right]\right)\\
 & +4D_{\tau}r_{\tau}\alpha_{2}\log(4/\delta)W+r_{\tau}^{2}\left(80\alpha_{2}^{2}\log(4/\delta)+2\alpha_{1}^{2}\right)W.
\end{align*}
\end{proof}

Next, as before, we bound the martingale difference sequences$|\sum_{s=1}^{t}\gamma_{s}\langle\xi_{s}^{u},\frac{x_{*}-x_{s}}{D_{s}}\rangle|$
and $\sum_{t=1}^{\tau}\gamma_{t}^{2}(\|\xi_{t}^{u}\|^{2}-\E_{t}[\|\xi_{t}^{u}\|^{2}])$
respectively.
\begin{lem}
\label{lem:dog-inner}When $\mu=0$, under the choices described in
Theorem \ref{thm:lip-dog-prob}, we have with probability at least
$1-\frac{\delta}{2}$, for any $\tau\geq1$,
\[
\left|\sum_{t=1}^{\tau}\gamma_{t}\left\langle \xi_{t}^{u},\frac{x_{*}-x_{t}}{D_{t}}\right\rangle \right|\leq\alpha_{2}\left(\frac{4}{3}+2\sqrt{5W}\right)\log\frac{4}{\delta}.
\]
\end{lem}
\begin{proof}
Note that $Z_{t}\coloneqq\gamma_{t}\left\langle \xi_{t}^{u},\frac{x_{*}-x_{t}}{D_{t}}\right\rangle \in\F_{t}$
is a martingale difference sequence. Besides, there is
\begin{align*}
\left|Z_{t}\right| & \leq\gamma_{t}\left\Vert \xi_{t}^{u}\right\Vert \leq\gamma_{t}\cdot2M_{t}\le2\alpha_{2};\\
\E_{t}\left[\left(Z_{t}\right)^{2}\right] & \leq\gamma_{t}^{2}\E_{t}\left[\left\Vert \xi_{t}^{u}\right\Vert ^{2}\right]\leq\gamma_{t}^{2}\cdot10\sigma^{p}M_{t}^{2-p}\leq\frac{10\alpha_{2}^{2}\log(4/\delta)}{tw_{t}}.
\end{align*}
Note that for any $T\geq1$
\[
\sum_{t=1}^{T}\E_{t}\left[\left(Z_{t}\right)^{2}\right]\leq10\alpha_{2}^{2}\log(4/\delta)\sum_{t=1}^{T}\frac{1}{tw_{t}}\leq10\alpha_{2}^{2}\log(4/\delta)W.
\]
Now let $R=2\alpha_{2}$ and $F=10\alpha_{2}^{2}\log(4/\delta)W$,
by Freedman's inequality (Corollary \ref{cor:ez-any-freedman-2}),
we have with probability at least $1-\frac{\delta}{2}$, for any $\tau\geq1$
\[
\left|\sum_{t=1}^{\tau}\gamma_{t}\left\langle \xi_{t}^{u},\frac{x_{*}-x_{t}}{D_{t}}\right\rangle \right|\leq\frac{2R}{3}\log\frac{4}{\delta}+\sqrt{2F\log\frac{4}{\delta}}\leq\alpha_{2}\left(\frac{4}{3}+2\sqrt{5W}\right)\log\frac{4}{\delta}.
\]
\end{proof}

\begin{lem}
\label{lem:dog-xi-u}When $\mu=0$, under the choices described in
Theorem \ref{thm:lip-dog-prob}, we have with probability at least
$1-\frac{\delta}{2}$, for any $\tau\geq1$,
\[
\sum_{t=1}^{\tau}\gamma_{t}^{2}\left(\left\Vert \xi_{t}^{u}\right\Vert ^{2}-\E_{t}\left[\left\Vert \xi_{t}^{u}\right\Vert ^{2}\right]\right)\leq\alpha_{2}^{2}\left(\frac{16}{3}+4\sqrt{5W}\right)\log\frac{4}{\delta}.
\]
\end{lem}
\begin{proof}
Note that $Z_{t}\coloneqq\gamma_{t}^{2}\left(\left\Vert \xi_{t}^{u}\right\Vert ^{2}-\E_{t}\left[\left\Vert \xi_{t}^{u}\right\Vert ^{2}\right]\right)\in\F_{t}$
is a martingale difference sequence. Besides, there is
\begin{align*}
\left|Z_{t}\right| & \leq\gamma_{t}^{2}\left(\left\Vert \xi_{t}^{u}\right\Vert ^{2}+\E_{t}\left[\left\Vert \xi_{t}^{u}\right\Vert ^{2}\right]\right)\leq\gamma_{t}^{2}\cdot8M_{t}^{2}\le8\alpha_{2}^{2};\\
\E_{t}\left[\left(Z_{t}\right)^{2}\right] & \leq\gamma_{t}^{4}\E_{t}\left[\left\Vert \xi_{t}^{u}\right\Vert ^{4}\right]\leq\gamma_{t}^{4}\cdot4M_{t}^{2}\cdot10\sigma^{p}M_{t}^{2-p}\leq\frac{40\alpha_{2}^{4}\log(4/\delta)}{tw_{t}}
\end{align*}
Note that for any $T\geq1$
\[
\sum_{t=1}^{T}\E_{t}\left[\left(Z_{t}\right)^{2}\right]\leq40\alpha_{2}^{2}\log(4/\delta)\sum_{t=1}^{T}\frac{1}{tw_{t}}\leq40\alpha_{2}^{4}\log(4/\delta)W.
\]
Now let $R=8\alpha_{2}^{2}$ and $F=40\alpha_{2}^{4}\log(4/\delta)W$,
by Freedman's inequality (Corollary \ref{cor:ez-any-freedman-2}),
we have with probability at least $1-\frac{\delta}{2}$, for any $\tau\geq1$
\[
\sum_{t=1}^{\tau}\gamma_{t}^{2}\left(\left\Vert \xi_{t}^{u}\right\Vert ^{2}-\E_{t}\left[\left\Vert \xi_{t}^{u}\right\Vert ^{2}\right]\right)\leq\frac{2R}{3}\log\frac{4}{\delta}+\sqrt{2F\log\frac{4}{\delta}}\leq\alpha_{2}^{2}\left(\frac{16}{3}+4\sqrt{5W}\right)\log\frac{4}{\delta}.
\]
\end{proof}

The next inequality is inspired by Lemma 3 in \cite{ivgi2023dog}.
Our result is slightly tighter than the bound given in \cite{ivgi2023dog}.
\begin{lem}
\label{lem:dog-sequnce}Suppose $y_{t\geq1}>0$ is a non-decreasing
sequence, then for any $T\geq1$, there is
\[
\max_{t\in\left[T\right]}\sum_{s=1}^{t}\frac{y_{s}}{y_{t+1}}\geq\frac{T}{\left(\frac{y_{T+1}}{y_{1}}\right)^{\frac{1}{T}}\left(1+\log\frac{y_{T+1}}{y_{1}}\right)}.
\]
\end{lem}
\begin{proof}
Let $Y_{t}=\sum_{s=1}^{t}\frac{y_{s}}{y_{t+1}}$ where $Y_{0}=0$,
then we have
\begin{align*}
y_{t+1}Y_{t}-y_{t}Y_{t-1}= & y_{t}\\
\Rightarrow Y_{t}-\frac{y_{t}}{y_{t+1}}Y_{t-1}= & \frac{y_{t}}{y_{t+1}}.
\end{align*}
Summing up from $t=1$ to $T$ to get
\begin{align*}
Y_{T}+\sum_{t=1}^{T-1}\left(1-\frac{y_{t+1}}{y_{t+2}}\right)Y_{t}= & \sum_{t=1}^{T}\frac{y_{t}}{y_{t+1}}\overset{(a)}{\geq}T\left(\frac{y_{1}}{y_{T+1}}\right)^{\frac{1}{T}}\\
\Rightarrow\left(\max_{t\in\left[T\right]}Y_{t}\right)\left(1+\sum_{t=1}^{T-1}1-\frac{y_{t+1}}{y_{t+2}}\right)\geq & T\left(\frac{y_{1}}{y_{T+1}}\right)^{\frac{1}{T}}.
\end{align*}
Where $(a)$ is due to AM-GM inequality. Now by using $1-\frac{1}{x}\leq\log x$,
we obtain
\[
\sum_{t=1}^{T-1}1-\frac{y_{t+1}}{y_{t+2}}\leq\sum_{t=1}^{T-1}\log\frac{y_{t+2}}{y_{t+1}}=\log\frac{y_{T+1}}{y_{2}}\leq\log\frac{y_{T+1}}{y_{1}}.
\]
Finally, we conclude
\begin{align*}
\left(\max_{t\in\left[T\right]}Y_{t}\right)\left(1+\log\frac{y_{T+1}}{y_{1}}\right)\geq & T\left(\frac{y_{1}}{y_{T+1}}\right)^{\frac{1}{T}}\\
\Rightarrow\max_{t\in\left[T\right]}Y_{t}\geq & \frac{T}{\left(\frac{y_{T+1}}{y_{1}}\right)^{\frac{1}{T}}\left(1+\log\frac{y_{T+1}}{y_{1}}\right)}.
\end{align*}
\end{proof}

Now we are ready to prove Theorem \ref{thm:lip-dog-prob}.

\begin{proof}[Proof of Theorem \ref{thm:lip-dog-prob}]
By Applying Lemma \ref{lem:basic-dog}, \ref{lem:dog-inner} and
\ref{lem:dog-xi-u}, we have with probability at least $1-\delta$,
for any $\tau\geq1$
\begin{align*}
d_{\tau+1}^{2}-d_{1}^{2}+\sum_{t=1}^{\tau}2r_{t}\gamma_{t}\Delta_{t}\leq & D_{\tau}r_{\tau}\alpha_{2}\left(\frac{16}{3}+8\sqrt{5W}\right)\log\frac{4}{\delta}+r_{\tau}^{2}\alpha_{2}^{2}\left(\frac{32}{3}+8\sqrt{5W}\right)\log\frac{4}{\delta}\\
 & +4D_{\tau}r_{\tau}\alpha_{2}\log(4/\delta)W+r_{\tau}^{2}\left(80\alpha_{2}^{2}\log(4/\delta)+2\alpha_{1}^{2}\right)W\\
= & D_{\tau}r_{\tau}\alpha_{2}\left(\frac{16}{3}+8\sqrt{5W}+4W\right)\log\frac{4}{\delta}+r_{\tau}^{2}\left[2\alpha_{1}^{2}W+\alpha_{2}^{2}\left(\frac{32}{3}+8\sqrt{5W}+80W\right)\log\frac{4}{\delta}\right].
\end{align*}
Note that we choose
\begin{align*}
\alpha_{1} & =\frac{1}{\sqrt{32W}}\Rightarrow2\alpha_{1}^{2}W=\frac{1}{16};\\
\alpha_{2} & =\frac{1}{8\left(\frac{16}{3}+8\sqrt{5W}+4W\right)\log\frac{4}{\delta}}\land\frac{1}{\sqrt{16\left(\frac{32}{3}+8\sqrt{5W}+80W\right)\log\frac{4}{\delta}}}\\
 & \Rightarrow\begin{cases}
\alpha_{2}\left(\frac{16}{3}+8\sqrt{5W}+4W\right)\log\frac{4}{\delta}\leq\frac{1}{8}\\
\alpha_{2}^{2}\left(\frac{32}{3}+8\sqrt{5W}+80W\right)\log\frac{4}{\delta}\leq\frac{1}{16}
\end{cases}.
\end{align*}
Hence, there is
\begin{equation}
d_{\tau+1}^{2}-d_{1}^{2}+\sum_{t=1}^{\tau}2\mathfrak{r}_{t}\gamma_{t}\Delta_{t}\leq\frac{D_{\tau}r_{\tau}+r_{\tau}^{2}}{8}.\label{eq:dog-final}
\end{equation}
Recall that 
\begin{align*}
r_{t}= & \left(\max_{s\in\left[t\right]}\|x_{1}-x_{s}\|\right)\lor r\leq\left(\max_{s\in\left[t\right]}\|x_{s}-x_{*}\|+d_{1}\right)\lor r=\left(D_{t}+d_{1}\right)\lor r.
\end{align*}
Thus
\begin{align*}
\frac{D_{\tau}r_{\tau}+r_{\tau}^{2}}{8}\leq & \frac{D_{\tau}\left[\left(D_{\tau}+d_{1}\right)\lor r\right]+\left[\left(D_{\tau}+d_{1}\right)\lor r\right]^{2}}{8}\\
\leq & \frac{\left[\left(D_{\tau}+d_{1}\right)\lor r\right]^{2}}{4}\leq\frac{D_{\tau}^{2}+d_{1}^{2}+r^{2}/2}{2}
\end{align*}
which implies
\begin{align*}
d_{\tau+1}^{2}-d_{1}^{2}+\sum_{t=1}^{\tau}2r_{t}\gamma_{t}\Delta_{t}\leq & \frac{D_{\tau}^{2}+d_{1}^{2}+r^{2}/2}{2}\\
\Rightarrow d_{\tau+1}^{2}+\sum_{t=1}^{\tau}2r_{t}\gamma_{t}\Delta_{t}\leq & \frac{D_{\tau}^{2}+3d_{1}^{2}+r^{2}/2}{2}.
\end{align*}

Noticing that $\sum_{t=1}^{\tau}2\mathfrak{r}_{t}\gamma_{t}\Delta_{t}\geq0,\forall\tau\geq1$,
then by a simple induction, we have for any $\tau\geq1$
\[
d_{\tau+1}^{2}\leq3d_{1}^{2}+r^{2}/2,
\]
which immediately implies $D_{\tau}^{2}\leq3d_{1}^{2}+r^{2}/2$. As
a consequence, there are
\[
d_{\tau}=(d_{1}+r),D_{\tau}=O(d_{1}+r),r_{\tau}=O(d_{1}+r),\forall\tau\geq1.
\]
Now we employ (\ref{eq:dog-final}) again to get for any $\tau\geq1$
\begin{align*}
\sum_{t=1}^{\tau}2r_{t}\gamma_{t}\Delta_{t}\leq & \frac{D_{\tau}r_{\tau}+r_{\tau}^{2}}{8}+d_{1}^{2}-d_{\tau+1}^{2}\\
= & \frac{D_{\tau}r_{\tau}+r_{\tau}^{2}}{8}+\left(d_{1}-d_{\tau+1}\right)\left(d_{1}+d_{\tau+1}\right)\\
\leq & \frac{D_{\tau}r_{\tau}+r_{\tau}^{2}}{8}+\left\Vert x_{1}-x_{\tau+1}\right\Vert \left(d_{1}+d_{\tau+1}\right)\\
\leq & \left(\frac{D_{\tau}+r_{\tau}}{8}+d_{1}+d_{\tau+1}\right)r_{\tau+1}\\
\leq & \left(\frac{D_{\tau}+\left(D_{\tau}+d_{1}\right)\lor r}{8}+d_{1}+d_{\tau+1}\right)r_{\tau+1}\\
= & O(d_{1}+r)r_{\tau+1}.
\end{align*}
Recall that
\[
\gamma_{t}=\frac{\alpha_{1}}{G\sqrt{tw_{t}}}\land\frac{\alpha_{2}}{M_{t}},M_{t}=2G\lor\sigma(tw_{t}/\log(4/\delta)){}^{\frac{1}{p}},w_{t}\text{ is non-decreasing}.
\]
So we have
\[
2\gamma_{\tau}\sum_{t=1}^{\tau}r_{t}\Delta_{t}\leq\sum_{t=1}^{\tau}2r_{t}\gamma_{t}\Delta_{t}\leq O(d_{1}+r)r_{\tau+1}\Rightarrow F(\bar{x}_{\tau})-F(x_{*})\leq\frac{O(d_{1}+r)}{\gamma_{\tau}\sum_{t=1}^{\tau}\frac{r_{t}}{r_{\tau+1}}}
\]
where $\bar{x}_{\tau}=\frac{\sum_{t=1}^{\tau}r_{t}x_{t}}{\sum_{t=1}^{\tau}r_{t}}$.

Now we have with probability at least $1-\delta$, for any $T\geq1$
(relabeling $\tau$ by $T$)
\[
F(\bar{x}_{T})-F(x_{*})\leq\frac{O(d_{1}+r)}{\gamma_{T}\sum_{t=1}^{T}\frac{r_{t}}{r_{T+1}}}.
\]
Invoking Lemma \ref{lem:dog-sequnce} for $r_{t}$ to get
\[
\sum_{t=1}^{I(T)}\frac{r_{t}}{r_{I(T)+1}}\geq\frac{T}{\left(\frac{r_{T+1}}{r_{1}}\right)^{\frac{1}{T}}\left(1+\log\frac{r_{T+1}}{r_{1}}\right)}=\frac{T}{\left(\frac{r_{T+1}}{r}\right)^{\frac{1}{T}}\left(1+\log\frac{r_{T+1}}{r}\right)}\geq\frac{T}{O\left(\left(\frac{r+d_{1}}{r}\right)^{\frac{1}{T}}\left(1+\log\frac{r+d_{1}}{r}\right)\right)}
\]
where $I(T)\in\mathrm{argmax}_{t\in\left[T\right]}\sum_{s=1}^{t}\frac{r_{s}}{r_{t+1}}$
and the last inequality holds because $r_{T+1}=O(d_{1}+r)$ is uniformly
upper bounded. Then
\begin{align*}
F(\bar{x}_{I(T)})-F(x_{*})\leq & O\left(\left(1+\log\frac{r+d_{1}}{r}\right)\left(r+d_{1}\right)\frac{\left(\frac{r+d_{1}}{r}\right)^{\frac{1}{T}}}{T\gamma_{I(T)}}\right)\\
\leq & O\left(\left(1+\log\frac{r+d_{1}}{r}\right)\left(r+d_{1}\right)\frac{\left(\frac{r+d_{1}}{r}\right)^{\frac{1}{T}}}{T\gamma_{T}}\right).
\end{align*}
When $T\geq\Omega(\log\frac{r+d_{1}}{r})$, we have $\left(\frac{r+d_{1}}{r}\right)^{\frac{1}{T}}=O(1)$,
which implies
\[
F(\bar{x}_{I(T)})-F(x_{*})\leq O\left(\left(1+\log\frac{r+d_{1}}{r}\right)\left(r+d_{1}\right)\frac{1}{T\gamma_{T}}\right).
\]
By plugging in
\begin{align*}
M_{t}= & 2G\lor\sigma(tw_{t}/\log(4/\delta)){}^{\frac{1}{p}},\\
\gamma_{t}= & \frac{\alpha_{1}}{G\sqrt{tw_{t}}}\land\frac{\alpha_{2}}{M_{t}},\\
\alpha_{1}= & \frac{1}{\sqrt{32W}},\\
\alpha_{2}= & \frac{1}{8\left(\frac{16}{3}+8\sqrt{5W}+4W\right)\log\frac{4}{\delta}}\land\frac{1}{\sqrt{16\left(\frac{32}{3}+8\sqrt{5W}+80W\right)\log\frac{4}{\delta}}},
\end{align*}
we conclude the proof.
\end{proof}

\section{Additional Theoretical Analysis When $\mu>0$\label{sec:app-str}}

In this section, we aim to prove Theorems \ref{thm:str-prob} and
\ref{thm:str-exp}.

\subsection{High-Probability Analysis When $\mu>0$\label{subsec:app-str-prob}}

To start with, we introduce a basic inequality in Lemma \ref{lem:basic-str-prob}.
\begin{lem}
\label{lem:basic-str-prob}When $\mu>0$, under the choices of $M_{t}=2G\lor Mt^{\frac{1}{p}}$
and $\eta_{t}=\frac{4}{\mu(t+1)}$, for any $\tau\in\left[T\right]$,
we have
\begin{align*}
\frac{\mu(\tau+1)\tau}{8}d_{\tau+1}^{2}+\sum_{t=1}^{\tau}t\Delta_{t}\leq & \mathbb{D}_{\tau}\left(\sum_{t=1}^{\tau}t\left\langle \xi_{t}^{u},\frac{x_{*}-x_{t}}{\mathbb{D}_{t}}\right\rangle +\frac{8\left(\left\Vert \xi_{t}^{u}\right\Vert ^{2}-\E_{t}\left[\left\Vert \xi_{t}^{u}\right\Vert ^{2}\right]\right)}{\mu C}\right)\\
 & +320\mathbb{D}_{\tau}\left(M\sqrt{(\sigma/M)^{p}+(\sigma/M)^{2p}}T^{\frac{1}{p}}+G\sqrt{1+(\sigma/G)^{p}}\sqrt{T}\right)
\end{align*}
where
\begin{align*}
\mathbb{D}_{t} & \coloneqq C\lor\max_{s\in\left[t\right]}\sqrt{s(s-1)}d_{s};\\
C & \coloneqq\frac{G\sqrt{1+(\sigma/G)^{p}}}{\mu}\sqrt{T}\lor\frac{M\sqrt{(\sigma/M)^{p}+(\sigma/M)^{2p}}}{\mu}T^{\frac{1}{p}}\lor\frac{M_{T}}{\mu}.
\end{align*}
\end{lem}
\begin{proof}
We first invoke Lemma \ref{lem:basic} to get
\begin{align}
\Delta_{t}+\frac{\eta_{t}^{-1}}{2}d_{t+1}^{2}-\frac{\eta_{t}^{-1}-\mu}{2}d_{t}^{2}\leq & \langle\xi_{t},x_{*}-x_{t}\rangle+\eta_{t}\left(2\left\Vert \xi_{t}^{u}\right\Vert ^{2}+2\left\Vert \xi_{t}^{b}\right\Vert ^{2}+G^{2}\right)\nonumber \\
= & \langle\xi_{t}^{b},x_{*}-x_{t}\rangle+\langle\xi_{t}^{u},x_{*}-x_{t}\rangle+\eta_{t}\left(2\left\Vert \xi_{t}^{u}\right\Vert ^{2}+2\left\Vert \xi_{t}^{b}\right\Vert ^{2}+G^{2}\right)\nonumber \\
\overset{(a)}{\leq} & \frac{\left\Vert \xi_{t}^{b}\right\Vert ^{2}}{\mu}+\frac{\mu d_{t}^{2}}{4}+\langle\xi_{t}^{u},x_{*}-x_{t}\rangle+\eta_{t}\left(2\left\Vert \xi_{t}^{u}\right\Vert ^{2}+2\left\Vert \xi_{t}^{b}\right\Vert ^{2}+G^{2}\right)\nonumber \\
\Rightarrow\Delta_{t}+\frac{\eta_{t}^{-1}}{2}d_{t+1}^{2}-\frac{\eta_{t}^{-1}-\mu/2}{2}d_{t}^{2}\leq & \langle\xi_{t}^{u},x_{*}-x_{t}\rangle+\frac{\left\Vert \xi_{t}^{b}\right\Vert ^{2}}{\mu}+\eta_{t}\left(2\left\Vert \xi_{t}^{u}\right\Vert ^{2}+2\left\Vert \xi_{t}^{b}\right\Vert ^{2}+G^{2}\right)\label{eq:str-1}
\end{align}
where $(a)$ is by $\langle\xi_{t}^{b},x_{*}-x_{t}\rangle\leq\left\Vert \xi_{t}^{b}\right\Vert ^{2}/\mu+\mu\left\Vert x_{t}-x_{*}\right\Vert ^{2}/4=\left\Vert \xi_{t}^{b}\right\Vert ^{2}/\mu+\mu d_{t}^{2}/4$.
Now, plugging $\eta_{t}=\frac{4}{\mu(t+1)}$ into (\ref{eq:str-1})
and multiplying both sides by $t/\mathbb{D}_{t}$ to obtain
\begin{align}
\frac{t\Delta_{t}}{\mathbb{D}_{t}}+\frac{\mu(t+1)t}{8\mathbb{D}_{t}}d_{t+1}^{2}-\frac{\mu t(t-1)}{8\mathbb{D}_{t}}d_{t}^{2}\leq & t\left\langle \xi_{t}^{u},\frac{x_{*}-x_{t}}{\mathbb{D}_{t}}\right\rangle +\frac{8\left(\left\Vert \xi_{t}^{u}\right\Vert ^{2}+\left\Vert \xi_{t}^{b}\right\Vert ^{2}\right)+t\left\Vert \xi_{t}^{b}\right\Vert ^{2}+4G^{2}}{\mu\mathbb{D}_{t}}\nonumber \\
\leq & t\left\langle \xi_{t}^{u},\frac{x_{*}-x_{t}}{\mathbb{D}_{t}}\right\rangle +\frac{8\left(\left\Vert \xi_{t}^{u}\right\Vert ^{2}+\left\Vert \xi_{t}^{b}\right\Vert ^{2}\right)+t\left\Vert \xi_{t}^{b}\right\Vert ^{2}+4G^{2}}{\mu C}\nonumber \\
= & t\left\langle \xi_{t}^{u},\frac{x_{*}-x_{t}}{\mathbb{D}_{t}}\right\rangle +\frac{8\left(\left\Vert \xi_{t}^{u}\right\Vert ^{2}-\E_{t}\left[\left\Vert \xi_{t}^{u}\right\Vert ^{2}\right]\right)}{\mu C}\nonumber \\
 & +\frac{8\left(\E_{t}\left[\left\Vert \xi_{t}^{u}\right\Vert ^{2}\right]+\left\Vert \xi_{t}^{b}\right\Vert ^{2}\right)+t\left\Vert \xi_{t}^{b}\right\Vert ^{2}+4G^{2}}{\mu C}.\label{eq:str-2}
\end{align}

Next, we bound $\E_{t}[\|\xi_{t}^{u}\|^{2}]\le10\sigma^{p}M_{t}^{2-p}$,
$\|\xi_{t}^{b}\|^{2}\leq10\sigma^{p}M_{t}^{2-p}$ and $\|\xi_{t}^{b}\|\le2\sigma^{p}M_{t}^{1-p}$
by using Lemma \ref{lem:err-bound} to get
\begin{align}
 & \frac{8\left(\E_{t}\left[\left\Vert \xi_{t}^{u}\right\Vert ^{2}\right]+\left\Vert \xi_{t}^{b}\right\Vert ^{2}\right)+t\left\Vert \xi_{t}^{b}\right\Vert ^{2}+4G^{2}}{\mu C}\nonumber \\
\leq & \frac{160\sigma^{p}M_{t}^{2-p}+4t\sigma^{2p}M_{t}^{2-2p}+4G^{2}}{\mu C}\nonumber \\
= & \frac{160\sigma^{p}(2G\lor Mt^{\frac{1}{p}})^{2-p}+4t\sigma^{2p}(2G\lor Mt^{\frac{1}{p}})^{2-2p}+4G^{2}}{\mu C}\nonumber \\
\leq & \frac{\left(160\sigma^{p}M^{2-p}+4\sigma^{2p}M^{2-2p}\right)t^{\frac{2}{p}-1}+320\sigma^{p}G^{2-p}+4G^{2}}{\mu C}\nonumber \\
\leq & 160M\sqrt{(\sigma/M)^{p}+(\sigma/M)^{2p}}\cdot\frac{t^{\frac{2}{p}-1}}{T^{\frac{1}{p}}}+320G\sqrt{1+(\sigma/G)^{p}}\cdot\frac{1}{\sqrt{T}}\label{eq:str-3}
\end{align}

Combining (\ref{eq:str-2}) and (\ref{eq:str-3}) and summing up from
$t=1$ to $\tau$, we have
\begin{align*}
 & \frac{\mu(\tau+1)\tau}{8\mathbb{D}_{\tau}}d_{\tau+1}^{2}+\sum_{t=2}^{\tau}\left(\frac{1}{\mathbb{D}_{t-1}}-\frac{1}{\mathbb{D}_{t}}\right)\frac{\mu t(t-1)}{8}d_{t}^{2}+\sum_{t=1}^{\tau}\frac{t\Delta_{t}}{\mathbb{D}_{t}}\\
\leq & \sum_{t=1}^{\tau}t\left\langle \xi_{t}^{u},\frac{x_{*}-x_{t}}{\mathbb{D}_{t}}\right\rangle +\frac{8\left(\left\Vert \xi_{t}^{u}\right\Vert ^{2}-\E_{t}\left[\left\Vert \xi_{t}^{u}\right\Vert ^{2}\right]\right)}{\mu C}\\
 & +\sum_{t=1}^{\tau}160M\sqrt{(\sigma/M)^{p}+(\sigma/M)^{2p}}\cdot\frac{t^{\frac{2}{p}-1}}{T^{\frac{1}{p}}}+320G\sqrt{1+(\sigma/G)^{p}}\cdot\frac{1}{\sqrt{T}}\\
\leq & \sum_{t=1}^{\tau}t\left\langle \xi_{t}^{u},\frac{x_{*}-x_{t}}{\mathbb{D}_{t}}\right\rangle +\frac{8\left(\left\Vert \xi_{t}^{u}\right\Vert ^{2}-\E_{t}\left[\left\Vert \xi_{t}^{u}\right\Vert ^{2}\right]\right)}{\mu C}\\
 & +160M\sqrt{(\sigma/M)^{p}+(\sigma/M)^{2p}}\frac{\frac{p}{2}\cdot((\tau+1)^{\frac{2}{p}}-1)}{T^{\frac{1}{p}}}+320G\sqrt{1+(\sigma/G)^{p}}\frac{\tau}{\sqrt{T}}\\
\leq & \sum_{t=1}^{\tau}t\left\langle \xi_{t}^{u},\frac{x_{*}-x_{t}}{\mathbb{D}_{t}}\right\rangle +\frac{8\left(\left\Vert \xi_{t}^{u}\right\Vert ^{2}-\E_{t}\left[\left\Vert \xi_{t}^{u}\right\Vert ^{2}\right]\right)}{\mu C}\\
 & +320M\sqrt{(\sigma/M)^{p}+(\sigma/M)^{2p}}T^{\frac{1}{p}}+320G\sqrt{1+(\sigma/G)^{p}}\sqrt{T}.
\end{align*}
Finally, we use $\frac{1}{\mathbb{D}_{t-1}}-\frac{1}{\mathbb{D}_{t}}\geq0$
and $\mathbb{\mathbb{D}_{\tau}\geq}\mathbb{D}_{t},\forall t\leq\tau$
to finish the proof.
\end{proof}

The same as the case of $\mu=0$. Our goal is to find a high-probability
bound of $\sum_{t=1}^{\tau}t\left\langle \xi_{t}^{u},\frac{x_{*}-x_{t}}{\mathbb{D}_{t}}\right\rangle $
and $\sum_{t=1}^{\tau}\frac{8}{\mu C}(\|\xi_{t}^{u}\|^{2}-\E_{t}[\|\xi_{t}^{u}\|^{2}])$.
Note that both of them are martingale difference sequences, hence,
we can use Freedman's inequality again. The formal results are presented
in the following Lemmas \ref{lem:inner-str} and \ref{lem:xi-u-str}.
\begin{lem}
\label{lem:inner-str}When $\mu>0$, under the choice of $M_{t}=2G\lor Mt^{\frac{1}{p}}$,
we have with probability at least $1-\frac{\delta}{2}$, for any $\tau\in\left[T\right]$,
\[
\sum_{t=1}^{\tau}t\left\langle \xi_{t}^{u},\frac{x_{*}-x_{t}}{\mathbb{D}_{t}}\right\rangle \leq3G\log\frac{4}{\delta}+13\left(M\log\frac{4}{\delta}+\sqrt{\left(\sigma^{p}G^{2-p}+\sigma^{p}M^{2-p}\right)\log\frac{4}{\delta}}\right)T^{\frac{1}{p}}.
\]
\end{lem}
\begin{proof}
We first note that $Z_{t}\coloneqq t\left\langle \xi_{t}^{u},\frac{x_{*}-x_{t}}{\mathbb{D}_{t}}\right\rangle \in\F_{t}$
is a martingale difference sequence. Next, observe that
\[
\frac{td_{t}}{\mathbb{D}_{t}}\leq\frac{td_{t}}{\sqrt{t(t-1)}d_{t}\lor\frac{G\sqrt{1+(\sigma/G)^{p}}}{\mu}\sqrt{T}}\leq\begin{cases}
\sqrt{\frac{t}{t-1}}\leq2 & t\geq2\\
\frac{d_{1}}{\frac{G}{\mu}}\leq2 & t=1
\end{cases}.
\]
Hence, we know
\[
\left|Z_{t}\right|\leq\left\Vert \xi_{t}^{u}\right\Vert \frac{td_{t}}{\mathbb{D}_{t}}\leq2M_{t}\leq2M_{T}=4G\lor2MT^{\frac{1}{p}}.
\]
Besides, we know
\begin{align*}
\sum_{t=1}^{\tau}\E_{t}\left[Z_{t}^{2}\right]\leq & \sum_{t=1}^{\tau}\E_{t}\left[\left\Vert \xi_{t}^{u}\right\Vert ^{2}\frac{t^{2}d_{t}^{2}}{\mathbb{D}_{t}^{2}}\right]\leq\sum_{t=1}^{\tau}10\sigma^{p}M_{t}^{2-p}\cdot4\\
= & \sum_{t=1}^{\tau}40\sigma^{p}M_{t}^{2-p}=\sum_{t=1}^{\tau}40\sigma^{p}\cdot\left(2G\lor Mt^{\frac{1}{p}}\right)^{2-p}\\
\leq & \sum_{t=1}^{\tau}80\sigma^{p}G^{2-p}+40\sigma^{p}M^{2-p}t^{\frac{2}{p}-1}\\
\leq & 80\sigma^{p}G^{2-p}\tau+40\sigma^{p}M^{2-p}\cdot\frac{p}{2}\left((\tau+1)^{\frac{2}{p}}-1\right)\\
\leq & 80\sigma^{p}G^{2-p}T+80\sigma^{p}M^{2-p}T^{\frac{2}{p}}\\
\leq & 80\left(\sigma^{p}G^{2-p}+\sigma^{p}M^{2-p}\right)T^{\frac{2}{p}}.
\end{align*}
Let $R=4G\lor2MT^{\frac{1}{p}}$, $F=80\left(\sigma^{p}G^{2-p}+\sigma^{p}M^{2-p}\right)T^{\frac{2}{p}}$.
By Freedman's inequality (Corollary \ref{cor:ez-any-freedman-1}),
with probability at least $1-\frac{\delta}{2}$, we have
\begin{align*}
\sum_{t=1}^{\tau}Z_{t}\leq & \frac{2R}{3}\log\frac{4}{\delta}+\sqrt{2F\log\frac{4}{\delta}}\\
= & \frac{8G\lor4MT^{\frac{1}{p}}}{3}\log\frac{4}{\delta}+\sqrt{160\left(\sigma^{p}G^{2-p}+\sigma^{p}M^{2-p}\right)\log\frac{4}{\delta}}T^{\frac{1}{p}}\\
\leq & 3G\log\frac{4}{\delta}+13\left(M\log\frac{4}{\delta}+\sqrt{\left(\sigma^{p}G^{2-p}+\sigma^{p}M^{2-p}\right)\log\frac{4}{\delta}}\right)T^{\frac{1}{p}}.
\end{align*}
\end{proof}

\begin{lem}
\label{lem:xi-u-str}When $\mu>0$, under the choice of $M_{t}=2G\lor Mt^{\frac{1}{p}}$,
we have with probability at least $1-\frac{\delta}{2}$, for any $\tau\in\left[T\right]$,
\[
\sum_{t=1}^{\tau}\frac{8\left(\left\Vert \xi_{t}^{u}\right\Vert ^{2}-\E_{t}\left[\left\Vert \xi_{t}^{u}\right\Vert ^{2}\right]\right)}{\mu C}\leq88G\log\frac{4}{\delta}+104\left(M\log\frac{4}{\delta}+\sqrt{\left(\sigma^{p}G^{2-p}+\sigma^{p}M^{2-p}\right)\log\frac{4}{\delta}}\right)T^{\frac{1}{p}}.
\]
\end{lem}
\begin{proof}
We first note that $\frac{\left\Vert \xi_{t}^{u}\right\Vert ^{2}-\E_{t}\left[\left\Vert \xi_{t}^{u}\right\Vert ^{2}\right]}{\mu C}\in\F_{t}$
is a martingale difference sequence. Next, observe that
\[
\frac{\left|\left\Vert \xi_{t}^{u}\right\Vert ^{2}-\E_{t}\left[\left\Vert \xi_{t}^{u}\right\Vert ^{2}\right]\right|}{\mu C}\leq\frac{\left\Vert \xi_{t}^{u}\right\Vert ^{2}+\E_{t}\left[\left\Vert \xi_{t}^{u}\right\Vert ^{2}\right]}{\mu C}\leq\frac{8M_{t}^{2}}{\mu C}\leq\frac{8M_{t}^{2}}{M_{T}}\leq8M_{T}=16G\lor8MT^{\frac{1}{p}}.
\]
Besides, we know
\begin{align*}
\sum_{t=1}^{\tau}\E_{t}\left[\frac{\left(\left\Vert \xi_{t}^{u}\right\Vert ^{2}-\E_{t}\left[\left\Vert \xi_{t}^{u}\right\Vert ^{2}\right]\right)^{2}}{\mu^{2}C^{2}}\right]\leq & \sum_{t=1}^{\tau}\frac{\E_{t}\left[\left\Vert \xi_{t}^{u}\right\Vert ^{4}\right]}{\mu^{2}C^{2}}\leq\sum_{t=1}^{\tau}\frac{4M_{t}^{2}\cdot10\sigma^{p}M_{t}^{2-p}}{M_{T}^{2}}\\
= & \sum_{t=1}^{\tau}40\sigma^{p}M_{t}^{2-p}=\sum_{t=1}^{\tau}40\sigma^{p}\cdot\left(2G\lor Mt^{\frac{1}{p}}\right)^{2-p}\\
\leq & 80\left(\sigma^{p}G^{2-p}+\sigma^{p}M^{2-p}\right)T^{\frac{2}{p}}.
\end{align*}
Let $R=16G\lor8MT^{\frac{1}{p}}$, $F=80\left(\sigma^{p}G^{2-p}+\sigma^{p}M^{2-p}\right)T^{\frac{2}{p}}$.
By Freedman's inequality (Corollary \ref{cor:ez-any-freedman-1}),
with probability at least $1-\frac{\delta}{2}$, we have
\begin{align*}
\sum_{t=1}^{\tau}\frac{\left\Vert \xi_{t}^{u}\right\Vert ^{2}-\E_{t}\left[\left\Vert \xi_{t}^{u}\right\Vert ^{2}\right]}{\mu C}\leq & \frac{2R}{3}\log\frac{4}{\delta}+\sqrt{2F\log\frac{4}{\delta}}\\
= & \frac{32G\lor16MT^{\frac{1}{p}}}{3}\log\frac{4}{\delta}+\sqrt{160\left(\sigma^{p}G^{2-p}+\sigma^{p}M^{2-p}\right)\log\frac{4}{\delta}}T^{\frac{1}{p}}\\
= & 11G\log\frac{4}{\delta}+13\left(M\log\frac{4}{\delta}+\sqrt{\left(\sigma^{p}G^{2-p}+\sigma^{p}M^{2-p}\right)\log\frac{4}{\delta}}\right)T^{\frac{1}{p}}\\
\Rightarrow\sum_{t=1}^{\tau}\frac{8\left(\left\Vert \xi_{t}^{u}\right\Vert ^{2}-\E_{t}\left[\left\Vert \xi_{t}^{u}\right\Vert ^{2}\right]\right)}{\mu C}\leq & 88G\log\frac{4}{\delta}+104\left(M\log\frac{4}{\delta}+\sqrt{\left(\sigma^{p}G^{2-p}+\sigma^{p}M^{2-p}\right)\log\frac{4}{\delta}}\right)T^{\frac{1}{p}}.
\end{align*}
\end{proof}

With the above lemmas, we are able to prove Theorem \ref{thm:str-prob}.

\begin{proof}[Proof of Theorem \ref{thm:str-prob}]
We first define a constant $K$ as follows
\begin{align}
K\coloneqq & \frac{\mu C^{2}}{8}+\frac{16}{\mu}\left[91G\log\frac{4}{\delta}+117\left(M\log\frac{4}{\delta}+\sqrt{\left(\sigma^{p}G^{2-p}+\sigma^{p}M^{2-p}\right)\log\frac{4}{\delta}}\right)T^{\frac{1}{p}}\right]^{2}\nonumber \\
 & +\frac{16}{\mu}\left[320\left(M\sqrt{(\sigma/M)^{p}+(\sigma/M)^{2p}}T^{\frac{1}{p}}+G\sqrt{1+(\sigma/G)^{p}}\sqrt{T}\right)\right]^{2}\label{eq:str-prob-def-k}\\
= & O\left(\log^{2}(1/\delta)\left(\frac{G^{2}+\sigma^{2}}{\mu}T+\frac{M^{2}+\sigma^{2p}M^{2-2p}+\sigma^{p}G^{2-p}}{\mu}T^{\frac{2}{p}}\right)\right)\nonumber 
\end{align}
where $C=\frac{G\sqrt{1+(\sigma/G)^{p}}}{\mu}\sqrt{T}\lor\frac{M\sqrt{(\sigma/M)^{p}+(\sigma/M)^{2p}}}{\mu}T^{\frac{1}{p}}\lor\frac{M_{T}}{\mu}$
is defined in Lemma \ref{lem:basic-str-prob}.

We sart with Lemma \ref{lem:basic-str-prob} to get for any $\tau\in\left[T\right]$,
there is
\begin{align}
\frac{\mu(\tau+1)\tau}{8}d_{\tau+1}^{2}+\sum_{t=1}^{\tau}t\Delta_{t}\leq & \mathbb{D}_{\tau}\left(\sum_{t=1}^{\tau}t\left\langle \xi_{t}^{u},\frac{x_{*}-x_{t}}{\mathbb{D}_{t}}\right\rangle +\frac{8\left(\left\Vert \xi_{t}^{u}\right\Vert ^{2}-\E_{t}\left[\left\Vert \xi_{t}^{u}\right\Vert ^{2}\right]\right)}{\mu C}\right)\nonumber \\
 & +320\mathbb{D}_{\tau}\left(M\sqrt{(\sigma/M)^{p}+(\sigma/M)^{2p}}T^{\frac{1}{p}}+G\sqrt{1+(\sigma/G)^{p}}\sqrt{T}\right).\label{eq:str-prob}
\end{align}
\begin{itemize}
\item Bounding the term $\sum_{t=1}^{\tau}t\left\langle \xi_{t}^{u},\frac{x_{*}-x_{t}}{\mathbb{D}_{t}}\right\rangle $:
By Lemma \ref{lem:inner-str}, we have with probability at least $1-\frac{\delta}{2}$,
for any $\tau\in\left[T\right]$:
\begin{equation}
\sum_{t=1}^{\tau}t\left\langle \xi_{t}^{u},\frac{x_{*}-x_{t}}{\mathbb{D}_{t}}\right\rangle \leq3G\log\frac{4}{\delta}+13\left(M\log\frac{4}{\delta}+\sqrt{\left(\sigma^{p}G^{2-p}+\sigma^{p}M^{2-p}\right)\log\frac{4}{\delta}}\right)T^{\frac{1}{p}}.\label{eq:str-prob-inner}
\end{equation}
\item Bounding the term $\sum_{t=1}^{\tau}\frac{8\left(\left\Vert \xi_{t}^{u}\right\Vert ^{2}-\E_{t}\left[\left\Vert \xi_{t}^{u}\right\Vert ^{2}\right]\right)}{\mu C}$:
By Lemma \ref{lem:xi-u-str}, we have with probability at least $1-\frac{\delta}{2}$,
for any $\tau\in\left[T\right]$:
\begin{equation}
\sum_{t=1}^{\tau}\frac{8\left(\left\Vert \xi_{t}^{u}\right\Vert ^{2}-\E_{t}\left[\left\Vert \xi_{t}^{u}\right\Vert ^{2}\right]\right)}{\mu C}\leq88G\log\frac{4}{\delta}+104\left(M\log\frac{4}{\delta}+\sqrt{\left(\sigma^{p}G^{2-p}+\sigma^{p}M^{2-p}\right)\log\frac{4}{\delta}}\right)T^{\frac{1}{p}}.\label{eq:str-prob-xi-u-2}
\end{equation}
\end{itemize}
Combining (\ref{eq:str-prob}), (\ref{eq:str-prob-inner}), (\ref{eq:str-prob-xi-u-2}),
we have with probability at least $1-\delta$, for any $\tau\in\left[T\right]$:
\begin{align*}
\frac{\mu(\tau+1)\tau}{8}d_{\tau+1}^{2}+\sum_{t=1}^{\tau}t\Delta_{t}\leq & \mathbb{D}_{\tau}\left[91G\log\frac{4}{\delta}+117\left(M\log\frac{4}{\delta}+\sqrt{\left(\sigma^{p}G^{2-p}+\sigma^{p}M^{2-p}\right)\log\frac{4}{\delta}}\right)T^{\frac{1}{p}}\right]\\
 & +320\mathbb{D}_{\tau}\left(M\sqrt{(\sigma/M)^{p}+(\sigma/M)^{2p}}T^{\frac{1}{p}}+G\sqrt{1+(\sigma/G)^{p}}\sqrt{T}\right)\\
\leq & \frac{\mu\mathbb{D}_{\tau}^{2}}{16}+\frac{8}{\mu}\left[91G\log\frac{4}{\delta}+117\left(M\log\frac{4}{\delta}+\sqrt{\left(\sigma^{p}G^{2-p}+\sigma^{p}M^{2-p}\right)\log\frac{4}{\delta}}\right)T^{\frac{1}{p}}\right]^{2}\\
 & +\frac{8}{\mu}\left[320\left(M\sqrt{(\sigma/M)^{p}+(\sigma/M)^{2p}}T^{\frac{1}{p}}+G\sqrt{1+(\sigma/G)^{p}}\sqrt{T}\right)\right]^{2}\\
\overset{(a)}{\leq} & \frac{\mu\left(\max_{s\in\left[\tau\right]}s(s-1)d_{s}^{2}+C^{2}\right)}{16}\\
 & +\frac{8}{\mu}\left[91G\log\frac{4}{\delta}+117\left(M\log\frac{4}{\delta}+\sqrt{\left(\sigma^{p}G^{2-p}+\sigma^{p}M^{2-p}\right)\log\frac{4}{\delta}}\right)T^{\frac{1}{p}}\right]^{2}\\
 & +\frac{8}{\mu}\left[320\left(M\sqrt{(\sigma/M)^{p}+(\sigma/M)^{2p}}T^{\frac{1}{p}}+G\sqrt{1+(\sigma/G)^{p}}\sqrt{T}\right)\right]^{2}\\
\overset{(b)}{=} & \frac{\max_{s\in\left[\tau\right]}\frac{\mu s(s-1)}{8}d_{s}^{2}}{2}+\frac{K}{2}
\end{align*}
where $(a)$ is by $\mathbb{D}_{\tau}^{2}=(C\lor\max_{s\in\left[\tau\right]}\sqrt{s(s-1)}d_{s})^{2}\leq\max_{s\in\left[\tau\right]}s(s-1)d_{s}^{2}+C^{2}$;
$(b)$ is due to the definition of $K$ (see (\ref{eq:str-prob-def-k})).
Hence, by using $\Delta_{t}\geq0$, we have for any $\tau\in\left[T\right]$,
\[
\frac{\mu(\tau+1)\tau}{8}d_{\tau+1}^{2}\leq\frac{\max_{s\in\left[\tau\right]}\frac{\mu s(s-1)}{8}d_{s}^{2}}{2}+\frac{K}{2},
\]
which implies $\frac{\mu t(t-1)}{8}d_{t}^{2}\leq K$ for any $t\in\left[T+1\right]$
by simple induction.

Finally, we consider time $T$ to get with probability at least $1-\delta$
\[
\frac{\mu(T+1)T}{8}d_{T+1}^{2}+\sum_{t=1}^{\tau}t\Delta_{t}\leq K.
\]
Note that $F(\bar{x}_{T})-F(x_{*})\leq\frac{2}{T(T+1)}\sum_{t=1}^{T}t\Delta_{t}$
by the convexity of $F$ where $\bar{x}_{T}=\frac{2}{T(T+1)}\sum_{t=1}^{T}tx_{t}$,
we conclude that
\[
\frac{\mu(T+1)T}{8}d_{T+1}^{2}+F(\bar{x}_{T})-F(x_{*})\leq K.
\]
We get the desired result by plugging $K$.
\end{proof}

\subsection{In-Expectaion Analysis When $\mu>0$\label{subsec:app-str-exp}}

The proof of Theorem \ref{thm:str-exp} is inispired by \cite{zhang2020adaptive}.

\begin{proof}[Proof of Theorem \ref{thm:str-exp}]
We first invoke Lemma \ref{lem:basic} to get
\begin{align}
\Delta_{t}+\frac{\eta_{t}^{-1}}{2}d_{t+1}^{2}-\frac{\eta_{t}^{-1}-\mu}{2}d_{t}^{2}\leq & \langle\xi_{t},x_{*}-x_{t}\rangle+\eta_{t}\left(2\left\Vert \xi_{t}^{u}\right\Vert ^{2}+2\left\Vert \xi_{t}^{b}\right\Vert ^{2}+G^{2}\right)\nonumber \\
= & \langle\xi_{t}^{b},x_{*}-x_{t}\rangle+\langle\xi_{t}^{u},x_{*}-x_{t}\rangle+\eta_{t}\left(2\left\Vert \xi_{t}^{u}\right\Vert ^{2}+2\left\Vert \xi_{t}^{b}\right\Vert ^{2}+G^{2}\right)\nonumber \\
\overset{(a)}{\leq} & \frac{\left\Vert \xi_{t}^{b}\right\Vert ^{2}}{\mu}+\frac{\mu d_{t}^{2}}{4}+\langle\xi_{t}^{u},x_{*}-x_{t}\rangle+\eta_{t}\left(2\left\Vert \xi_{t}^{u}\right\Vert ^{2}+2\left\Vert \xi_{t}^{b}\right\Vert ^{2}+G^{2}\right)\nonumber \\
\Rightarrow\Delta_{t}+\frac{\eta_{t}^{-1}}{2}d_{t+1}^{2}-\frac{\eta_{t}^{-1}-\mu/2}{2}d_{t}^{2}\leq & \langle\xi_{t}^{u},x_{*}-x_{t}\rangle+\frac{\left\Vert \xi_{t}^{b}\right\Vert ^{2}}{\mu}+\eta_{t}\left(2\left\Vert \xi_{t}^{u}\right\Vert ^{2}+2\left\Vert \xi_{t}^{b}\right\Vert ^{2}+G^{2}\right)\label{eq:str-exp-1}
\end{align}
where $(a)$ is by $\langle\xi_{t}^{b},x_{*}-x_{t}\rangle\leq\left\Vert \xi_{t}^{b}\right\Vert ^{2}/\mu+\mu\left\Vert x_{t}-x_{*}\right\Vert ^{2}/4=\left\Vert \xi_{t}^{b}\right\Vert ^{2}/\mu+\mu d_{t}^{2}/4$.
Now, plugging $\eta_{t}=\frac{4}{\mu(t+1)}$ into (\ref{eq:str-exp-1})
and multiplying both sides by $t$ to obtain
\begin{align}
t\Delta_{t}+\frac{\mu(t+1)t}{8}d_{t+1}^{2}-\frac{\mu t(t-1)}{8}d_{t}^{2}\leq & t\langle\xi_{t}^{u},x_{*}-x_{t}\rangle+\frac{8\left(\left\Vert \xi_{t}^{u}\right\Vert ^{2}+\left\Vert \xi_{t}^{b}\right\Vert ^{2}\right)+t\left\Vert \xi_{t}^{b}\right\Vert ^{2}+4G^{2}}{\mu}\nonumber \\
\Rightarrow\E\left[t\Delta_{t}\right]+\frac{\mu(t+1)t}{8}\E\left[d_{t+1}^{2}\right]-\frac{\mu t(t-1)}{8}\E\left[d_{t}^{2}\right]\leq & \frac{\E\left[8\left(\left\Vert \xi_{t}^{u}\right\Vert ^{2}+\left\Vert \xi_{t}^{b}\right\Vert ^{2}\right)+t\left\Vert \xi_{t}^{b}\right\Vert ^{2}+4G^{2}\right]}{\mu}\nonumber \\
= & \frac{\E\left[8\left(\E_{t}\left[\left\Vert \xi_{t}^{u}\right\Vert ^{2}\right]+\left\Vert \xi_{t}^{b}\right\Vert ^{2}\right)+t\left\Vert \xi_{t}^{b}\right\Vert ^{2}+4G^{2}\right]}{\mu}.\label{eq:str-exp-2}
\end{align}

Next, we bound $\E_{t}[\|\xi_{t}^{u}\|^{2}]\le10\sigma^{p}M_{t}^{2-p}$,
$\|\xi_{t}^{b}\|^{2}\leq10\sigma^{p}M_{t}^{2-p}$ and $\|\xi_{t}^{b}\|\le2\sigma^{p}M_{t}^{1-p}$
by using Lemma \ref{lem:err-bound} to get
\begin{align}
 & \frac{8\left(\E_{t}\left[\left\Vert \xi_{t}^{u}\right\Vert ^{2}\right]+\left\Vert \xi_{t}^{b}\right\Vert ^{2}\right)+t\left\Vert \xi_{t}^{b}\right\Vert ^{2}+4G^{2}}{\mu}\nonumber \\
\leq & \frac{160\sigma^{p}M_{t}^{2-p}+4t\sigma^{2p}M_{t}^{2-2p}+4G^{2}}{\mu}\nonumber \\
= & \frac{160\sigma^{p}(2G\lor Mt^{\frac{1}{p}})^{2-p}+4t\sigma^{2p}(2G\lor Mt^{\frac{1}{p}})^{2-2p}+4G^{2}}{\mu}\nonumber \\
\leq & \frac{160\sigma^{p}M^{2-p}+4\sigma^{2p}M^{2-2p}}{\mu}t^{\frac{2}{p}-1}+\frac{320\sigma^{p}G^{2-p}+4G^{2}}{\mu}.\label{eq:str-exp-3}
\end{align}
Combining (\ref{eq:str-exp-2}) and (\ref{eq:str-exp-3}) and summing
up from $t=1$ to $\tau$, we have
\begin{align*}
 & \frac{\mu(\tau+1)\tau}{8}\E\left[d_{\tau+1}^{2}\right]+\sum_{t=1}^{\tau}\E\left[t\Delta_{t}\right]\\
\leq & \frac{160\sigma^{p}M^{2-p}+4\sigma^{2p}M^{2-2p}}{\mu}t^{\frac{2}{p}-1}+\frac{320\sigma^{p}G^{2-p}+4G^{2}}{\mu}\\
\leq & \frac{160\sigma^{p}M^{2-p}+4\sigma^{2p}M^{2-2p}}{\mu}\cdot\frac{p}{2}((\tau+1)^{\frac{2}{p}}-1)+\frac{320\sigma^{p}G^{2-p}+4G^{2}}{\mu}\tau\\
\leq & \frac{320\sigma^{p}M^{2-p}+8\sigma^{2p}M^{2-2p}}{\mu}\tau^{\frac{2}{p}}+\frac{320\sigma^{p}G^{2-p}+4G^{2}}{\mu}\tau\\
\leq & \frac{320\sigma^{p}M^{2-p}+8\sigma^{2p}M^{2-2p}}{\mu}\tau^{\frac{2}{p}}+\frac{320\sigma^{2}+324G^{2}}{\mu}\tau.
\end{align*}

Finally, choosing $\tau=T$ and using $\sum_{t=1}^{T}t\E\left[\Delta_{t}\right]\geq\frac{T(T+1)}{2}\E\left[F(\bar{x}_{T})-F(x_{*})\right]$
by the convexity of $F$ and $\bar{x}_{T}=\frac{2}{T(T+1)}\sum_{t=1}^{T}x_{t}$,
we conclude
\begin{align*}
\E\left[F(\bar{x}_{T})-F(x_{*})\right] & \leq O\left(\frac{G^{2}+\sigma^{2}}{\mu T}+\frac{\sigma^{p}M^{2-p}+\sigma^{2p}M^{2-2p}}{\mu T^{\frac{2(p-1)}{p}}}\right);\\
\E\left[\left\Vert x_{T+1}-x_{*}\right\Vert ^{2}\right] & \leq O\left(\frac{G^{2}+\sigma^{2}}{\mu^{2}T}+\frac{\sigma^{p}M^{2-p}+\sigma^{2p}M^{2-2p}}{\mu^{2}T^{\frac{2(p-1)}{p}}}\right).
\end{align*}
\end{proof}

\end{document}